\documentclass[11pt]{amsart}
\usepackage{amsmath,amssymb,txfonts}
\usepackage{amssymb}
\usepackage{amsmath}
\usepackage{mathrsfs}
\usepackage{amsmath,amssymb}
\usepackage{color}

\newtheorem{theorem}{Theorem}[section]
\newtheorem{lemma}[theorem]{Lemma}
\newtheorem{proposition}[theorem]{Proposition}

\theoremstyle{definition}
\newtheorem{definition}[theorem]{Definition}

\theoremstyle{remark}
\newtheorem{remark}[theorem]{Remark}

\numberwithin{equation}{section}

\begin{document}
\title[Carleson measures and the characterization of BMO type spaces]
{Regularity of fractional heat semigroup  associated with  Schr\"odinger operators}

\author[P. Li]{Pengtao\ Li} 
\address{School of Mathematics and Statistics, Qingdao University, Qingdao, Shandong 266071, China}
\email{ptli@qdu.edu.cn}

\author[Z. Wang]{Ziyong\ Wang}
\address{School of Mathematics and Statistics, Qingdao University, Qingdao, Shandong 266071, China}
\email{1206856960@qq.com}

\author[T. Qian]{Tao\ Qian}
\address{Department of Mathematics, FST, University of Macau, PO Box 3001, Macau, China}
\email{fsttq@umac.mo}

\author[C. Zhang]{Chao\ Zhang}
\address{School of Statistics and Mathematics, Zhejiang Gongshang University, Hangzhou, Zhejiang 310018, China}
\email{zaoyangzhangchao@163.com
}

\begin{abstract}
Let
$L=-\Delta+V$
 be a  Schr\"odinger operator, where
the potential $V$ belongs to the reverse H\"older class. By the subordinative formula, we introduce the fractional heat semigroup $\{e^{-t{L}^{\alpha}}\}_{t>0}, \alpha>0$, associated with ${L}$.  By the aid of the fundamental solution of the heat equation:
 $$\partial_{t}u+L  u=\partial_{t}u -\Delta u+V u=0,$$
 we estimate the gradient and the time-fractional derivatives of the fractional heat kernel $K^{{L}}_{\alpha,t}(\cdot, \cdot)$, respectively. This method is independent of the Fourier transform, and can be applied to the second order differential operators whose heat kernels satisfying Gaussian upper bounds. As an application,  we establish a Carleson measure characterization of the Campanato type space $BMO^{\gamma}_{{L}}(\mathbb{R}^{n})$ via $\{e^{-t{L}^{\alpha}}\}_{t>0}$.
\end{abstract}

\thanks{P. Li was supported by the National Natural Science Foundation of China of
China (Grant Nos.11871293,  12071272); Shandong Natural Science
Foundation of China (Grant No. ZR2017JL008). C. Zhang was supported by the National Natural Science Foundation of China(Grant Nos. 11971431, 11401525), the Zhejiang Provincial Natural Science Foundation of China(Grant No. LY18A010006) and the first Class Discipline of Zhejiang-A(Zhejiang Gongshang University-Statistics).}
\maketitle
\tableofcontents
 \pagenumbering{arabic}
\section{Introduction }

The aim of this paper is to investigate the fractional heat semigroup of Schr\"odinger operators
   $$L:=-\Delta +V(x),$$
where $-\Delta$ denotes the Laplace operator: $\Delta=\sum\limits^{n}_{i=1}{\partial^{2}}/{\partial x_{i}^{2}}$ and $V$ is a nonnegative potential belonging to the reverse H\"older class $B_{q}$:
\begin{definition}\label{def-1.1} A nonnegative locally $L^{q}$ integrable function $V$ on $\mathbb R^{n}$
is said to belong to $B_{q}, 1<q<\infty,$ if there exists $C>0$
such that the reverse H\"{o}lder inequality
\begin{equation}\label{eq-2.1}
\Big(\frac{1}{|B|}\int_{B}V^{q}(x)dx\Big)^{1/q}\leq  \frac{C}{|B|}\int_{B}V(x)dx
\end{equation}
holds for every ball $B$ in $\mathbb R^{n}.$
\end{definition}

 For the special case: $V=0$, $L=-\Delta$, the fractional heat semigroup can be defined via the Fourier transform:
\begin{equation}\label{eq-1.3}
\big({e^{-t(-\Delta)^{\alpha}}(f)}\big)^\wedge(\xi):=e^{-t|\xi|^{2\alpha}}\widehat{f}(\xi),\ \alpha\in (0,1].
\end{equation}
In the literature, the fractional heat semigroup $\{e^{-t(-\Delta)^{\alpha}}\}_{t>0}$ has been widely used in the study of
partial differential equations, harmonic analysis, potential theory and modern probability theory. For example, the semigroup $\{e^{-t(-\Delta)^{\alpha}}\}_{t>0}$ is
usually applied to construct  the linear part of solutions to fluid equations in
the mathematic physics, e.g. the generalized
Naiver-Stokes equation, the quasi-geostrophic equation, the generalized MHD equations. In the field of probability theory, the researchers use
$\{e^{-t(-\Delta)^{\alpha}}\}_{t>0}$ to describe some kind of Markov processes with jumps. For further information and the related applications of fractional heat  semigroups $\{e^{-t(-\Delta)^{\alpha}}\}_{t>0}$, we refer the reader to \cite{chang, Gri, jiangrenjin2, Zhai}.

Denote by $K_{\alpha, t}(\cdot, \cdot)$ the integral kernel of $e^{-t(-\Delta)^{\alpha}}$ , i.e.,
\begin{equation}\label{eq-1.7}
K_{\alpha, t}(x)=(2\pi)^{-n/2}\int_{\mathbb R^n}e^{ix\cdot \xi-t|\xi|^{2\alpha}}d\xi
\end{equation}
and denote by $K^{\theta}_{\alpha, t}(\cdot)$ the integral kernel of
$(-\Delta)^{\theta/2}e^{-t(-\Delta)^{\alpha}}$. In \cite{miao}, by an
invariant derivative technique and the Fourier analysis  method, Miao-Yuan-Zhang concluded that
 the kernels $K_{\alpha, t}$ and $K^{\theta}_{\alpha, t}$ satisfy  the following pointwise estimates, respectively (cf. \cite[Lemmas 2.1\  and \
 2.2]{miao}),
\begin{equation*}\label{eq-miao-result}
\left\{ \begin{aligned}
K_{\alpha, t}(x)&\lesssim \frac{t}{(t^{{1}/{2\alpha}}+|x|)^{n+2\alpha}}\quad \forall\ (x,t)\in\mathbb R_+^{n+1};\\
K_{\alpha, t}^{\theta}(x)&\lesssim
\frac{1}{(t^{1/2\alpha}+|x|)^{n+\theta}}\quad \forall\
(x,t)\in\mathbb R_+^{n+1}.
\end{aligned} \right.
\end{equation*}

Compared with $-\Delta$, for arbitrary Schr\"odinger operator $L$ with the non-negative potential $V$, the fractional heat
 semigroup $\{e^{-tL^{\alpha}}\}_{t>0}, \alpha\in(0,1)$, can not be defined via (\ref{eq-1.3}). Also, it is obvious that the methods in \cite{miao} are
 invalid for the estimation of the integral kernels of
 $\{e^{-t{L}^{\alpha}}\}_{t>0}$. In this paper, by the functional calculus, we observe that the integral kernel of the Poisson semigroup associated with $L$
can be defined as
\begin{equation}\label{eq-1.4-1}
P^{L}_{t}(x,y)=\int^{\infty}_{0}\frac{e^{-u}}{\sqrt{u}}K^{L}_{t^{2}/4u}(x,y)du=\int^{\infty}_{0}\frac{te^{-t^{2}/4s}}{2s^{3/2}}K^{L}_{s}(x,y)ds,
\end{equation}
where $K^{L}_{s}(\cdot,\cdot)$ denotes the integral kernel of $e^{-sL}$, i.e.,
$$e^{-sL}(f)(x):=\int_{\mathbb{R}^{n}}K^{L}_{s}(x,y)f(y)dy.$$
Recall that $K^{L}_{t}(\cdot,\cdot)$ is a positive, symmetric function on $\mathbb{R}^{n}\times\mathbb{R}^{n}$, and
satisfies $\int_{\mathbb{R}^{n}}K^{L}_{t}(x,y)dy\leq 1$. Generally, for $\alpha>0$, the subordinative formula (cf. \cite{Gri})
indicates that there exists a continuous function $\eta^{\alpha}_{t}(\cdot)$ on $(0,\infty)$ such that
 \begin{equation}\label{eq-sub-for-1}
K^{L}_{\alpha,t}(x,y)=\int^{\infty}_{0}\eta^{\alpha}_{t}(s)K^{L}_{s}(x,y)ds.
\end{equation}
The identity (\ref{eq-sub-for-1}) enables us to estimate $K^{L}_{\alpha, t}(\cdot,\cdot)$ via the heat kernel $K^{L}_{t}(\cdot,\cdot)$. Let
 $\rho(\cdot)$ be the auxiliary function defined by (\ref{eq-2.3}) below. In Propositions \ref{prop-3.1} and \ref{prop-3.2},  we can obtain the following pointwise estimates of $K^{L}_{\alpha,t}(\cdot,\cdot)$:
 for every $N>0$, there exists a constant $C_{N}$ such that
\begin{equation}\label{eq-1.6}
\Big|K^{{L}}_{\alpha,t}(x,y)\Big|\leq \frac{C_{N}t}{(t^{1/2\alpha}+|x-y|)^{n+2\alpha}}\Big(1+\frac{t^{1/2\alpha}}{\rho(x)}+\frac{t^{1/2\alpha}}{\rho(y)}\Big)^{-N},
\end{equation}
and for every $N>0$, $0<\delta'<\min\{1, 2-n/q\}$ and all $|h|\leq t^{1/\alpha}$, there exists a constant $C_{N}$ such that
\begin{eqnarray}\label{eq-1.8}
&&\Big|K^{{L}}_{\alpha,t}(x+h,y)-K^{{L}}_{\alpha,t}(x,y)\Big|\leq \frac{C_{N}t ({|h|}/{t^{1/2\alpha}})^{\delta'}}{(t^{1/2\alpha}+|x-y|)^{n+2\alpha}}\Big(1+\frac{t^{1/2\alpha}}{\rho(x)}+\frac{t^{1/2\alpha}}{\rho(y)}\Big)^{-N}.
\end{eqnarray}

Based on the estimates (\ref{eq-1.6}) and  (\ref{eq-1.8}), we consider the regularity properties of $K^{L}_{\alpha,t}(\cdot,\cdot)$.
Let $\nabla_{x,t}$ denote the gradient operator on $\mathbb R^{n}$, that is, $\nabla_{x,t}=(\nabla_{x}, {\partial}/{\partial t})$, where
$\nabla_{x}=(\partial/\partial x_{1},\partial/\partial x_{2},\ldots, \partial/\partial x_{n})$. We obtain an energy estimate of the solution to the equation:
\begin{equation}\label{eq-1.9}
\partial_{t}u+L u=\partial_{t}u-\Delta u+Vu=0.
\end{equation}
By the fundamental solution of $-\Delta$, we prove that, for any $N>0$, there exists a constant $C_{N}$ such that
$$|\nabla_{x}K^{L}_{t}(x,y)|\leq\left\{
\begin{aligned}
&\frac{C_{N}}{t^{(n+1)/2}}e^{-c|x-y|^{2}/t}\Big(1+\frac{\sqrt{t}}{\rho(x)}+\frac{\sqrt{t}}{\rho(y)}\Big)^{-N},&\quad \sqrt{t}\leq |x-y|;\\
&\frac{C_{N}}{|x-y|t^{n/2}}e^{-c|x-y|^{2}/t}\Big(1+\frac{\sqrt{t}}{\rho(x)}+\frac{\sqrt{t}}{\rho(y)}\Big)^{-N},&\quad \sqrt{t}\geq |x-y|;\\
&\frac{C_{N}}{t^{(n+1)/2}}\Big(1+\frac{\sqrt{t}}{\rho(x)}+\frac{\sqrt{t}}{\rho(y)}\Big)^{-N},&\quad \forall\ t, x,y.
\end{aligned}\right.$$
A direct computation, together with the subordinative formula, gives
$$|\nabla_{x}K^{L}_{\alpha,t}(x,y)|\leq \frac{C_{N}t^{1-1/2\alpha}}{(t^{1/2\alpha}+|x-y|)^{n+2\alpha}}\Big(1+\frac{t^{1/2\alpha}}{\rho(x)}\Big)^{-N}
\Big(1+\frac{t^{1/2\alpha}}{\rho(y)}\Big)^{-N},$$
see Proposition \ref{prop-3.3-1}.
By a similar method, we obtain the H\"older regularity of $\nabla_{x}K^{L}_{\alpha, t}(\cdot, \cdot)$, i.e., for $|h|<|x-y|/4$ and $\delta'=1-n/q$,
$$|\nabla_{x}K^{L}_{\alpha,t}(x+h,y)-\nabla_{x}K^{L}_{\alpha,t}(x,y)|
\leq C_{N}\Big(\frac{|h|}{t^{1/2\alpha}}\Big)^{\delta'}\frac{1}{t^{1/2\alpha}}\frac{t}{(t^{1/2\alpha}+|x-y|)^{n+2\alpha}},$$
see Proposition \ref{prop-3.8}.

In Section \ref{sec-3.3}, we focus on the time-fractional derivatives of $K^{L}_{\alpha,t}(\cdot,\cdot)$. Recently, there has been an increasing interest in fractional calculus since time-fractional operators are proved to be very useful for modeling purpose. For example, the following fractional heat equations
\begin{equation}\label{eq-1.1}
\partial^{\beta}_{t}u(x,t)=\Delta u(x,t)
\end{equation}
are used to describe heat propagation  in inhomogeneous media. It is known that as opposed to the classical heat equation, the equation (\ref{eq-1.1}) is known to exhibit sub diffusive behaviour and is  related with anomalous diffusions or diffusions in non-homogeneous media, with random fractal structures. Recall that the fractional derivative of $K^{L}_{\alpha,t}(\cdot,\cdot)$ is defined as
\begin{equation}\label{eq-1.4}
\partial_{t}^{\beta}K^{L}_{\alpha,t}(x,y):=\frac{e^{-i\pi(m-\beta)}}{\Gamma(m-\beta)}\int^{\infty}_{0}\partial^{m}_{t}K^{L}_{\alpha,t+u}(x,y)
u^{m-\beta}\frac{du}{u},\ \beta>0\text{ and } m=[\beta]+1.
\end{equation}
In Section \ref{sec-3.1}, we first obtain the regularity estimates of  $t^{m}\partial^{m}_{t}K^{L}_{\alpha,t}(\cdot,\cdot)$
denoted by $\widetilde{D}_{\alpha,t}^{L,m}(\cdot,\cdot)$, see Proposition \ref{pro2.6}.
Then the desired  estimates of $\partial_{t}^{\beta}K^{L}_{\alpha,t}(\cdot,\cdot)$ can be deduced from  (\ref{eq-1.4}) and Proposition \ref{pro2.6}, see Propositions \ref{prop-3.6}-\ref{prop-3.5}, respectively.


As an application, in Section \ref{sec-4},  we characterize the Camapnato type spaces associated with $L$, denoted by $BMO^{\gamma}_{L}(\mathbb R^{n})$, via the fractional heat semigroup $\{e^{-tL^{\alpha}}\}_{t>0}$. In the last decades, the characterizations of function spaces associated with Schr\"odinger operators via semigroups and Carleson measures have attracted the attentions of many authors. Let $V\in B_{q}$, $q>n/2$. Using the family of operators $\{t\partial_{t}e^{-tL}\}_{t>0}$, the Carleson measure characterization of $BMO_{L}(\mathbb R^{n})$
 was obtained by Dziuba\'nski-Garrig\'os-Mart\'inez-Torrea-Zienkiewicz \cite{DGMTZ}. Replacing the potential $V$ by a general Radon measure $\mu$, in \cite{WY},
Wu-Yan extended the result of  \cite{DGMTZ} to generalized Schr\"odinger operators. The analogue in the  setting of Heisenberg groups was obtained by Lin-Liu \cite{LL}.   Ma-Stinga-Torrea-Zhang \cite{MSTZ} characterized the Campanato type spaces associated with $L$ via the fractional derivatives of the Poisson semigroup. For further information on this topic, we refer to \cite{DY2, DYZ, HZ, jiangrenjin, Songliang, yyz1, YZhang} and the references therein. Assume that $L=-\Delta+V$ with $V\in B_{q}, q>n$. By the regularity estimates obtained in Section \ref{sec-3}, we establish the following equivalent characterizations:  for $0<\gamma<\min\{2\alpha,\ 2\alpha\beta\}$,
\begin{eqnarray*}
 f\in BMO^{\gamma}_{L}(\mathbb R^{n})
&\sim&\sup_{B}  \frac{1}{|B|^{1+2\gamma/n}}\int^{r_{B}^{2\alpha}}_{0}\int_{B}|t^{\beta}\partial_{t}^{\beta}e^{-tL^{\alpha}}(f)(x)|^{2}\frac{dxdt}{t}<\infty\\
&\sim& \sup_{B}\frac{1}{|B|^{1+2\gamma/n}}\int^{r^{2\alpha}_{B}}_{0}\int_{B}|t^{1/2\alpha}\nabla_{\alpha}e^{-tL^{\alpha}}(f)(x)|^{2}
 \frac{dxdt}{t}<\infty,\nonumber
\end{eqnarray*}
where $\nabla_{\alpha}:=(\nabla_{x}, \partial_{t}^{1/2\alpha}).$ See Theorems  \ref{th-4.2} and \ref{th-4.3}, respectively.

\begin{remark}
 The regularity estimates obtained in this paper generalize several results on the regularities of Schr\"odinger operators. Letting $\alpha=1/2$, $K^{L}_{1/2,t}(\cdot,\cdot)=P^{L}_{t}(\cdot,\cdot)$, the Poisson kernel associated with the Schr\"odinger operator. For this case,  Propositions \ref{prop-3.3-1} and \ref{prop-3.8} come back to \cite[Lemma 3.9]{DYZ}. Also, the regularities of  $\partial_{t}^{\beta}K^{L}_{\alpha,t}(\cdot,\cdot)$ obtained in Section \ref{sec-3.3} generalize \cite[Proposition 3.6, (b), (c), (d)]{MSTZ}.
\end{remark}
\begin{remark}
 For the case of $\alpha=1/2$, the regularities of $\partial^{\beta}_{t}K_{1/2,t}^{L}$ have been studied by  Ma-Stinga-Torrea-Zhang \cite{MSTZ}.
We point out that our method is slightly different from that of \cite{MSTZ}. In \cite{MSTZ}, via the Hermite polynomials $H_m(\cdot)$, the authors convert the estimate of $\partial_{t}^{\beta}P_{t}^{L}(\cdot,\cdot)$ to the estimate of $\partial_{t}K^{L}_{t}(\cdot,\cdot)$, see \cite[(3.12)]{MSTZ}. In Section \ref{sec-3.3}, we estimate the time-fractional derivatives of $K^{L}_{\alpha,t}(\cdot,\cdot)$ via $\partial_{t}^{m}K^{L}_{\alpha, t}(\cdot,\cdot)$ instead of the Hermite polynomials.
\end{remark}

{\it Some notations}:
\begin{itemize}
\item ${\mathsf U}\approx{\mathsf V}$ represents that
there is a constant $c>0$ such that $c^{-1}{\mathsf V}\le{\mathsf
U}\le c{\mathsf V}$ whose right inequality is also written as
${\mathsf U}\lesssim{\mathsf V}$. Similarly, one writes ${\mathsf V}\gtrsim{\mathsf U}$ for
${\mathsf V}\ge c{\mathsf U}$.

\item For convenience, the positive constants $C$
may change from one line to another and usually depend
on the dimension $n,$ $\alpha,$ $\beta$
and other fixed parameters.

\item Let $B$ be a ball with the radius $r$. In the rest of this paper, for $c>0$, we denote by $B_{cr}$ the ball with the same center and radius $cr$.
\end{itemize}

\section{Preliminaries}\label{sec2}
\subsection{The Schr\"odinger operator}
Let $L=-\Delta+V$ be the Schr\"{o}dinger differential operator
on $\mathbb R^{n}, n\geq 3.$ Throughout the paper, we will assume that
$V$ is a nonzero, nonnegative potential, and belongs to the reverse H\"older class $B_{q}, q>n/2,$ which is defined in Definition \ref{def-1.1}.
By H\"{o}lder's inequality, we can get $B_{q_{1}}\subset B_{q_{2}}$
for $q_{1}\geq q_{2}>1$. One remarkable feature about the class $B_{q}$
 is that if $V\in B_{q}$ for some $q>1$, then there exists
$\varepsilon>0$, which depends only on $n$ and the constant $C$ in
(\ref{eq-2.1}), such that $V\in B_{q+\varepsilon}$. It is also well known that
if $V\in B_{q}, q>1$ then $V(x)dx$ is a doubling measure. Namely,
for any $r>0,x\in \mathbb R^{n}$,
\begin{equation}\label{eq-2.4}
\int_{B(x,2r)}V(y)dy\leq C_{0}\int_{B(x,r)}V(y)dy.
\end{equation}


The auxiliary function $m(x,V)$ is defined by
\begin{equation}\label{eq-2.3}
\frac{1}{m(x,V)}:=\sup\Big\{r>0:\ \frac{1}{r^{n-2}}\int_{B(x,r)}V(y)dy\leq 1\Big\}.
\end{equation}
Clearly, $0<m(x,V)<\infty$ for every $x\in \mathbb R^{n}$ and if
$r={1}/{m(x,V)}$, then
$\frac{1}{r^{n-2}}\int_{B(x,r)}V(y)dy=1.$ For simplicity, we
denote ${1}/{m(x,V)}$ by $\rho(x)$ in the proof sometime.
We state some
properties of $m(x,V)$ which will be used in the proofs of main results.

\begin{lemma} \label{le-2.2}{\rm (\cite[Lemma 1.2]{Shen-2})} There exists a
constant $C>0$ such that for every $0<r<R<\infty$ and $ y\in
\mathbb{R}^{n}$, we have
$$\frac{1}{r^{n-2}}\int_{B(y,r)}V(x)dx\leq C\Big(\frac{r}{R}\Big)^{2-n/q}\frac{1}{R^{n-2}}\int_{B(y,R)}^{}V(x)dx.$$
\end{lemma}
\begin{lemma} {\rm(\cite[Lemma 3]{Kurata})}  For every constant $C_{1}>1$, there exists a constant
$C_{2}>1$ such that if
$$\frac{1}{C_{1}}\leq\frac{1}{r^{n-2}}\int_{B(x,r)}^{}V(y)dy\leq
C_{1},$$ then $C^{-1}_{2}\leq rm(x,V)\leq C_{2}$.
\end{lemma}
\begin{lemma}\label{lem2.4}
 {\rm(\cite[Lemma 1.4]{Shen-2})}\quad For every constant $C_{1}\geq 1$, there is a constant
$C_{2}\geq 1$ such that
$$\frac{1}{C_{2}}\leq\frac{m(x,V)}{m(y,V)}\leq C_{2}$$ for $|x-y|\leq
C_{1}\rho(x).$ Moreover, there exist constants $k_{0}, C, c>0$ such
that
$$\left\{\begin{aligned}
&m(y,V)\leq C(1+|x-y|m(x,V))^{k_{0}}m(x,V);\\
&m(y,V)\geq cm(x,V)(1+|x-y|m(x,V))^{-k_{0}/(1+k_{0})}.
\end{aligned}\right.
$$
\end{lemma}
\begin{lemma}\label{le-2.5}
 {\rm(\cite[Lemma 1.8]{Shen-2})}  There exist constants $k_{0}, C>0$
such that for $R\geq m(x,V)^{-1}$,
$$\frac{1}{R^{n-2}}\int_{B(x,R)}V(y)dy\leq C(Rm(x,V))^{k_{0}}.$$
\end{lemma}

\begin{lemma}\label{le-2.6} {\rm \cite[Lemma 1]{GLP}}
Suppose $V\in B_{q}$, $q>n/2$. Let $m_{0}>\log_{2}C_{0}+1$, where $C_{0}$ is the constant in (\ref{eq-2.4}). Then for any $x_{0}\in\mathbb R^{n}$, $R>0$,
$$\frac{1}{\{1+Rm(x_{0},V)\}^{m_{0}}}\int_{B(x_{0}, R)}V(x)dx\leq CR^{n-2}.$$
\end{lemma}

As a corollary of \cite[Corollary 4.8]{Dz2}, we have
\begin{lemma}\label{lem2.6}
 There exist constants $C$, $\delta$ and $l$ such that
$$\frac{1}{t^{n/2}}\int_{\mathbb{R}^{n}}e^{-c|x-y|^{2}/t}V(y)dy\leq
\left\{\begin{aligned}
 \frac{C}{t}\Big(\frac{\sqrt{t}}{\rho(x)}\Big)^{\delta},\ \sqrt{t}< m(x,V)^{-1};\\
 \frac{C}{t}\Big(\frac{\sqrt{t}}{\rho(x)}\Big)^{l},\ \sqrt{t}\geq m(x, V)^{-1}.
 \end{aligned}\right.$$
\end{lemma}

 Since the potential $V$ is nonnegative, it follows from the Feynman-Kac formula that the kernels $K^{L}_{t}(\cdot,\cdot)$ have a Gaussian upper bound:
$$0\leq K^{L}_{t}(x,y)\leq \frac{1}{(4\pi t)^{n/2}}e^{-|x-y|^{2}/4t}.$$
Furthermore,
\begin{proposition}\label{lem2.1} {\rm (\cite[Theorem 4.10]{dz2})} For
every $N>0$, there exist  constants $C_{N}$ and $c$ such that for all $x,y\in\mathbb{R}^n,$
\begin{equation}\label{eq-gauss-bdd}
0\leq K^{L}_{t}(x,y)\leq
\frac{C_{N}}{t^{n/2}}e^{-c|x-y|^{2}/t}\Big(1+\frac{\sqrt{t}}{\rho(x)}+\frac{\sqrt{t}}{\rho(y)}\Big)^{-N}.
\end{equation}
\end{proposition}

\begin{proposition}\label{prop2}{\rm(\cite[Proposition 4.11]{dz2})}  For every $0<\delta'<\delta_{0}=\min\{1,
2-n/q\}$ and every
$N>0$, there exist constants $C_{N}>0$ and $c$ such that for $|h|<\sqrt{t}$,
$$|K^{L}_{t}(x+h,y)-K^{L}_{t}(x,y)|\leq \frac{C_{N}}{t^{n/2}}\Big(\frac{|h|}{\sqrt{t}}\Big)^{\delta'}e^{-c|x-y|^{2}/t}
\Big(1+\frac{\sqrt{t}}{\rho(x)}+\frac{\sqrt{t}}{\rho(y)}\Big)^{-N}.$$
\end{proposition}

\begin{remark}
By Proposition \ref{lem2.1}, it is easy to see that  the condition $|h|<\sqrt{t}$ in  Proposition \ref{prop2} can be replaced by $|h|<|x-y|/2$.
\end{remark}

Let $Q_{t,m}^{L}(x,y):=t^{m}\partial_{t}^{m}K^{L}_{t}(x,y)$, $m\in\mathbb Z_{+}$. Then
\begin{proposition}\label{prop-2.1} {\rm (\cite[Proposition 3.3]{HLL-2})}
Let $m\in\mathbb Z_{+}.$
\item{\rm (i)} For every $N>0$, there exist constants $C_{N}>0$ and $c>0$ such that
$$|Q^{L}_{t,m}(x,y)|\leq \frac{C_{N}}{t^{n/2}}e^{-c|x-y|^{2}/t}\Big(1+\frac{\sqrt{t}}{\rho(x)}+\frac{\sqrt t}{\rho(y)}\Big)^{-N}.$$
\item{\rm (ii)} Let $0<\delta'\leq\delta_{0}$, where $\delta_0$ appears in Proposition \ref{prop2}. For every $N>0$,
there exist constants $C_{N}>0$ and $c$ such that  for $|h|<t$,
$$|Q^{L}_{t,m}(x+h,y)-Q^{L}_{t,m}(x,y)|\leq\frac{C_{N}}{t^{n/2}}e^{-c|x-y|^{2}/t}\Big(\frac{|h|}{\sqrt t}\Big)^{\delta'}
\Big(1+\frac{\sqrt t}{\rho(x)}+\frac{\sqrt t}{\rho(y)}\Big)^{-N}.$$
\item{\rm (iii)} For every $N>0$ and $0<\delta'\leq\delta_{0}$, there exists a constant $C_{N}>0$ such that
$$\Big|\int_{\mathbb{R}^{n}}{}Q^{L}_{t,m}(x,y)d\mu(y)\Big|\leq C_{N}\Big(\frac{\sqrt t}{\rho(x)}\Big)^{\delta'}\Big(1+\frac{\sqrt t}{\rho(x)}\Big)^{-N}.$$
\end{proposition}

\subsection{Fractional heat kernels associated with $L$}
In this section,  we first state some backgrounds on the fractional heat semigroup and the fractional heat kernel associated with $L$.
For the case $V\neq0$, the fractional heat semigroup associated with $L$ can not be defined via the Fourier multiplier method (\ref{eq-1.3}) as
 the Laplace operator. We strike out on a new path and introduce the fractional heat semigroup via the subordinative formula.

The Schr\"odinger operator $L$ can be seen as the generator of the semigroup
 $\{e^{-tL}\}_{t>0}$, i.e.,
 $$L(f):=\lim_{t\rightarrow 0}\frac{f-e^{-tL}f}{t},$$
where the limit is in $L^{2}(\mathbb{R}^{n})$.
$L$ is a self-adjoint, positive definite operator. The integral kernels of the semigroups $\{e^{-tL}\}_{t>0}$ are denoted by $K^{L}_{t}(\cdot,\cdot)$.
It is easy to verify that the kernel $K^{L}_{t}(\cdot,\cdot)$ satisfies
$$\left\{\begin{aligned}
&{\rm (i)}\ K^{L}_{t}(x,y)\geq 0,\ x,y\in\mathbb{R}^{n};\\
&{\rm (ii)}\ K^{L}_{t}(x,y)=K^{L}_{t}(y,x);\\
&{\rm (iii)}\ K^{L}_{s+t}(x,y)=\int_{\mathbb{R}^{n}}K^{L}_{s}(x,z)K^{L}_{t}(z,y)dz;\\
&{\rm (iv)}\ \lim\limits_{t\rightarrow 0+}\int_{\mathbb{R}^{n}}K^{L}_{t}(x,y)f(y)dy=f(x)\quad \forall\ f\in L^{2}(\mathbb{R}^{n}).
\end{aligned}\right.$$
For $\alpha\in(0,1)$, the fractional power of $L$, denoted by $L^{\alpha}$, is defined as
$$L^{\alpha}(f)=\frac{1}{\Gamma(-\alpha)}\int^{\infty}_{0}\Big(e^{-t\sqrt{L}}f(x)-f(x)\Big)\frac{dt}{t^{1+2\alpha}}\quad \forall\ f\in L^{2}(\mathbb{R}^{n}).$$
Here $\{e^{-t\sqrt{L}}\}_{t>0}$ denotes the Poisson semigroup related to $L$ with the kernel $P^{L}_{t}(\cdot,\cdot)$ defined as
\begin{eqnarray*}
P^{L}_{t}(x,y)&=&\int^{\infty}_{0}\frac{e^{-u}}{\sqrt{u}}K^{L}_{t^{2}/4u}(x,y)du=\int^{\infty}_{0}\frac{e^{-t^{2}/4s}}{2s^{3/2}}K^{L}_{s}(x,y)ds.
\end{eqnarray*}

By the subordinative formula, we know that there exists a non-negative continuous function $\eta^{\alpha}_{t}(\cdot)$ satisfying (cf. \cite{Gri})
\begin{equation}\label{eq-heat}
\left\{\begin{aligned}
&\eta^{\alpha}_{t}(s)=\frac{1}{t^{1/\alpha}}\eta^{\alpha}_{1}(s/t^{1/\alpha});\\
&\eta^{\alpha}_{t}(s)\lesssim \frac{t}{s^{1+\alpha}}  \ \forall\ s,t>0;\\
&\int^{\infty}_{0}s^{-\gamma}\eta^{\alpha}_{1}(s)ds<\infty,\ \gamma>0;\\
&\eta^{\alpha}_{t}(s)\simeq\frac{t}{s^{1+\alpha}} \,\ \forall\,\
s\geq t^{1/\alpha}>0,
\end{aligned}\right.
\end{equation}
such that $K^{L}_{\alpha,t}(\cdot,\cdot)$ can be expressed as
\begin{equation}\label{eq-sub-for}
K^{L}_{\alpha,t}(x,y)=\int^{\infty}_{0}\eta^{\alpha}_{t}(s)K^{{L}}_{s}(x,y)ds,
\end{equation}
 see \cite{Gri} for some examples of $\eta^{\alpha}_{t}(\cdot)$. The function $\eta^{\alpha}_{t}(\cdot)$ plays an important role in the estimate of the fractional heat kernel $K^{L}_{\alpha,t}(\cdot,\cdot)$.
  Take $\alpha=1/2$ for example: by (\ref{eq-1.4-1}), we can see that $\eta^{1/2}_{t}(s)=\frac{t}{2s^{3/2}}e^{-t^{2}/4s}\ \forall\ s, t>0.$ It is
  easy to verify that such $\eta^{1/2}_{t}(\cdot)$ satisfies conditions in (\ref{eq-heat}). For the special case $L=-\Delta$, a direct computation gives
$$P^{-\Delta}_{t}(x,y)=\int^{\infty}_{0}\frac{e^{-t^{2}/4s}t}{2s^{3/2}}s^{-n/2}e^{-|x-y|^{2}/s}ds=\frac{c_{n}t}{(t^{2}+|x-y|^{2})^{(n+1)/2}},$$
which coincides with the classical Poisson kernel obtained via (\ref{eq-1.7}).

\subsection{Campanato type  spaces associated with $L$}

The Campanato type space associated with $L$
is defined as follows.
\begin{definition}
 The space $BMO^{\gamma}_{L}(\mathbb R^{n})$, $0< \gamma\leq 1$, is defined as the set of all locally integrable functions $f$ satisfying
there exists a constant $C$ such that
\begin{equation}\label{eq-2.2}
\sup_{B}\frac{1}{|B|^{1+\gamma/n}}\int_{B}|f(x)-f(B,V)|dx\leq C<\infty,
\end{equation}
where the supremum is taken over all balls $B$ centered at $x_{B}$ with radius $r_{B}$, and
$$f(B,V):=
\begin{cases}
f_{B},\ r_{B}<\rho(x_{B});\\
0,\ r_{B}\geq\rho(x_{B}).
\end{cases}$$
 The norm $\|f\|_{BMO^{\gamma}_{L}}$ is defined as the infimum of the constants $C$ such that (\ref{eq-2.2}) above
holds.
\end{definition}

\begin{proposition}{\rm (\cite[Proposition 4.3]{MSTZ})}
Let $B=B(x,r)$ with $r<\rho(x)$. If $f\in BMO^{\gamma}_{L}(\mathbb R^{n}), 0<\gamma\leq1$, then there exists a constant $C_{\gamma}$ such that
$|f_{B}|\leq C_{\gamma}(\rho(x))^{\gamma}\|f\|_{BMO^{\gamma}_{L}}$.
\end{proposition}

The space $BMO^{\gamma}_{L}(\mathbb R^{n})$ is equivalent to the following Lipschitz type space related to $L$.
\begin{definition}
For $0<\gamma \leq 1$, a continuous function $f$ defined on $\mathbb R^{n}$ belongs to the space $C^{0,\gamma}_{L}(\mathbb R^{n})$ if
$$\sup_{x, y\in\mathbb R^{n}}\frac{|f(x)-f(y)|}{|x-y|^{\gamma}}<\infty\text{ and }
\sup_{x\in\mathbb R^{n}}\frac{|f(x)|}{\rho(x)^{\gamma}}<\infty.$$
\end{definition}

\begin{proposition}{\rm (\cite[Proposition 4.6]{MSTZ})}
 If\, $0<\gamma\leq 1$,  then the spaces $BMO^\gamma_L(\mathbb R^{n})$ and $C^{0,\gamma}_{L}(\mathbb R^{n})$ are equal and their norms are
equivalent.
\end{proposition}

It is well known that Hardy spaces $H^{p}(\mathbb R^{n}), 0<p\leq 1$, are the predual spaces of Campanato spaces (cf. \cite{FS1}). In 2000s, such dual relationship was extended to function spaces associated with operators, see \cite{DY1, DZ1,DGMTZ, hofmann,LTZ,yyz2}. For a Schr\"odinger operator $L$, the following Hardy type spaces $H^{p}_{L}(\mathbb R^{n}), 0<p\leq 1$, were introduced in \cite{Dz2,dz2}.
\begin{definition}
For $0<p\leq 1$, an integrable function $f$ is an element of the Hardy type space $H^{p}_{L}(\mathbb R^{n})$  if the maximal function
$$T^{\ast}(f)(x):=\sup_{s>0}|T^{L}_{s}(f)(x)|$$
belongs to $L^{p}(\mathbb R^{n})$. The quasi-norm in $H^{p}_{L}(\mathbb R^{n})$ is defined by
$$\|f\|_{H^{p}_{L}}:=\|T^{\ast}(f)\|_{L^{p}}.$$
\end{definition}

Let ${\delta_0}=\min\{1, 2-n/q\}$ and $n/(n+\delta_{0})<p\leq 1$. An atom of $H^{p}_{L}(\mathbb R^{n})$ associated with a ball $B(x_{B}, r_{B})$
is a function $a$ such that
$$\left\{\begin{aligned}
&\text{ supp }a\subseteq B(x_{B}, r_{B}),\ r_{B}\leq \rho(x_{B});\\
&\|a\|_{L^{\infty}}\leq |B(x_{B}, r_{B})|^{-1/p};\\
&\int_{\mathbb R^{n}}a(x)dx=0,\ r_{B}<\rho(x_{B})/4.
\end{aligned}\right.$$

In \cite{dz2}, Dziuba\'nski and Zienkiewicz obtained the following atomic characterization of $H^{p}_{L}(\mathbb R^{n})$.
\begin{proposition} {\rm (\cite[Theorem 1.13]{dz2})}
Let $n/(n+\delta_0)<p\leq 1$. $f\in H^{p}_{L}(\mathbb R^{n})$ if and only if $f=\sum_{j}\lambda_{j}a_{j}$, where $\{a_{j}\}$ are $H^{p}_{L}$-atoms and $\sum_{j}|\lambda_{j}|^{p}<\infty$.
\end{proposition}

\begin{theorem} {\rm (\cite[Theorem 4.5]{MSTZ})}
Let $0\leq \gamma<1$. Then the dual space of $H^{n/(n+\gamma)}_{L}(\mathbb R^{n})$ is $BMO^{\gamma}_{L}(\mathbb R^{n})$.
 More precisely, any continuous linear functional $\Phi$ over $H^{n/(n+\gamma)}_{L}(\mathbb R^{n})$ can be represented as
 $$\Phi(a)=\int_{\mathbb R^{n}}f(x)a(x)dx$$
 for some function $f\in BMO^{\gamma}_{L}(\mathbb R^{n})$ and all $H^{n/(n+\gamma)}_{L}$-atoms $a$. Moreover, the operator norm $\|\Phi\|_{op}\sim
 \|f\|_{BMO^{\gamma}_{L}}$.
\end{theorem}

\begin{lemma}\label{le-4.1}{\rm (\cite[Lemma 7]{DGMTZ})}
Let $q_{t}(x,y)$ be a function of $x,y\in\mathbb R^{n}$, $t>0$. Assume that for every $N>0$, there exists a constant $C_{N}$ such that for some $\gamma'\geq\gamma$,
$$|q_{t}(x,y)|\leq C_{N}(1+t/\rho(x)+t/\rho(y))^{-N}t^{-n}(1+|x-y|/t)^{-(n+\gamma')}.$$
Then for every $H^{n/(n+\gamma)}_{L}$-atom $g$ supported on $B(x_{0}, r)$, there exists $C_{N, x_{0}, r}>0$ such that
$$\sup_{t>0}\Big|\int_{\mathbb R^{n}}q_{t}(x,y)g(y)dy\Big|\leq C_{N, x_{0}, r}(1+|x|)^{-(n+\gamma')},\ x\in\mathbb R^{n}.$$
\end{lemma}

\section{Regularities on fractional heat semigroups associated with $L$}\label{sec-3}
The aim of this section is to estimate the regularities of the fractional heat kernel $K_{\alpha,t}^{L}(\cdot,\cdot)$. By use of (\ref{eq-sub-for-1}), we first estimate $\partial_{t}^{m}K_{\alpha,t}^{L}(\cdot,\cdot), m\geq 1.$ Then, via the solution to (\ref{eq-1.9}), we investigate the spatial gradient of $K_{\alpha,t}^{L}(\cdot,\cdot)$. At last, we obtain the estimation of the time-fractional derivatives of  $K_{\alpha,t}^{L}(\cdot,\cdot)$.
\subsection{Regularities of the fractional heat kernel}\label{sec-3.1}
We first investigate the regularities of $K^{L}_{\alpha,t}(\cdot,\cdot)$.
\begin{proposition}\label{prop-3.1}
Let $\alpha\in(0,1)$. For every $N>0$, there exists a constant $C_{N}$ such that
$$\Big|K^{L}_{\alpha,t}(x,y)\Big|\leq \frac{C_{N}t}{(t^{1/2\alpha}+|x-y|)^{n+2\alpha}}\Big(1+\frac{t^{1/2\alpha}}{\rho(x)}+\frac{t^{1/2\alpha}}{\rho(y)}\Big)^{-N}.$$
\end{proposition}

\begin{proof}
By Proposition \ref{lem2.1}, we use (\ref{eq-gauss-bdd}), (\ref{eq-heat}) and (\ref{eq-sub-for}) to obtain
\begin{eqnarray*}
\Big|K^{L}_{\alpha,t}(x,y)\Big|&\leq&\int^{\infty}_{0}\frac{t}{s^{1+\alpha}}\Big|K^{L}_{s}(x,y)\Big|ds\\
&\leq& C_{N}\int^{\infty}_{0}\frac{te^{-c|x-y|^{2}/s}}{s^{1+\alpha+n/2}}\Big(1+\frac{\sqrt{s}}{\rho(x)}\Big)^{-N}\Big(1+\frac{\sqrt{s}}{\rho(y)}\Big)^{-N}ds.
\end{eqnarray*}
Taking the change of variable $s=t^{1/\alpha}u$, we have
\begin{eqnarray*}
\Big|K^{L}_{\alpha,t}(x,y)\Big|&\leq&C_{N}\int^{\infty}_{0}\frac{t^{1+1/\alpha}}{(t^{1/\alpha}u)^{1+\alpha}}(t^{1/\alpha}u)^{-n/2}e^{-\frac{c|x-y|^{2}}{t^{1/\alpha}u}}
\Big(1+\frac{t^{1/2\alpha}\sqrt{u}}{\rho(x)}\Big)^{-N}\Big(1+\frac{t^{1/2\alpha}\sqrt{u}}{\rho(y)}\Big)^{-N}du\\
&\leq&\frac{C_{N}}{t^{n/2\alpha}}\Big(\frac{t^{1/2\alpha}}{\rho(x)}\Big)^{-N}\Big(\frac{t^{1/2\alpha}}{\rho(y)}\Big)^{-N}\int^{\infty}_{0}
\frac{1}{u^{1+\alpha}}u^{-n/2-N}e^{-\frac{c|x-y|^{2}}{t^{1/\alpha}u}}du.
\end{eqnarray*}
Let $\frac{|x-y|^{2}}{t^{1/\alpha }u}=r^{2}$. Then
\begin{eqnarray*}
\Big|K^{L}_{\alpha,t}(x,y)\Big|&\leq&C_{N}t^{-n/2\alpha}
\Big(\frac{t^{1/2\alpha}}{\rho(x)}\Big)^{-N}\Big(\frac{t^{1/2\alpha}}{\rho(y)}\Big)^{-N}
\int^{\infty}_{0}\Big(\frac{|x-y|^{2}}{t^{1/\alpha}r^{2}}\Big)^{-1-\alpha-n/2-N}e^{-cr^{2}}\frac{|x-y|^{2}}{t^{1/\alpha}r^{3}}dr\\
&\leq&C_{N}\Big(\frac{t^{1/2\alpha}}{\rho(x)}\Big)^{-N}\Big(\frac{t^{1/2\alpha}}{\rho(y)}\Big)^{-N}
\frac{t^{1+N/\alpha}}{|x-y|^{2\alpha+n+2N}}
\int^{\infty}_{0}r^{2\alpha+n+2N-1}e^{-cr^{2}}dr\\
&\leq&C_{N}\Big(\frac{t^{1/2\alpha}}{\rho(x)}\Big)^{-N}\Big(\frac{t^{1/2\alpha}}{\rho(y)}\Big)^{-N}\frac{t^{1+N/\alpha}}{|x-y|^{n+2\alpha+2N}},
\end{eqnarray*}
which gives
\begin{equation}\label{eq-3.1-11}
\Big(\frac{t^{1/2\alpha}}{\rho(x)}\Big)^{N}\Big(\frac{t^{1/2\alpha}}{\rho(y)}\Big)^{N}\Big|K^{L}_{\alpha,t}(x,y)\Big|\leq
\frac{C_{N}t^{1+N/\alpha}}{|x-y|^{n+2\alpha+2N}}.
\end{equation}

On the other hand, using the change of variables again, we obtain
\begin{eqnarray*}
\Big|K^{L}_{\alpha,t}(x,y)\Big|&\leq&C_{N} \int^{\infty}_{0}s^{-n/2}\eta^{\alpha}_{1}\Big(\frac{s}{t^{1/\alpha}}\Big)\Big(1+\frac{\sqrt{s}}{\rho(x)}\Big)^{-N}
\Big(1+\frac{\sqrt{s}}{\rho(y)}\Big)^{-N}ds\\
&\leq&C_{N}\int^{\infty}_{0}(t^{1/\alpha}\tau)^{-n/2}\frac{1}{t^{1/\alpha}}\eta^{\alpha}_{1}(\tau)
\Big(1+\frac{\sqrt{\tau}t^{1/2\alpha}}{\rho(x)}\Big)^{-N}
\Big(1+\frac{\sqrt{\tau}t^{1/2\alpha}}{\rho(y)}\Big)^{-N}t^{1/\alpha}d\tau\\
&\leq&C_{N}t^{-n/2\alpha}\Big(\frac{t^{1/2\alpha}}{\rho(x)}\Big)^{-N}\Big(\frac{t^{1/2\alpha}}{\rho(y)}\Big)^{-N}
\int^{\infty}_{0}\tau^{-n/2-N}\eta^{\alpha}_{1}(\tau)d\tau.
\end{eqnarray*}
The above estimate implies that
\begin{equation}\label{eq-3.2-11}
\Big(\frac{t^{1/2\alpha}}{\rho(x)}\Big)^{N}\Big(\frac{t^{1/2\alpha}}{\rho(y)}\Big)^{N}\Big|K^{L}_{\alpha,t}(x,y)\Big|\leq \frac{C_{N}}{t^{n/2\alpha}}.
\end{equation}
Now, combining (\ref{eq-3.1-11})\ and \ (\ref{eq-3.2-11}),  we have
$$\Big(\frac{t^{1/2\alpha}}{\rho(x)}\Big)^{N}\Big(\frac{t^{1/2\alpha}}{\rho(y)}\Big)^{N}\Big|K^{L}_{\alpha,t}(x,y)\Big|\leq C_{N}\min\Big\{
\frac{t^{1+N/\alpha}}{|x-y|^{n+2\alpha+2N}},\ t^{-n/2\alpha}\Big\},$$
which, together with the arbitrariness of $N$, indicates that
$$\Big|K^{L}_{\alpha,t}(x,y)\Big|\leq \frac{C_{N}t}{(t^{1/2\alpha}+|x-y|)^{n+2\alpha}}\Big(1+\frac{t^{1/2\alpha}}{\rho(x)}+\frac{t^{1/2\alpha}}{\rho(y)}\Big)^{-N}.$$
This completes the proof of Proposition \ref{prop-3.1}.
\end{proof}

\begin{proposition}\label{prop-3.2}
Let $\alpha\in (0,1)$. For any $N>0$, there exists  a constant $C_{N}>0$ such that for every $0<\delta'<\delta_{0}=\min\{1,
2-n/q\}$ and all $|h|\leq t^{1/2\alpha}$,
\begin{eqnarray*}
|K^{L}_{\alpha,t}(x+h,y)-K^{L}_{\alpha,t}(x,y)|
&\leq&\frac{C_{N}(|h|/t^{1/2\alpha})^{\delta}t}{(t^{1/2\alpha}+|x-y|)^{n+2\alpha}}
\Big(1+\frac{t^{1/2\alpha}}{\rho(x)}+\frac{t^{1/2\alpha}}{\rho(x)}\Big)^{-N}.
\end{eqnarray*}
\end{proposition}

\begin{proof}
The proof is similar to that of Proposition \ref{prop-3.1}. We first assume that $|h|<|x-y|/2$. By the subordinative formula (\ref{eq-sub-for}), we can use Proposition \ref{prop2} to get
\begin{eqnarray*}
\Big|K^{L}_{\alpha,t}(x+h,y)-K^{L}_{\alpha,t}(x,y)\Big|
&\leq&C_{N} \int^{\infty}_{0}\frac{t}{s^{1+\alpha}}s^{-n/2}e^{-c|x-y|^{2}/s}
\Big(|h|/\sqrt{s}\Big)^{\delta'}\Big(1+\frac{\sqrt{s}}{\rho(x)}\Big)^{-N}\Big(1+\frac{\sqrt{s}}{\rho(y)}\Big)^{-N}ds\\
&\leq&C_{N}\Big(\frac{|h|}{t^{1/2\alpha}}\Big)^{\delta'}\Big(\frac{t^{1/2\alpha}}{\rho(x)}\Big)^{-N}\Big(\frac{t^{1/2\alpha}}{\rho(y)}\Big)^{-N}
\frac{t^{1+N/\alpha+\delta'/2\alpha}}{|x-y|^{2\alpha+n+2N+\delta'}}\\
&\leq&C_{N}\Big(\frac{t^{1/2\alpha}}{\rho(x)}\Big)^{-N}\Big(\frac{t^{1/2\alpha}}{\rho(y)}\Big)^{-N}\frac{t^{1+N/\alpha}|h|^{\delta'}}{|x-y|^{2\alpha+n+2N+\delta'}},
\end{eqnarray*}
which implies
\begin{equation}\label{eq-3.3-1}
\Big(\frac{t^{1/2\alpha}}{\rho(x)}\Big)^{N}\Big(\frac{t^{1/2\alpha}}{\rho(y)}\Big)^{N}\Big|K^{L}_{\alpha,t}(x+h,y)-K^{L}_{\alpha,t}(x,y)\Big|
\leq \frac{C_{N}t^{1+N/\alpha}|h|^{\delta'}}{|x-y|^{2\alpha+n+2N+\delta'}}.
\end{equation}

On the other hand, letting $\tau=s/t^{1/\alpha}$, we have
\begin{eqnarray*}
&&\Big|K^{L}_{\alpha,t}(x+h,y)-K^{L}_{\alpha,t}(x,y)\Big|\\
&&\leq C_{N}\int^{\infty}_{0}s^{-n/2}\frac{1}{t^{1/\alpha}}\eta^{\alpha}_{1}(s/t^{1/\alpha})
\Big(\frac{|h|}{\sqrt{s}}\Big)^{\delta'}\Big(1+\frac{\sqrt{s}}{\rho(x)}\Big)^{-N}\Big(1+\frac{\sqrt{s}}{\rho(y)}\Big)^{-N}ds\\
&&\leq C_{N}\int^{\infty}_{0}(t^{1/\alpha}\tau)^{-n/2}\eta^{\alpha}_{1}(\tau)\Big(\frac{|h|}{\sqrt{\tau}t^{1/2\alpha}}\Big)^{\delta'}
\Big(1+\frac{\sqrt{\tau}t^{1/2\alpha}}{\rho(x)}\Big)^{-N}\Big(1+\frac{\sqrt{\tau}t^{1/2\alpha}}{\rho(y)}\Big)^{-N}d\tau\\
&&\leq C_{N}\Big(\frac{t^{1/2\alpha}}{\rho(x)}\Big)^{-N}\Big(\frac{t^{1/2\alpha}}{\rho(y)}\Big)^{-N}t^{-n/2\alpha-\delta'/2\alpha}|h|^{\delta'}
\int^{\infty}_{0}\tau^{-n/2-N-\delta'/2}\eta^{\alpha}_{1}(\tau)d\tau.
\end{eqnarray*}
This gives
\begin{equation}\label{eq-3.4}
\Big(\frac{t^{1/2\alpha}}{\rho(x)}\Big)^{N}\Big(\frac{t^{1/2\alpha}}{\rho(y)}\Big)^{N}
\Big|K^{L}_{\alpha,t}(x+h,y)-K^{L}_{\alpha,t}(x,y)\Big|\leq C_{N}t^{-n/2\alpha}\Big(\frac{|h|}{t^{1/2\alpha}}\Big)^{\delta'}.
\end{equation}

The estimates (\ref{eq-3.3-1})\  and \ (\ref{eq-3.4}) indicate that
$$\Big(\frac{t^{1/2\alpha}}{\rho(x)}\Big)^{N}\Big(\frac{t^{1/2\alpha}}{\rho(y)}\Big)^{N}
\Big|K^{L}_{\alpha,t}(x+h,y)-K^{L}_{\alpha,t}(x,y)\Big|\leq C_{N}\min\Bigg\{\frac{t^{1+N/\alpha}|h|^{\delta'}}{|x-y|^{2\alpha+n+2N+\delta'}},\ t^{-n/2\alpha}\Big(\frac{|h|}{t^{1/2\alpha}}\Big)^{\delta'}\Bigg\}.$$
Due to the arbitrariness of $N$, we have
$$\Big|K^{L}_{\alpha,t}(x+h,y)-K^{L}_{\alpha,t}(x,y)\Big|\leq \frac{C_{N}t}{(t^{1/2\alpha}+|x-y|)^{n+2\alpha}}\Big(\frac{|h|}{t^{1/2\alpha}}\Big)^{\delta'}
\Big(1+\frac{t^{1/2\alpha}}{\rho(x)}+\frac{t^{1/2\alpha}}{\rho(x)}\Big)^{-N}.$$
This proves Proposition \ref{prop-3.2} under the assumption $|h|<|x-y|/2$.

Now we prove this proposition  for the case $|h|<t^{1/2\alpha}$. For $|h|<|x-y|/2<t^{1/2\alpha}$ or $|h|<t^{1/2\alpha}<|x-y|/2$, the desired estimate can be deduced from (\ref{eq-3.3-1})\  and \ (\ref{eq-3.4}). It remains to consider the case $|x-y|/2<|h|<t^{1/2\alpha}$. We split
\begin{eqnarray*}
\Big|K^{L}_{\alpha,t}(x+h,y)-K^{L}_{\alpha,t}(x,y)\Big|&\leq&S_{1}+S_{2},
\end{eqnarray*}
where
$$\left\{\begin{aligned}
S_{1}&:=\int_{|h|<\sqrt{s}}\eta^{\alpha}_{t}(s)\Big|K_{\alpha,s}^{L}(x+h,y)-K_{\alpha,s}^{L}(x,y)\Big|ds;\\
S_{2}&:=\int_{|h|\geq \sqrt{s}}\eta^{\alpha}_{t}(s)\Big|K_{\alpha,s}^{L}(x+h,y)-K_{\alpha,s}^{L}(x,y)\Big|ds.
\end{aligned}\right.$$
For $S_{1}$, since $|h|<\sqrt{s}$, we can follow the procedure of (\ref{eq-3.4}) to deduce that
\begin{eqnarray*}
S_{1}&\lesssim&\int_{|h|<\sqrt{s}}s^{-n/2}\frac{1}{t^{1/\alpha}}\eta^{\alpha}_{1}(s/t^{1/\alpha})
\Big(\frac{|h|}{\sqrt{s}}\Big)^{\delta'}\Big(1+\frac{\sqrt{s}}{\rho(x)}\Big)^{-N}\Big(1+\frac{\sqrt{s}}{\rho(y)}\Big)^{-N}ds\\
&\lesssim& \int^{\infty}_{0}(t^{1/\alpha}\tau)^{-n/2}\eta^{\alpha}_{1}(\tau)\Big(\frac{|h|}{\sqrt{\tau}t^{1/2\alpha}}\Big)^{\delta'}
\Big(1+\frac{\sqrt{\tau}t^{1/2\alpha}}{\rho(x)}\Big)^{-N}\Big(1+\frac{\sqrt{\tau}t^{1/2\alpha}}{\rho(y)}\Big)^{-N}d\tau\\
&\lesssim& \Big(\frac{t^{1/2\alpha}}{\rho(x)}\Big)^{-N}\Big(\frac{t^{1/2\alpha}}{\rho(y)}\Big)^{-N}t^{-n/2\alpha}\Big(\frac{|h|}{t^{1/2\alpha}}\Big)^{\delta'}
\int^{\infty}_{0}\tau^{-n/2-N-\delta'/2}\eta^{\alpha}_{1}(\tau)d\tau\\
&\lesssim& \Big(\frac{t^{1/2\alpha}}{\rho(x)}\Big)^{-N}\Big(\frac{t^{1/2\alpha}}{\rho(y)}\Big)^{-N}t^{-n/2\alpha}\Big(\frac{|h|}{t^{1/2\alpha}}\Big)^{\delta'}.
\end{eqnarray*}
We further divide $S_{2}$ into $S_{2}=S_{2,1}+S_{2,2}$, where
$$\left\{\begin{aligned}
S_{2,1}&:=\int_{|h|\geq\sqrt{s}}\eta^{\alpha}_{t}(s)\Big|K_{\alpha,s}^{L}(x,y)\Big|ds;\\
S_{2,2}&:=\int_{|h|\geq \sqrt{s}}\eta^{\alpha}_{t}(s)\Big|K_{\alpha,s}^{L}(x+h,y)\Big|ds.
\end{aligned}\right.$$
Noticing $|h|>\sqrt{s}$, for $\delta'>0$, it follows from Proposition \ref{prop-3.1} that
\begin{eqnarray*}
S_{2,1}&\lesssim&\int_{|h|\geq\sqrt{s}}s^{-n/2}\frac{1}{t^{1/\alpha}}\eta^{\alpha}_{1}(s/t^{1/\alpha})
\Big(1+\frac{\sqrt{s}}{\rho(x)}\Big)^{-N}\Big(1+\frac{\sqrt{s}}{\rho(y)}\Big)^{-N}ds\\
&\lesssim&\int_{0}^{\infty}s^{-n/2}\frac{1}{t^{1/\alpha}}\eta^{\alpha}_{1}(s/t^{1/\alpha})
\Big(\frac{|h|}{\sqrt{s}}\Big)^{\delta'}\Big(1+\frac{\sqrt{s}}{\rho(x)}\Big)^{-N}\Big(1+\frac{\sqrt{s}}{\rho(y)}\Big)^{-N}ds\\
&\lesssim& \int^{\infty}_{0}(t^{1/\alpha}\tau)^{-n/2}\eta^{\alpha}_{1}(\tau)\Big(\frac{|h|}{\sqrt{\tau}t^{1/2\alpha}}\Big)^{\delta'}
\Big(1+\frac{\sqrt{\tau}t^{1/2\alpha}}{\rho(x)}\Big)^{-N}\Big(1+\frac{\sqrt{\tau}t^{1/2\alpha}}{\rho(y)}\Big)^{-N}d\tau\\
&\lesssim& \Big(\frac{t^{1/2\alpha}}{\rho(x)}\Big)^{-N}\Big(\frac{t^{1/2\alpha}}{\rho(y)}\Big)^{-N}t^{-n/2\alpha}\Big(\frac{|h|}{t^{1/2\alpha}}\Big)^{\delta'}.
\end{eqnarray*}
For $S_{2,2}$, similarly, we use Proposition \ref{prop-3.1} again to deduce that
\begin{eqnarray*}
S_{2,2}&\lesssim&\int_{|h|\geq\sqrt{s}}s^{-n/2}\frac{1}{t^{1/\alpha}}\eta^{\alpha}_{1}(s/t^{1/\alpha})
\Big(1+\frac{\sqrt{s}}{\rho(x+h)}\Big)^{-2N}\Big(1+\frac{\sqrt{s}}{\rho(y)}\Big)^{-2N}ds\\
&\lesssim&\int_{0}^{\infty}s^{-n/2}\frac{1}{t^{1/\alpha}}\eta^{\alpha}_{1}(s/t^{1/\alpha})
\Big(\frac{|h|}{\sqrt{s}}\Big)^{\delta'}\Big(1+\frac{\sqrt{s}}{\rho(y)}\Big)^{-2N}ds\\
&\lesssim& \int^{\infty}_{0}(t^{1/\alpha}\tau)^{-n/2}\eta^{\alpha}_{1}(\tau)\Big(\frac{|h|}{\sqrt{\tau}t^{1/2\alpha}}\Big)^{\delta'}
\Big(1+\frac{\sqrt{\tau}t^{1/2\alpha}}{\rho(y)}\Big)^{-2N}d\tau\\
&\lesssim& \Big(\frac{t^{1/2\alpha}}{\rho(y)}\Big)^{-2N}t^{-n/2\alpha}\Big(\frac{|h|}{t^{1/2\alpha}}\Big)^{\delta'},
\end{eqnarray*}
which, together with the arbitrariness of $N$, indicates that
$$S_{2,2}\lesssim \Big(1+\frac{t^{1/2\alpha}}{\rho(y)}\Big)^{-2N}t^{-n/2\alpha}\Big(\frac{|h|}{t^{1/2\alpha}}\Big)^{\delta'}.$$
Because $|x-y|/2<t^{1/2\alpha}$, by Lemma \ref{lem2.4}, it holds 
$$m(y,V)\geq c\frac{m(x,V)}{(1+|x-y|m(x,V))^{k_{0}/(1+k_{0})}}\gtrsim \frac{m(x,V)}{(1+t^{1/2\alpha}m(x,V))^{k_{0}/(1+k_{0})}},$$
which gives
\begin{eqnarray*}
 \Big(1+\frac{t^{1/2\alpha}}{\rho(y)}\Big)^{-N}&=&{\Big(1+{t^{1/2\alpha}}{m(y,V)}\Big)^{-N}}\\
 &\lesssim&\frac{(1+t^{1/2\alpha}m(x,V))^{k_{0}N/(1+k_{0})}}{\Big((1+t^{1/2\alpha}m(x,V))^{k_{0}/(1+k_{0})}+t^{1/2\alpha}{{m(x,V)}}\Big)^{N}}\\
 &\lesssim&\Big(1+\frac{t^{1/2\alpha}}{\rho(y)}\Big)^{-N'}.
\end{eqnarray*}
The estimates for $S_{1}$ and $S_{2}$, together with $|x-y|/2<t^{1/2\alpha}$, imply that
\begin{eqnarray*}
\Big|K^{L}_{\alpha,t}(x+h,y)-K^{L}_{\alpha,t}(x,y)\Big|&\lesssim& t^{-n/2\alpha}\Big(\frac{|h|}{t^{1/2\alpha}}\Big)^{\delta'}\Big(1+\frac{t^{1/2\alpha}}{\rho(x)}\Big)^{-N}\Big(1+\frac{t^{1/2\alpha}}{\rho(y)}\Big)^{-N}\\
&\lesssim&\frac{(|h|/t^{1/2\alpha})^{\delta}t}{(t^{1/2\alpha}+|x-y|)^{n+2\alpha}}
\Big(1+\frac{t^{1/2\alpha}}{\rho(x)}+\frac{t^{1/2\alpha}}{\rho(x)}\Big)^{-N}.
\end{eqnarray*}
\end{proof}

For $m\in\mathbb{Z}^{+}$ and $t>0$, define
$\widetilde{D}_{\alpha,t}^{L,m}(\cdot,\cdot)=t^{m}\partial^{m}_{t}K^{L}_{\alpha,t}(\cdot,\cdot).$
We can obtain the following  estimates about the kernel $\widetilde{D}_{\alpha,t}^{L,m}(\cdot,\cdot)$.

\begin{proposition}\label{pro2.6} Let $\alpha\in(0,1)$, $m\in\mathbb Z_{+}$ and $\delta=\min\{2\alpha,\delta_{0}\}$, where $\delta_0$ appears in Proposition \ref{prop2}.
\item{\rm (i)} For any $N>0$, there exists a constant $C_{N}>0$ such that
$$|\widetilde{D}_{\alpha,t}^{L,m}(x,y)|\leq \frac{C_{N}t^{m}}{(t^{1/2\alpha}+|x-y|)^{n+2\alpha m}}\Big(1+\frac{{t^{1/2\alpha}}}{\rho(x)}+\frac{{t^{1/2\alpha}}}{\rho(y)}\Big)^{-N}.$$
\item{\rm (ii)} Let $0<\delta'\leq\delta$. For any $N>0$, there exists a constant $C_{N}>0$ such that for all $|h|\leq t^{1/2\alpha}$,
$$|\widetilde{D}_{\alpha,t}^{L,m}(x+h,y)-\widetilde{D}_{\alpha,t}^{L,m}(x,y)|\leq C_{N}\Big(\frac{|h|}{{t^{1/2\alpha}}}\Big)^{\delta'}\frac{t^{m}}{(t^{1/2\alpha}+|x-y|)^{n+2\alpha m}}\Big(1+\frac{{t^{1/2\alpha}}}{\rho(x)}+\frac{{t^{1/2\alpha}}}{\rho(y)}\Big)^{-N}.$$
\item{\rm (iii)} Let $0<\delta'\leq\delta$. For any $N>0$, there exists a constant $C_{N}>0$ such that
$$\Big|\int_{\mathbb{R}^{n}}\widetilde{D}_{\alpha,t}^{L,m}(x,y)dy\Big|\leq C_{N}\frac{({t^{1/2\alpha}}/\rho(x))^{\delta'}}{(1+{t^{1/2\alpha}}/\rho(x))^{N}}.$$
\end{proposition}
\begin{proof}
For (i), since $\eta^{\alpha}_{t}(s)=\frac{1}{t^{1/\alpha}}\eta_{1}^{\alpha}(s/t^{1/\alpha})$,
\begin{eqnarray*}
  K^{L}_{\alpha,t}(x,y) &=& \int^{\infty}_{0}\frac{1}{t^{1/\alpha}}\eta_{1}^{\alpha}(s/t^{1/\alpha})K^{L}_{s}(x,y)ds =  \int^{\infty}_{0}\eta^{\alpha}_{1}(\tau)K^{L}_{t^{1/\alpha}\tau}(x,y)d\tau.
\end{eqnarray*}
Hence
\begin{eqnarray*}
  \Big|\frac{\partial^{m}}{\partial t^{m}}K^{L}_{\alpha,t}(x,y)\Big| &=& \Big|\frac{\partial^{m}}{\partial t^{m}}\Big(\int^{\infty}_{0}\eta^{\alpha}_{1}(\tau)K^{L}_{t^{1/\alpha}\tau}(x,y)d\tau\Big)\Big|
   = \Big|\int^{\infty}_{0}\eta^{\alpha}_{1}(\tau)\frac{\partial^{m}}{\partial t^{m}}K^{L}_{t^{1/\alpha}\tau}(x,y)d\tau\Big|.
\end{eqnarray*}
By (i) of Proposition \ref{prop-2.1} and the higher-order derivative formula of the composite function, we can get
\begin{eqnarray*}
  \Big|\frac{\partial^{m}}{\partial t^{m}}K^{L}_{\alpha,t}(x,y)\Big| &\lesssim&\sum^{m}_{i=1}\Big|\int^{\infty}_{0}\eta^{\alpha}_{1}(\tau)t^{i/\alpha-m}\tau^{i}
  \frac{\partial^{i}}{\partial s^{i}}K^{L}_{s}(x,y)\Big|_{s=t^{1/\alpha}\tau}d\tau\Big|  \\
   &\leq&  C_{N}t^{-m}\int^{\infty}_{0}\eta^{\alpha}_{1}(\tau)(t^{1/\alpha}\tau)^{-n/2}e^{-c|x-y|^{2}/t^{1/\alpha}\tau}
   \Big(1+\frac{\sqrt{\tau}t^{1/2\alpha}}{\rho(x)}\Big)^{-N}\Big(1+\frac{\sqrt{\tau}t^{1/2\alpha}}{\rho(y)}\Big)^{-N}d\tau.
\end{eqnarray*}
Notice that $\eta_{1}^{\alpha}(\tau)\leq {C}/{\tau^{1+\alpha}}$. By  changing of variables, we obtain
\begin{eqnarray*}
\Big|\frac{\partial^{m}}{\partial t^{m}}K^{L}_{\alpha,t}(x,y)\Big|&\leq& \frac{C_{N}}{t^{m+n/2\alpha}}\Big(\frac{t^{1/2\alpha}}{\rho(x)}\Big)^{-N}\Big(\frac{t^{1/2\alpha}}{\rho(y)}\Big)^{-N}
\int^{\infty}_{0}\tau^{-1-\alpha-n/2-N}e^{-c|x-y|^{2}/t^{1/\alpha}\tau}d\tau\\
&\leq& \frac{C_{N}t^{1+N/\alpha-m}}{|x-y|^{2\alpha+n+2M}}\Big(\frac{t^{1/2\alpha}}{\rho(x)}\Big)^{-N}\Big(\frac{t^{1/2\alpha}}{\rho(y)}\Big)^{-N}.
\end{eqnarray*}
On the other hand,
\begin{eqnarray*}
  \Big|\frac{\partial^{m}}{\partial t^{m}}K^{L}_{\alpha,t}(x,y)\Big| &=& \Big|\int^{\infty}_{0}\eta^{\alpha}_{1}(\tau)\frac{\partial^{m}}{\partial t^{m}}K^{L}_{t^{1/\alpha}\tau}(x,y)d\tau\Big| \\
   &\leq& C_{N}t^{-m}\Big(\frac{t^{1/2\alpha}}{\rho(x)}\Big)^{-N}\Big(\frac{t^{1/2\alpha}}{\rho(y)}\Big)^{-N}\int^{\infty}_{0}\eta^{\alpha}_{1}(\tau)
   (\tau t^{1/\alpha})^{-n/2}\tau^{-N}d\tau \\
   &\leq& \frac{C_{N}}{t^{m+n/2\alpha}}\Big(\frac{t^{1/2\alpha}}{\rho(x)}\Big)^{-N}\Big(\frac{t^{1/2\alpha}}{\rho(y)}\Big)^{-N}.
\end{eqnarray*}
Finally, we have proved that for arbitrary $N>0$,
$$\Big(\frac{t^{1/2\alpha}}{\rho(x)}\Big)^{N}\Big(\frac{t^{1/2\alpha}}{\rho(y)}\Big)^{N}\Big|\frac{\partial^{m}}{\partial t^{m}}K^{L}_{\alpha,t}(x,y)\Big|\leq C_{N}
\min\Big\{\frac{t^{1+N/\alpha-m}}{|x-y|^{2\alpha+n+2N}},\ \frac{1}{t^{m+n/2\alpha}}\Big\},$$
which gives
$$|\widetilde{D}_{\alpha,t}^{L,m}(x,y)|\leq \frac{C_{N}t^{m}}{(t^{1/2\alpha}+|x-y|)^{n+2\alpha m}}\Big(1+\frac{{t^{1/2\alpha}}}{\rho(x)}+\frac{{t^{1/2\alpha}}}{\rho(y)}\Big)^{-N}.$$

For (ii), via the subordinative formula (\ref{eq-sub-for}), we can complete the proof by using (ii) of Proposition \ref{prop-2.1}.   We omit the details.

For (iii), it is easy to see that $e^{-sL^{\alpha}}(f)(x)=\int^{\infty}_{0}\eta^{\alpha}_{s}(\tau)e^{-\tau L}(f)(x)d\tau.$ Hence
\begin{eqnarray*}
  \frac{\partial^{m}}{\partial t^{m}}K^{L}_{\alpha,t}(x,y) = \frac{\partial^{m}}{\partial t^{m}}\Big(\int^{\infty}_{0}\eta^{\alpha}_{1}(\tau)K^{L}_{\tau t^{1/\alpha}}(x,y)d\tau\Big)
   = C_{m}\sum^{m}_{i=1}\int^{\infty}_{0}\eta^{\alpha}_{1}(\tau)Q^{L}_{\sqrt{t^{1/\alpha}\tau},i}(x,y)d\tau.
\end{eqnarray*}
It follows from (iii) of Proposition \ref{prop-2.1} that
\begin{eqnarray*}
  \Big|\int_{\mathbb{R}^{n}}\widetilde{D}_{\alpha,t}^{L,m}(x,y)dy\Big| &\lesssim&\int^{\infty}_{0}\eta^{\alpha}_{1}(\tau)\Big|\int_{\mathbb{R}^{n}}Q^{L}_{\sqrt{t^{1/\alpha}\tau},i}(x,y)dy\Big|d\tau  \\
   &\leq& C_{N}\int^{\infty}_{0}\eta^{\alpha}_{1}(\tau)\frac{(\sqrt{t^{1/\alpha}\tau}/\rho(x))^{\delta'}}{(1+\sqrt{t^{1/\alpha}\tau}/\rho(x))^{N}}d\tau.
\end{eqnarray*}
If $t^{1/2\alpha}>\rho(x)$, then
\begin{eqnarray*}
   \Big|\int_{\mathbb{R}^{n}}\widetilde{D}_{\alpha,t}^{L,m}(x,y)dy\Big|&\leq& C_{N}\rho(x)^{N-\delta'}t^{(\delta'-N)/2\alpha}\int^{\infty}_{0}\eta_{1}^{\alpha}(\tau)\tau^{(\delta'-N)/2}d\tau  \\
   &\leq&  \frac{C_{N}({t^{1/2\alpha}}/\rho(x))^{\delta'}}{(1+{t^{1/2\alpha}}/\rho(x))^{N}}.
\end{eqnarray*}
If $t^{1/2\alpha}\leq\rho(x)$, then
\begin{eqnarray*}
 \Big|\int_{\mathbb{R}^{n}}\widetilde{D}_{\alpha,t}^{L,m}(x,y)dy\Big|  \leq C_{N}\int^{\infty}_{0}\eta_{1}^{\alpha}(\tau)(\sqrt{t^{1/\alpha}\tau}/\rho(x))^{\delta'}d\tau
     \leq \frac{C_{N}({t^{1/2\alpha}}/\rho(x))^{\delta'}}{(1+{t^{1/2\alpha}}/\rho(x))^{N}}.
\end{eqnarray*}
\end{proof}

\subsection{Estimation on the spatial gradient}
In this section, we investigate the spatial gradient of $K^{L}_{\alpha, t}(\cdot,\cdot)$, $\alpha>0$. For the special case $\alpha=1/2$, i.e., the Poisson kernel, the regularity estimates have been obtained in \cite[Lemma 3.9]{DYZ}
\begin{lemma}\label{le3.1}
Suppose that $V\in B_{q}$ for some $q>n$.  For every $N>0$, there exist constants $C_{N}>0$  and $c>0$ such that
for all $x, y\in\mathbb R^{n}$ and $t>0$, the kernels $K_{t}^{L}(\cdot,\cdot)$  satisfy the following estimates:
$$|\nabla_{x}K^{L}_{t}(x,y)|\leq
\left\{\begin{aligned}
\frac{C_{N}}{t^{(n+1)/2}}e^{-c|x-y|^{2}/t}\Big(1+\frac{\sqrt{t}}{\rho(x)}+\frac{\sqrt{t}}{\rho(y)}\Big)^{-N},\ \sqrt{t}\leq |x-y|;\\
\frac{C_{N}}{|x-y|t^{n/2}}e^{-c|x-y|^{2}/t}\Big(1+\frac{\sqrt{t}}{\rho(x)}+\frac{\sqrt{t}}{\rho(y)}\Big)^{-N},\ \sqrt{t}> |x-y|.
\end{aligned}\right.$$
\end{lemma}

\begin{proof}
Let $\Gamma_{0}(\cdot,\cdot)$ be the fundamental solution of $-\Delta$ in $\mathbb R^{n}$, i.e.,
$$\Gamma_{0}(x,y)=-\frac{1}{n(n-2)\omega(n)}\frac{1}{|x-y|^{n-2}},\ n\geq 3,$$
where $\omega(n)$ denotes  the area of the unit sphere in $\mathbb R^{n}$. Fix $t>0$ and $x_{0}, y_{0}\in\mathbb{R}^{n}$. Assume that $u(\cdot,\cdot)$ is a solution to the equation
$$\partial_{t}u+L u=\partial_{t}u+(-\Delta)u+Vu=0.$$
Let $\eta\in C^{\infty}_{0}(B(x_{0}, 2R))$ with some $R>0$ such that
$\eta=1\text{ on }B(x_{0}, 3R/2)$,
$|\nabla\eta|\leq C/R$ and $|\nabla^{2}\eta|\leq C/ R^{2}.$
Noticing that $\partial_{t} u+Lu=0$, we can obtain
\begin{eqnarray*}
-\Delta(u\eta)&=&-\sum^{n}_{i=1}\Big(\frac{\partial^{2} u}{\partial x_{i}^{2}}\cdot\eta+2\frac{\partial u}{\partial x_{i}}\frac{\partial \eta}{\partial x_{i}}
+u\cdot\frac{\partial^{2}\eta}{\partial x_{i}^{2}}\Big)\\
&=&-\Delta u\cdot\eta-2\nabla u\cdot\nabla\eta-u\cdot\Delta\eta\\
&=&-(\partial_{t}u)\eta-Vu\eta-2\nabla u\nabla\eta-u\nabla\eta,
\end{eqnarray*}
which, together with integration by parts, gives
\begin{eqnarray*}
&&-\int_{\mathbb R^{n}}\Gamma_{0}(x,y)\nabla u(y,t)\cdot\nabla\eta(y) dy
=-\sum^{n}_{i=1}\int_{\mathbb{R}^{n}}\Gamma_{0}(x,y)
\frac{\partial u(y,t)}{\partial y_{i}}\frac{\partial \eta(y)}{\partial y_{i}}dy\\
&&\quad=\sum^{n}_{i=1}\int_{\mathbb{R}^{n}}\Big(\frac{\partial}{\partial y_{i}}\Gamma_{0}(x,y)\Big)\frac{\partial \eta(y)}{\partial x_{i}}
u(y,t)dy+\sum^{n}_{i=1}\int_{\mathbb{R}^{n}}\Gamma_{0}(x,y)\Big(\frac{\partial^{2}\eta(y)}{\partial x_{i}^{2}}\Big)
u(y,t)dy\\
&&\quad=\int_{\mathbb R^{n}}\nabla_{y}\Gamma_{0}(x,y)\cdot\nabla \eta(y)u(y,t)dy+\int_{\mathbb R^{n}}\Gamma_{0}(x,y)\cdot\Delta \eta(y)u(y,t)dy.
\end{eqnarray*}
Then we can get
\begin{eqnarray*}
u(x,t)\eta(x)&=&\int_{\mathbb{R}^{n}}\Gamma_{0}(x,y)\Big\{-V(y)u(y,t)\eta(y)-\eta(y)\partial_{t}u(y,t)-2\nabla u(y,t)\cdot\nabla\eta(y)-u(y,t)\Delta\eta(y)\Big\}dy\\
&=&\int_{\mathbb R^{n}}\Gamma_{0}(x,y)\Big\{-V(y)u(y,t)\eta(y)-\eta(y)\partial_{t}u(y,t)+\Delta\eta(y)\cdot u(y,t)\Big\}dy\\
&&+2\int_{\mathbb{R}^{n}}\nabla_{y}\Gamma_{0}(x,y)\nabla\eta(y)u(y,t)dy.
\end{eqnarray*}
Notice that it follows from Lemma \ref{le-2.6} that
$$\int_{B(x_{0}, 2R)}\frac{V(y)}{|x-y|^{n-1}}dy\leq \frac{C}{R^{n-1}}\int_{B(x_{0}, 2R)}V(y)dy
\leq\frac{C}{R}\Big(1+\frac{R}{\rho(x_{0})}\Big)^{m_{0}},\ m_{0}>1.$$
Thus for $x\in B(x_{0}, R)$, it holds
\begin{eqnarray*}
|\nabla_{x} u(x,t)|&=&|\nabla_{x}(u(x,t)\eta(x))|\\
&\leq&\int_{B(x_{0},2R)}\frac{V(y)|u(y,t)||\eta(y)|}{|x-y|^{n-1}}dy+\int_{B(x_{0}, 2R)}\frac{|\partial_{t}u(y,t)||\eta(y)|}{|x-y|^{n-1}}dy
+\frac{C}{R^{n+1}}\int_{B(x_{0}, 2R)}|u(y,t)|dy\\
&\leq&\frac{C}{R}\sup_{B(x_{0}, 2R)}|u(y,t)|\Big\{\Big(1+\frac{R}{\rho(x_{0})}\Big)^{m_{0}}+1\Big\}+CR\sup_{B(x_{0}, R)}|\partial_{t}u(y,t)|.
\end{eqnarray*}

Take $u(x,t)=K^{L}_{t}(x, y_{0})$ and $R<\min\{|x_{0}-y_{0}|/8,\ \rho(x_{0})\}$. We obtain that
$$|\nabla_{x}K^{L}_{t}(x_{0}, y_{0})|\leq\frac{C}{R}\sup_{B(x_{0}, 2R)}|K^{L}_{t}(x, y_{0})|\Big\{\Big(1+\frac{R}{\rho(x_{0})}\Big)^{m_{0}}+1\Big\}
+R\sup_{B(x_{0}, 2R)}|\partial_{t}K_{t}^{L}(x, y_{0})|.$$

If $x\in B(x_{0}, 2R)$, then $|x-y_{0}|\sim |x_{0}-y_{0}|$. Also $\rho(x)\sim \rho(x_{0})$ for $|x-x_{0}|<2R<2\rho(x_{0})$.
It follows from Propositions \ref{lem2.1}\  and \ \ref{prop-2.1} that for any $N>0$ there exists a constant $C_{N}$ such that
\begin{equation}\label{eq-3.8}
\left\{\begin{aligned}
&\sup_{x\in B(x_{0}, 2R)}|K^{L}_{t}(x, y_{0})|\leq\frac{C_{N}}{t^{n/2}}e^{-c|x_{0}-y_{0}|^{2}/t}\Big(1+\frac{\sqrt{t}}{\rho(x_{0})}+\frac{\sqrt{t}}{\rho(y_{0})}\Big)^{-N};\\
&\sup_{x\in B(x_{0}, 2R)}|t\partial_{t}K^{L}_{t}(x, y_{0})|\leq\frac{C_{N}}{t^{n/2}}e^{-c|x_{0}-y_{0}|^{2}/t}\Big(1+\frac{\sqrt{t}}{\rho(x_{0})}+\frac{\sqrt{t}}{\rho(y_{0})}\Big)^{-N}.
\end{aligned}\right.
\end{equation}
Finally, it can be deduced from (\ref{eq-3.8}) that
\begin{eqnarray*}
|\nabla_{x}K^{L}_{t}(x_{0}, y_{0})|&\lesssim&\frac{C_{N}}{R}t^{-n/2}e^{-c|x_{0}-y_{0}|^{2}/t}
\Big(1+\frac{\sqrt{t}}{\rho(x_{0})}+\frac{\sqrt{t}}{\rho(y_{0})}\Big)^{-N}\Big\{1+\frac{R^{2}}{t}\Big\}.
\end{eqnarray*}

The rest of the proof is divided into three cases.

{\it Case 1: $R>\sqrt{t}$}. For this case, $\sqrt{t}<R<\min\{|x_{0}-y_{0}|/8,\ \rho(x_{0})\}$.
We split $$|\nabla_{x}K^{L}_{t}(x_{0}, y_{0})|\leq C_{N}(M_{1}+M_{2}),$$ where
$$\left\{\begin{aligned}
&M_{1}:=\frac{\sqrt{t}}{R}\frac{1}{t^{(n+1)/2}}e^{-c|x_{0}-y_{0}|^{2}/t}
\Big(1+\frac{\sqrt{t}}{\rho(x_{0})}+\frac{\sqrt{t}}{\rho(y_{0})}\Big)^{-N};\\
&M_{2}:=\frac{\sqrt{t}}{R}\frac{1}{t^{(n+1)/2}}e^{-c|x_{0}-y_{0}|^{2}/t}
\Big(1+\frac{\sqrt{t}}{\rho(x_{0})}+\frac{\sqrt{t}}{\rho(y_{0})}\Big)^{-N}\frac{R^{2}}{t}.
\end{aligned}\right.$$
It is obvious that
$$M_{1}\lesssim \frac{1}{t^{(n+1)/2}}e^{-c|x_{0}-y_{0}|^{2}/t}
\Big(1+\frac{\sqrt{t}}{\rho(x_{0})}+\frac{\sqrt{t}}{\rho(y_{0})}\Big)^{-N}.$$
Similarly, for the term $M_{3}$, we can also get
\begin{eqnarray*}
M_{2}&\lesssim&\frac{1}{t^{(n+1)/2}}e^{-c|x_{0}-y_{0}|^{2}/t}
\Big(1+\frac{\sqrt{t}}{\rho(x_{0})}+\frac{\sqrt{t}}{\rho(y_{0})}\Big)^{-N}\frac{R^{2}}{t}\\
&\lesssim&\frac{1}{t^{(n+1)/2}}e^{-c|x_{0}-y_{0}|^{2}/t}
\Big(1+\frac{\sqrt{t}}{\rho(x_{0})}+\frac{\sqrt{t}}{\rho(y_{0})}\Big)^{-N}\frac{|x_{0}-y_{0}|^{2}}{t}\\
&\lesssim&\frac{1}{t^{(n+1)/2}}e^{-c|x_{0}-y_{0}|^{2}/t}
\Big(1+\frac{\sqrt{t}}{\rho(x_{0})}+\frac{\sqrt{t}}{\rho(y_{0})}\Big)^{-N}.
\end{eqnarray*}

{\it Case 2: $0<R\leq \sqrt{t}<\min\{|x_{0}-y_{0}|/8,\ \rho(x_{0})\}$}. We write
$$|\nabla_{x}K^{L}_{t}(x_{0}, y_{0})|\leq \frac{C_{N}}{t^{(n+1)/2}}e^{-c|x_{0}-y_{0}|^{2}/t}
\Big(1+\frac{\sqrt{t}}{\rho(x_{0})}+\frac{\sqrt{t}}{\rho(y_{0})}\Big)^{-N}
\Big\{\frac{\sqrt{t}}{R}+\frac{R}{\sqrt{t}}\Big\}.$$
Because $R<\sqrt{t}<\min\{|x_{0}-y_{0}|/8,\ \rho(x_{0})\}$, taking the infimum for $R$ yields
$$|\nabla_{x}K^{L}_{t}(x_{0}, y_{0})|\leq \frac{C_{N}}{t^{(n+1)/2}}e^{-c|x_{0}-y_{0}|^{2}/t}
\Big(1+\frac{\sqrt{t}}{\rho(x_{0})}+\frac{\sqrt{t}}{\rho(y_{0})}\Big)^{-N}.$$

{\it Case 3: $0<R<\min\{|x_{0}-y_{0}|/8,\ \rho(x_{0})\}< \sqrt{t}$.} Similarly, we can see that

\begin{eqnarray*}
|\nabla_{x}K^{L}_{t}(x_{0}, y_{0})|&\leq&
\Big(1+\frac{\sqrt{t}}{\rho(x_{0})}+\frac{\sqrt{t}}{\rho(y_{0})}\Big)^{-N}\Big(\frac{\sqrt{t}}{R}+\frac{R}{\sqrt{t}}\Big).
\end{eqnarray*}
Since $0<R<\min\{|x_{0}-y_{0}|/8,\ \rho(x_{0})\}< \sqrt{t}$, the function $\sqrt{t}/R+R/\sqrt{t}$ is decreasing and with the infimum at
$R=\min\{|x_{0}-y_{0}|/8,\ \rho(x_{0})\}$. Then
\begin{eqnarray}\label{eq-3.1}
|\nabla_{x}K^{L}_{t}(x_{0}, y_{0})|&\leq&\frac{C_{N}}{t^{(n+1)/2}}e^{-c|x_{0}-y_{0}|^{2}/t}
\Big(1+\frac{\sqrt{t}}{\rho(x_{0})}+\frac{\sqrt{t}}{\rho(y_{0})}\Big)^{-N}\\
&&\times\Bigg\{\frac{\sqrt{t}}{\min\{|x_{0}-y_{0}|/8,\ \rho(x_{0})\}}
+\frac{\min\{|x_{0}-y_{0}|/8,\ \rho(x_{0})\}}{\sqrt{t}}\Bigg\}.\nonumber
\end{eqnarray}

{\it Case 3.1: $\rho(x_{0})\leq |x_{0}-y_{0}|/8$.} Since $N$ is  arbitrary, we can deduce from (\ref{eq-3.1}) that
\begin{eqnarray*}
|\nabla_{x}K^{L}_{t}(x_{0}, y_{0})|&\leq&\frac{C_{N}}{t^{(n+1)/2}}e^{-c|x_{0}-y_{0}|^{2}/t}
\Big(1+\frac{\sqrt{t}}{\rho(x_{0})}+\frac{\sqrt{t}}{\rho(y_{0})}\Big)^{-N}\Big(\frac{\sqrt{t}}{\rho(x_{0})}
+\frac{\rho(x_{0})}{\sqrt{t}}\Big)\\
&\leq& \frac{C_{N}}{t^{(n+1)/2}}e^{-c|x_{0}-y_{0}|^{2}/t}
\Big(1+\frac{\sqrt{t}}{\rho(x_{0})}+\frac{\sqrt{t}}{\rho(y_{0})}\Big)^{-N}.
\end{eqnarray*}

{\it Case 3.2: $\rho(x_{0})> |x_{0}-y_{0}|/8$.} For this case, by (\ref{eq-3.1}) again, it holds
\begin{eqnarray*}
|\nabla_{x}K^{L}_{t}(x_{0}, y_{0})|&\leq&\frac{C_{N}}{t^{(n+1)/2}}e^{-c|x_{0}-y_{0}|^{2}/t}
\Big(1+\frac{\sqrt{t}}{\rho(x_{0})}+\frac{\sqrt{t}}{\rho(y_{0})}\Big)^{-N}\Big(\frac{\sqrt{t}}{|x_{0}-y_{0}|}
+\frac{|x_{0}-y_{0}|}{\sqrt{t}}\Big)\\
&\leq&\frac{C_{N}}{t^{(n+1)/2}}e^{-c|x_{0}-y_{0}|^{2}/t}
\Big(1+\frac{\sqrt{t}}{\rho(x_{0})}+\frac{\sqrt{t}}{\rho(y_{0})}\Big)^{-N}\\
&&+\frac{C_{N}}{|x_{0}-y_{0}|t^{n/2}}e^{-c|x_{0}-y_{0}|^{2}/t}
\Big(1+\frac{\sqrt{t}}{\rho(x_{0})}+\frac{\sqrt{t}}{\rho(y_{0})}\Big)^{-N}.
\end{eqnarray*}

Finally, we obtain the following estimates:
\begin{eqnarray*}
&&|\nabla_{x}K^{L}_{t}(x_{0}, y_{0})|\\
&&\leq\left\{
\begin{aligned}
&\frac{C_{N}}{t^{(n+1)/2}}e^{-c|x_{0}-y_{0}|^{2}/t}
\Big(1+\frac{\sqrt{t}}{\rho(x_{0})}+\frac{\sqrt{t}}{\rho(y_{0})}\Big)^{-N},\  \sqrt{t}\leq\min\Big\{|x_{0}-y_{0}|/8,\ \rho(x_{0})\Big\};\\
& \frac{C_{N}}{t^{n/2}}e^{-c|x_{0}-y_{0}|^{2}/t}(\frac{1}{\sqrt{t}}+\frac{1}{|x_{0}-y_{0}|})
\Big(1+\frac{\sqrt{t}}{\rho(x_{0})}+\frac{\sqrt{t}}{\rho(y_{0})}\Big)^{-N},\ \sqrt{t}>\min\Big\{|x_{0}-y_{0}|/8,\ \rho(x_{0})\Big\}.
\end{aligned}\nonumber
\right.
\end{eqnarray*}
Then if $\sqrt{t}\geq |x_{0}-y_{0}|/8$,
$$|\nabla_{x}K^{L}_{t}(x_{0}, y_{0})|\leq \frac{C_{N}}{|x_{0}-y_{0}|t^{n/2}}e^{-c|x_{0}-y_{0}|^{2}/t}
\Big(1+\frac{\sqrt{t}}{\rho(x_{0})}+\frac{\sqrt{t}}{\rho(y_{0})}\Big)^{-N}.$$
This proves Lemma \ref{le3.1}.
\end{proof}

\begin{lemma}\label{le-3.2}
Suppose that $V\in B_{q}$ for some $q>n$. For every $N>0$, there exists a constant $C_{N}>0$ such that
for all $x, y\in\mathbb R^{n}$ and $t>0$, the semigroup kernels $K^{L}_{t}(\cdot,\cdot)$ satisfy the following estimate:
$$|\nabla_{x}K^{L}_{t}(x,y)|\leq \frac{C_{N}}{t^{(n+1)/2}}\Big(1+\frac{\sqrt{t}}{\rho(x)}+\frac{\sqrt{t}}{\rho(x)}\Big)^{-N}.$$
\end{lemma}

\begin{proof}
Assume that $u(\cdot,\cdot)$ is a solution of the
equation
$$\partial_{t}u+L u=\partial_{t}u+(-\Delta)u+Vu=0.$$
Similar to Lemma \ref{le3.1}, we can prove that
$$|\nabla_{x}u(x,t)|\leq\frac{C}{R}\sup_{B(x_{0},2R)}|u(y,t)|\Big\{\Big(1+\frac{R}{\rho(x_{0})}\Big)^{m_{0}}+1\Big\}+CR\sup_{B(x_{0}, R)}|\partial_{t}u(y,t)|.$$
Take $u(x,t)=K^{L}_{t}(x, y_{0})$ for fixed $y_{0}$, and let $R\in (0, \min\{\rho(x_{0}), \sqrt{t}\})$. It can be deduced from Propositions \ref{lem2.1}\ and \ \ref{prop-2.1}  that
$$\sup\limits_{x\in B(x_{0}, 2R)}\Big\{|K^{L}_{t}(x, y_{0})|+|t\partial_{t}K^{L}_{t}(x, y_{0})|\Big\}\leq \frac{C_{N}}{t^{n/2}}\Big(1+\frac{\sqrt{t}}{\rho(x_{0})}+\frac{\sqrt{t}}{\rho(y_{0})}\Big)^{-N}.$$
This, together with $R<\rho(x_{0})$, implies that
\begin{eqnarray}\label{eq-3.2}
|\nabla_{x}K^{L}_{t}(x_{0}, y_{0})|&\leq&\frac{C_{N}}{R}\frac{1}{t^{n/2}}\Big(1+\frac{\sqrt{t}}{\rho(x_{0})}+\frac{\sqrt{t}}{\rho(y_{0})}\Big)^{-N}
\Big\{1+\frac{R^{2}}{t}\Big\}\\
&\leq&\frac{C_{N}}{t^{(n+1)/2}}\Big(1+\frac{\sqrt{t}}{\rho(x_{0})}+\frac{\sqrt{t}}{\rho(y_{0})}\Big)^{-N}
\Big(\frac{\sqrt{t}}{R}+\frac{R}{\sqrt{t}}\Big).\nonumber
\end{eqnarray}

If $\sqrt{t}\leq\rho(x_{0})$, taking the infimum for $R$ on both sides of (\ref{eq-3.2}) reaches
\begin{eqnarray*}
|\nabla_{x}K^{L}_{t}(x_{0}, y_{0})|&\leq&\frac{C_{N}}{t^{(n+1)/2}}\Big(1+\frac{\sqrt{t}}{\rho(x_{0})}+\frac{\sqrt{t}}{\rho(y_{0})}\Big)^{-N}.
\end{eqnarray*}

If $\sqrt{t}>\rho(x_{0})$,  note that the function $h(t):=R/\sqrt{t}+\sqrt{t}/R$ is decreasing on $R\in (0, \rho(x_{0}))$. Taking the infimum again, we get
\begin{eqnarray*}
|\nabla_{x}K^{L}_{t}(x_{0}, y_{0})|&\leq&\frac{C_{N}}{t^{(n+1)/2}}\Big(1+\frac{\sqrt{t}}{\rho(x_{0})}+\frac{\sqrt{t}}{\rho(y_{0})}\Big)^{-N}
\Big\{\frac{\sqrt{t}}{\rho(x_{0})}+\frac{\rho(x_{0})}{\sqrt{t}}\Big\}\\
&\leq&\frac{C_{N}}{t^{(n+1)/2}}\Big(1+\frac{\sqrt{t}}{\rho(x_{0})}+\frac{\sqrt{t}}{\rho(y_{0})}\Big)^{-N}.
\end{eqnarray*}
This completes the proof of Lemma \ref{le-3.2}.
\end{proof}

Now we give the gradient estimate of $K^{L}_{\alpha,t}(\cdot,\cdot)$.
\begin{proposition}\label{prop-3.3-1}
Suppose $\alpha>0$ and $V\in B_{q}$ for some $q>n$. For every $N>0$, there exists a constant $C_{N}>0$  such that for all $x, y\in\mathbb R^{n}$ and $t>0$,
$$|t^{1/2\alpha}\nabla_{x}K^{L}_{\alpha,t}(x,y)|\leq \frac{C_{N}t}{(t^{1/2\alpha}+|x-y|)^{n+2\alpha}}\Big(1+\frac{t^{1/2\alpha}}{\rho(x)}+\frac{t^{1/2\alpha}}{\rho(y)}\Big)^{-N}.$$
\end{proposition}

\begin{proof}
The subordinate formula gives
\begin{eqnarray*}
\nabla_{x}K^{L}_{\alpha,t}(x,y)&=&\int^{\infty}_{0}\eta^{\alpha}_{t}(s)\nabla_{x}K^{L}_{s}(x,y)ds,
\end{eqnarray*}
which, together with Lemma \ref{le3.1}, implies that
$\big|\nabla_{x}K^{L}_{\alpha,t}(x,y)\big|\leq C_{N}(L_{1}+L_{2}),$
where
$$\left\{\begin{aligned}
& L_{1}:=\int^{|x-y|^{2}}_{0}\eta^{\alpha}_{t}(s)\frac{1}{s^{(n+1)/2}}e^{-c|x-y|^{2}/s}
\Big(1+\frac{\sqrt{s}}{\rho(x)}+\frac{\sqrt{s}}{\rho(y)}\Big)^{-N}ds;\\
&L_{2}:=\int^{\infty}_{|x-y|^{2}}\eta^{\alpha}_{t}(s)\frac{1}{s^{n/2}|x-y|}e^{-c|x-y|^{2}/s}
\Big(1+\frac{\sqrt{s}}{\rho(x)}+\frac{\sqrt{s}}{\rho(y)}\Big)^{-N}ds.
\end{aligned}\right.$$

For $L_{1}$, letting $s=t^{1/\alpha}u$, we can get
\begin{eqnarray*}
L_{1}&\leq& \int^{\infty}_{0}\frac{t}{(t^{1/\alpha}u)^{1+\alpha}}(t^{1/\alpha}u)^{-(n+1)/2}e^{-\frac{c|x-y|^{2}}{t^{1/\alpha}u}}
\Big(1+\frac{t^{1/2\alpha}\sqrt{u}}{\rho(x)}\Big)^{-N}\Big(1+\frac{t^{1/2\alpha}\sqrt{u}}{\rho(y)}\Big)^{-N}t^{1/\alpha}du\\
&\leq& \Big(\frac{t^{1/2\alpha}}{\rho(x)}\Big)^{-N}\Big(\frac{t^{1/2\alpha}}{\rho(y)}\Big)^{-N}t^{-(n+1)/2\alpha}\int^{\infty}_{0}u^{-(n+1)/2-(1+\alpha+N)}
e^{-\frac{c|x-y|^{2}}{t^{1/\alpha}u}}du\\
&\leq& \Big(\frac{t^{1/2\alpha}}{\rho(x)}\Big)^{-N}\Big(\frac{t^{1/2\alpha}}{\rho(y)}\Big)^{-N}\frac{t^{1+N/\alpha}}{|x-y|^{2\alpha+2N+n+1}}.
\end{eqnarray*}

Similarly, for the term $L_{2}$, a change of variables yields
\begin{eqnarray*}
L_{2}&\leq&\frac{1}{|x-y|}\int^{\infty}_{0}\frac{t}{(t^{1/\alpha}u)^{1+\alpha}}(t^{1/\alpha}u)^{-n/2}e^{-\frac{c|x-y|^{2}}{t^{1/\alpha}u}}
\Big(1+\frac{t^{1/2\alpha}\sqrt{u}}{\rho(x)}\Big)^{-N}\Big(1+\frac{t^{1/2\alpha}\sqrt{u}}{\rho(y)}\Big)^{-N}t^{1/\alpha}du\\
&\leq&\frac{1}{t^{n/2\alpha}|x-y|}\Big(\frac{t^{1/2\alpha}}{\rho(x)}\Big)^{-N}\Big(\frac{t^{1/2\alpha}}{\rho(y)}\Big)^{-N}
\int^{\infty}_{0}u^{-n/2-(1+\alpha+N)}e^{-\frac{c|x-y|^{2}}{t^{1/\alpha}u}}du\\
&\leq&\frac{1}{|x-y|t^{n/2\alpha}}\Big(\frac{t^{1/2\alpha}}{\rho(x)}\Big)^{-N}\Big(\frac{t^{1/2\alpha}}{\rho(y)}\Big)^{-N}
\int^{\infty}_{0}\Big(\frac{t^{1/\alpha}r^{2}}{|x-y|^{2}}\Big)^{1+\alpha+N+n/2}e^{-cr^{2}}\frac{|x-y|^{2}}{t^{1/\alpha}r^{3}}dr\\
&\leq&\Big(\frac{t^{1/2\alpha}}{\rho(x)}\Big)^{-N}\Big(\frac{t^{1/2\alpha}}{\rho(y)}\Big)^{-N}\frac{t^{1+N/\alpha}}{|x-y|^{2\alpha+2N+n+1}}.
\end{eqnarray*}
The estimates for $L_{1}$ and $L_{2}$ indicate that
$$\Big(\frac{t^{1/2\alpha}}{\rho(x)}\Big)^{N}\Big(\frac{t^{1/2\alpha}}{\rho(y)}\Big)^{N}|\nabla K^{L}_{\alpha,t}(x,y)|
\leq  \frac{C_{N}t^{1+N/\alpha}}{|x-y|^{2\alpha+2N+n+1}}.$$

On the other hand, by Lemma \ref{le-3.2} and  changing variable $\tau=s/t^{1/\alpha}$, we obtain
\begin{eqnarray*}
|\nabla_{x}K^{L}_{\alpha,t}(x,y)|&\leq& C_{N}\int^{\infty}_{0}s^{-(n+1)/2}\frac{1}{t^{1/\alpha}}\eta^{\alpha}_{1}(\frac{s}{t^{1/\alpha}})
\Big(1+\frac{\sqrt{s}}{\rho(x)}\Big)^{-N}\Big(1+\frac{\sqrt{s}}{\rho(y)}\Big)^{-N}ds\\
&\leq&C_{N}\int^{\infty}_{0}(t^{1/\alpha}\tau)^{-(n+1)/2}\eta^{\alpha}_{1}(\tau)
\Big(1+\frac{t^{1/2\alpha}\sqrt{\tau}}{\rho(x)}\Big)^{-N}\Big(1+\frac{t^{1/2\alpha}\sqrt{\tau}}{\rho(y)}\Big)^{-N}d\tau\\
&\leq&\frac{C_{N}}{t^{(n+1)/2\alpha}}\Big(\frac{t^{1/2\alpha}}{\rho(x)}\Big)^{-N}\Big(\frac{t^{1/2\alpha}}{\rho(y)}\Big)^{-N}.
\end{eqnarray*}

Finally, we obtain
$$\Big(\frac{t^{1/2\alpha}}{\rho(x)}\Big)^{N}\Big(\frac{t^{1/2\alpha}}{\rho(y)}\Big)^{N}|\nabla_{x}K^{L}_{\alpha,t}(x,y)|\leq
C_{N} \min\Big\{t^{-(n+1)/2\alpha},\ \frac{t^{1+N/\alpha}}{|x-y|^{n+1+2N+2\alpha}}\Big\}.$$
The arbitrariness of $N$ indicates that
$$|\nabla_{x}K^{L}_{\alpha,t}(x,y)|\leq
\frac{C_{N}}{t^{1/2\alpha}}\frac{t}{(t^{1/2\alpha}+|x-y|)^{n+2\alpha}}\Big(1+\frac{t^{1/2\alpha}}{\rho(x)}\Big)^{-N}
\Big(1+\frac{t^{1/2\alpha}}{\rho(y)}\Big)^{-N}.$$
\end{proof}

Below we estimate the Lipschitz continuity of $|\nabla_{x}K^{L}_{t}(\cdot,\cdot)|$.

\begin{lemma}\label{le-3.3}
Suppose that $\alpha>0$ and $V\in B_{q}$ for some $q>n$. Let $\delta'=1-n/q$. For every $N>0$, there exist  constants $C_{N}>0$ and $c>0$ such that for all $x,y\in\mathbb{R}^{n}$,
$t>0$ and  $|h|<|x-y|/4$,
\begin{eqnarray*}
&&|\nabla_{x}K^{L}_{t}(x+h,y)-\nabla_{x}K^{L}_{t}(x,y)|\\
&&\quad \leq\left\{\begin{aligned}
&\frac{C_{N}}{t^{(n+1)/2}}\Big(\frac{|h|}{\sqrt{t}}\Big)^{\delta'}e^{-c|x-y|^{2}/t}\Big(1+\frac{\sqrt{t}}{\rho(x)}+\frac{\sqrt{t}}{\rho(y)}\Big)^{-N},\
\sqrt{t}\leq |x-y|;\\
&\frac{C_{N}}{t^{n/2}|x-y|}\Big(\frac{|h|}{|x-y|}\Big)^{\delta'}e^{-c|x-y|^{2}/t}\Big(1+\frac{\sqrt{t}}{\rho(x)}+\frac{\sqrt{t}}{\rho(y)}\Big)^{-N},\
\sqrt{t}\geq |x-y|.
\end{aligned}\right.
\end{eqnarray*}
\end{lemma}
\begin{proof}
The proof  is similar to that of Lemma \ref{le3.1}. Let $\Gamma_{0}(\cdot,\cdot)$ be the fundamental solution of $-\Delta$ in $\mathbb R^{n}$.
 Assume that $\partial_{t}u+(-\Delta)u+Vu=0$. Let
$\eta\in C^{\infty}_{0}(B(x_{0}, 2R))$ such that $\eta=1$ on $B(x_{0}, 3R/2)$, $|\nabla\eta|\leq C/R$ and $|\nabla^{2}\eta|\leq C/R^{2}$. It is easy to see that
$$(-\Delta)(u\eta)=(-\Delta u)\eta-2\nabla u\cdot\nabla\eta+u\cdot(-\Delta\eta).$$
Similar to the proof of Lemma \ref{le3.1}, an integration by parts implies that
\begin{eqnarray*}
-\int_{\mathbb{R}^{n}}\Gamma_{0}(x,y)\nabla u(y,t)\cdot\nabla\eta(y) dy=\int_{\mathbb{R}^{n}}\nabla_{y}\Gamma_{0}(x,y)\nabla\eta(y)u(y,t)dy+\int_{\mathbb{R}^{n}}\Gamma_{0}(x,y)\Delta\eta(y)u(y,t)dy,
\end{eqnarray*}
which yields
\begin{eqnarray*}
u(x,t)\eta(x)
&=&\int_{\mathbb{R}^{n}}\Gamma_{0}(x,y)\Big\{(-\partial_{t} u(y,t))\eta(y)-V(y)u(y,t)\eta(y)+u(y,t)\Delta\eta(y)\Big\}dy\\
&&+ 2\int_{\mathbb{R}^{n}}\nabla_{y}\Gamma_{0}(x,y) u(y,t)\cdot\nabla\eta(y) dy.
\end{eqnarray*}

Then for $x\in B(x_{0}, R)$, $u(x,t)=u(x,t)\eta(x)$ and
\begin{eqnarray*}
\nabla_{x}^{2}u(x,t)&=&\nabla^{2}_{x}\Big\{\int_{\mathbb{R}^{n}}\Gamma_{0}(x,y)\Big\{(-\partial_{t} u(y,t))\eta(y)-V(y)u(y,t)\eta(y)+u(y,t)\Delta\eta(y)\Big\}dy\\
&&+2\int_{\mathbb{R}^{n}}\nabla_{y}\Gamma_{0}(x,y) u(y,t)\cdot\nabla\eta(y) dy\Big\},
\end{eqnarray*}
which gives $\|\nabla^{2}_{x}u(\cdot,t)\|_{q}\leq\sum\limits^{4}_{i=1}\|S_{i}(\cdot,t)\|_{q}$, where
$$\left\{\begin{aligned}
S_{1}(x,t)&:=\int_{\mathbb{R}^{n}}|\nabla^{2}_{x}\Gamma_{0}(x,y)|\cdot|\eta(y)|\cdot|\partial_{t}u(y,t)|dy;\\
S_{2}(x,t)&:=\int_{\mathbb{R}^{n}}|\nabla^{2}_{x}\Gamma_{0}(x,y)|\cdot|\Delta\eta(y)|\cdot|u(y,t)|dy;\\
S_{3}(x,t)&:=\int_{\mathbb{R}^{n}}|\nabla^{2}_{x}\nabla_{y}\Gamma_{0}(x,y)|\cdot|\nabla\eta(y)|\cdot|u(y,t)|dy;\\
S_{4}(x,t)&:=\int_{\mathbb{R}^{n}}|\nabla^{2}_{x}\Gamma_{0}(x,y)|\cdot|\eta(y)|V(y)|\cdot|u(y,t)|dy.
\end{aligned}\right.$$

Now we estimate the terms $\|S_{i}(\cdot,t)\|_{q}, i=1,2,3,4$, separately. For the term $\|S_{1}(\cdot,t)\|_{q}$, because $\nabla^{2}_{x}\Gamma_{0}(\cdot,\cdot)$ is a Calder\'on-Zygmund kernel,
we have
\begin{eqnarray*}
\|S_{1}(\cdot,t)\|_{q}&=&\Big\|\int_{\mathbb{R}^{n}}|\eta(y)||\partial_{t}u(y,t)||\nabla_{x}^{2}\Gamma_{0}(\cdot,y)|dy\Big\|_{q}\\
&\lesssim&\Big\{\sup_{y\in B(x_{0}, 2R)}|\partial_{t}u(y,t)|\Big\}\Big\|\int_{\mathbb{R}^{n}}|\eta(y)||\nabla_{x}^{2}\Gamma_{0}(\cdot,y)|dy\Big\|_{q}\\
&\lesssim&\|\eta\|_{q}\Big\{\sup_{y\in B(x_{0}, 2R)}|\partial_{t}u(y,t)|\Big\}.
\end{eqnarray*}

The estimate of  $S_{2}$ is similar to that of $S_{1}$. Noting that $\eta=1$ on $B(x_{0}, 3R/2)$, we can get
\begin{eqnarray*}
S_{2}(x,t)\lesssim\int_{B(x_{0}, 2R)\setminus B(x_{0}, 3R/2)}\frac{|u(y,t)||\Delta\eta(y)|}{|x-y|^{n}}dy
\lesssim\Big\{\sup\limits_{y\in B(x_{0}, 2R)}|u(y,t)|\Big\}\int_{B(x_{0}, 2R)\setminus B(x_{0}, 3R/2)}\frac{|\Delta\eta(y)|}{|x-y|^{n}}dy.
\end{eqnarray*}
For $3R/2<|y-x_{0}|<2R$ and $x\in B(x_{0}, R)$, a direct computation gives $|x-y|\sim R$ and
\begin{eqnarray*}
S_{2}(x,t)&\lesssim&\frac{\sup_{y\in B(x_{0}, 2R)}|u(y,t)|}{R^{n}}\int_{B(x_{0}, 2R)\setminus B(x_{0}, 3R/2)}|\Delta\eta(y)|dy\\
&\lesssim&\frac{\sup_{y\in B(x_{0}, 2R)}|u(y,t)|}{R^{n+2}}\int_{B(x_{0}, 2R)}dy\lesssim\frac{\sup_{B(x_{0}, 2R)}|u(y,t)|}{R^{2}}.
\end{eqnarray*}

Also
\begin{eqnarray*}
\|S_{2}(\cdot,t)\|_{q}&\lesssim&\Big\{\sup_{y\in B(x_{0}, 2R)}|u(y,t)|\Big\}\|\Delta\eta\|_{q}\\
&\lesssim&\Big\{\sup_{y\in B(x_{0}, 2R)}|u(y,t)|\Big\}\Big(\int_{B(x_{0}, 2R)\setminus B(x_{0}, 3R/2)}|\Delta\eta(y)|^{q}dy\Big)^{1/q}\\
&\lesssim&R^{n/q-2}\Big\{\sup_{y\in B(x_{0}, 2R)}|u(y,t)|\Big\}.
\end{eqnarray*}

Following the same procedure, we apply the Young inequality to obtain
\begin{eqnarray*}
\|S_{3}(\cdot,t)\|_{q}&\lesssim&\Big\{\sup_{y\in B(x_{0}, 2R)}|u(y,t)|\Big\}\Big\|\int_{B(x_{0}, 2R)\setminus B(x_{0}, 3R/2)}|\nabla^{2}_{x}\nabla_{y}\Gamma_{0}(\cdot,y)|\cdot
|\nabla\eta(y)|dy\Big\|_{q}\\
&\lesssim&\Big\{\sup_{B(x_{0}, 2R)}|u(y,t)|\Big\}\|\nabla\eta\|_{q}\cdot\Big\|\nabla^{2}_{x}\nabla_{y}\Gamma_{0}(x,\cdot)\chi_{B(x_{0}, 2R)\setminus B(x_{0}, 3R/2)}\Big\|_{1}\\
&\lesssim&\Big\{\sup_{y\in B(x_{0}, 2R)}|u(y,t)|\Big\}R^{n/q-1}\Big(\int_{B(x_{0}, 2R)\setminus B(x_{0}, 3R/2)}\frac{dy}{|x-y|^{n+1}}\Big)\\
&\lesssim&R^{n/q-2}\Big\{\sup_{y\in B(x_{0}, 2R)}|u(y,t)|\Big\}.
\end{eqnarray*}

At last, for the term $S_{4}$, by Lemma \ref{le-2.6} and the condition $V\in B_{q}$,
 we can obtain, via the $L^{p}$-boundedness of the operator with the kernel $\nabla^{2}_{x}\Gamma_{0}(\cdot,\cdot)$, that
\begin{eqnarray*}
\|S_{4}(\cdot,t)\|_{q}&\lesssim&\Big\|V(\cdot)|u(\cdot,t)||\eta(\cdot)|\Big\|_{q}\\
&\lesssim&\Big\{\sup_{y\in B(x_{0}, 2R)}|u(y,t)|\Big\}\Big(\int_{B(x_{0},2R)}V^{q}(y)dy\Big)^{1/q}\\
&\lesssim&\Big\{\sup_{y\in B(x_{0}, 2R)}|u(y,t)|\Big\}R^{n/q}\Big(\frac{1}{|B(x_{0}, 2R)|}\int_{B(x_{0},2R)}V^{q}(y)dy\Big)^{1/q}\\
&\lesssim&\Big\{\sup_{y\in B(x_{0}, 2R)}|u(y,t)|\Big\}R^{n/q-1}\Big(\frac{1}{R^{n-1}}\int_{B(x_{0},2R)}V(y)dy\Big)\\
&\lesssim&\Big\{\sup_{y\in B(x_{0}, 2R)}|u(y,t)|\Big\}R^{n/q-2}\Big(1+\frac{R}{\rho(x_{0})}\Big)^{m_{0}}.
\end{eqnarray*}

The estimates for $\|S_{i}(\cdot, t)\|_{q}, i=1,2,3,4$, indicate that
$$|\nabla^{2}_{x}u(x,t)|\lesssim R^{n/q-2}\Big\{\Big(1+\frac{R}{\rho(x_{0})}\Big)^{m_{0}}+1\Big\}\Big\{\sup_{y\in B(x_{0}, 2R)}|u(y,t)|\Big\}
+R^{n/q}\Big\{\sup_{y\in B(x_{0}, 2R)}|\partial_{t}u(y,t)|\Big\}.$$

Let $u(x_{0},t)=K^{L}_{t}(x_{0}, y_{0})$. Then
\begin{eqnarray*}
&&|\nabla_{x}K^{L}_{t}(x_{0}+h,y_{0})-\nabla_{x}K^{L}_{t}(x_{0}, y_{0})|\\
&&\lesssim |h|^{1-n/q}\Big(\int_{B(x_{0}, R)}|\nabla^{2}_{x}K_{t}^{L}(x, y_{0})|^{q}dx\Big)^{1/q}\\
&&\lesssim\Big(\frac{|h|}{\sqrt{t}}\Big)^{1-n/q}\Big(\frac{\sqrt{t}}{R}\Big)^{1-n/q}\frac{1}{R}\Big\{\Big(1+\frac{R}{\rho(x_{0})}\Big)^{m_{0}}+1\Big\}
\Big\{\sup_{x\in B(x_{0}, 2R)}|K_{t}^{L}(x, y_{0})|
+\frac{R^{2}}{t}\sup_{x\in B(x_{0}, 2R)}|t\partial_{t}K^{L}_{t}(x, y_{0})|\Big\}\\
&&\leq \frac{C_{N}}{R}\Big(\frac{|h|}{\sqrt{t}}\Big)^{1-n/q}\Big(\frac{\sqrt{t}}{R}\Big)^{1-n/q}\Big\{\Big(1+\frac{R}{\rho(x_{0})}\Big)^{m_{0}}+1\Big\}
\Big\{\sup_{x\in B(x_{0}, 2R)}\frac{1}{t^{n/2}}e^{-c|x-y_{0}|^{2}/t}\Big(1+\frac{\sqrt{t}}{\rho(x)}\Big)^{-N}\Big(1+\frac{\sqrt{t}}{\rho(y_{0})}\Big)^{-N}\Big\}\\
&&\quad+C_{N}\frac{R^{2}}{t}\Big\{\sup_{x\in B(x_{0}, 2R)}\frac{1}{t^{n/2}}e^{-c|x-y_{0}|^{2}/t}\Big(1+\frac{\sqrt{t}}{\rho(x)}\Big)^{-N}\Big(1+\frac{\sqrt{t}}{\rho(y_{0})}\Big)^{-N}\Big\}.
\end{eqnarray*}

Take $0<R<\min\{\rho(x_{0}), |x_{0}-y_{0}|/8\}$. If $x\in B(x_{0}, 2R)$, then $|x-x_{0}|<2\rho(x_{0})$, that is, $\rho(x)\sim\rho(x_{0})$.
Also, if $x\in B(x_{0}, 2R)$, $|x-x_{0}|<2R<|x_{0}-y_{0}|/4$, which means that $|x-y_{0}|\sim |x_{0}-y_{0}|$. We can get
\begin{eqnarray}\label{eq-3.9}
&&|\nabla_{x}K^{L}_{t}(x_{0}+h,y_{0})-\nabla_{x}K^{L}_{t}(x_{0}, y_{0})|\\
&&\quad\leq \frac{C_{N}}{t^{(n+1)/2}}\Big(\frac{|h|}{\sqrt{t}}\Big)^{1-n/q}e^{-c|x_{0}-y_{0}|^{2}/t}
\Big(1+\frac{\sqrt{t}}{\rho(x_{0})}\Big)^{-N}\Big(1+\frac{\sqrt{t}}{\rho(y_{0})}\Big)^{-N}
\Big(\frac{\sqrt{t}}{R}\Big)^{1-n/q}\Big(\frac{\sqrt{t}}{R}+\frac{R}{\sqrt{t}}\Big),\nonumber
\end{eqnarray}
 Define a function $F(x)=x^{1-n/q}(x+{1}/{x}),\ x>0$. Then
we can see that for $x>\sqrt{n/(2q-n)}$, $F'(x)>0$, i.e., $F$ is increasing, which means that  the function $f(R):=F(\sqrt{t}/R)$ is decreasing
for $R\in (0, \sqrt{(2q-n)t/n})$.
Below we divide the rest of the proof into two cases.

{\it Case 1: $0<R<\min\{\rho(x_{0}), |x_{0}-y_{0}|\}<\sqrt{t}$}. We further divide the discussion into two subcases.

{\it Case 1.1: $0<R<\rho(x_{0})<\sqrt{t}<\sqrt{(2q-n)t/n}$, i.e., $\rho(x_{0})\leq|x_{0}-y_{0}|$.}
Taking the infimum on both sides of (\ref{eq-3.9}), we can get
$$\Big(\frac{\sqrt{t}}{R}\Big)^{1-n/q}\Big(\frac{\sqrt{t}}{R}+\frac{R}{\sqrt{t}}\Big)
\lesssim\Big(\frac{\sqrt{t}}{\rho(x_{0})}\Big)^{1-n/q}\Big(\frac{\sqrt{t}}{\rho(x_{0})}+\frac{\rho(x_{0})}{\sqrt{t}}\Big)
\lesssim\Big(\frac{\sqrt{t}}{\rho(x_{0})}\Big)^{1-n/q}\Big(\frac{\sqrt{t}}{\rho(x_{0})}+1\Big),$$
where we have used the fact that $\rho(x_{0})<\sqrt{t}$. This gives
\begin{eqnarray*}
&&|\nabla_{x}K^{L}_{t}(x_{0}+h,y_{0})-\nabla_{x}K^{L}_{t}(x_{0}, y_{0})|\\
&&\leq \frac{C_{N}}{t^{(n+1)/2}}\Big(\frac{|h|}{\sqrt{t}}\Big)^{1-n/q}e^{-c|x_{0}-y_{0}|^{2}/t}\Big(1+\frac{\sqrt{t}}{\rho(x_{0})}\Big)^{-N}
\Big(1+\frac{\sqrt{t}}{\rho(y_{0})}\Big)^{-N}\Big\{\Big(\frac{\sqrt{t}}{\rho(x_{0})}\Big)^{1-n/q}\Big(\frac{\sqrt{t}}{\rho(x_{0})}+1\Big)\Big\}\\
&&\leq \frac{C_{N}}{t^{(n+1)/2}} \Big(\frac{|h|}{\sqrt{t}}\Big)^{1-n/q}e^{-c|x_{0}-y_{0}|^{2}/t}\Big(1+\frac{\sqrt{t}}{\rho(x_{0})}\Big)^{-N}
\Big(1+\frac{\sqrt{t}}{\rho(y_{0})}\Big)^{-N}.
\end{eqnarray*}

{\it Case 1.2: $0<R<|x_{0}-y_{0}|<\sqrt{t}$, i.e., $\rho(x_{0})>|x_{0}-y_{0}|$.} By taking the infimum on $f(R)$, we can get

\begin{eqnarray*}
\Big(\frac{\sqrt{t}}{R}\Big)^{1-n/q}\Big(\frac{\sqrt{t}}{R}+\frac{R}{\sqrt{t}}\Big)&\lesssim& \Big(\frac{\sqrt{t}}{|x_{0}-y_{0}|}\Big)^{1-n/q}\Big(\frac{\sqrt{t}}{|x_{0}-y_{0}|}+\frac{|x_{0}-y_{0}|}{\sqrt{t}}\Big)\\
&\lesssim&\Big(\frac{\sqrt{t}}{|x_{0}-y_{0}|}\Big)^{1-n/q}\Big(\frac{\sqrt{t}}{|x_{0}-y_{0}|}+1\Big).
\end{eqnarray*}

Then we obtain
\begin{eqnarray*}
&&|\nabla_{x}K^{L}_{t}(x_{0}+h,y_{0})-\nabla_{x}K^{L}_{t}(x_{0}, y_{0})|\\
&&\leq\frac{C_{N}}{t^{(n+1)/2}}\Big(\frac{|h|}{\sqrt{t}}\Big)^{1-n/q}\Big(\frac{\sqrt{t}}{|x_{0}-y_{0}|}\Big)^{1-n/q}e^{-c|x_{0}-y_{0}|^{2}/t}
\Big(1+\frac{\sqrt{t}}{\rho(x_{0})}\Big)^{-N}
\Big(1+\frac{\sqrt{t}}{\rho(y_{0})}\Big)^{-N}\Big(\frac{\sqrt{t}}{|x_{0}-y_{0}|}+1\Big).
\end{eqnarray*}
It is easy to see that
\begin{eqnarray*}
&&|\nabla_{x}K^{L}_{t}(x_{0}+h,y_{0})-\nabla_{x}K^{L}_{t}(x_{0}, y_{0})|\\
&&\lesssim\Big(\frac{|h|}{|x_{0}-y_{0}|}\Big)^{1-n/q}\frac{1}{t^{(n+1)/2}}e^{-c|x_{0}-y_{0}|^{2}/t}\Big(1+\frac{\sqrt{t}}{\rho(x_{0})}\Big)^{-N}
\Big(1+\frac{\sqrt{t}}{\rho(y_{0})}\Big)^{-N}\Big(\frac{\sqrt{t}}{|x_{0}-y_{0}|}+1\Big)\\
&&\lesssim\Big(\frac{|h|}{|x_{0}-y_{0}|}\Big)^{1-n/q}\frac{1}{t^{(n+1)/2}}e^{-c|x_{0}-y_{0}|^{2}/t}\Big(1+\frac{\sqrt{t}}{\rho(x_{0})}\Big)^{-N}
\Big(1+\frac{\sqrt{t}}{\rho(y_{0})}\Big)^{-N}\frac{\sqrt{t}}{|x_{0}-y_{0}|}\\
&&\quad+\Big(\frac{|h|}{|x_{0}-y_{0}|}\Big)^{1-n/q}\frac{1}{t^{(n+1)/2}}e^{-c|x_{0}-y_{0}|^{2}/t}(1+\frac{\sqrt{t}}{\rho(x_{0})})^{-N}
\Big(1+\frac{\sqrt{t}}{\rho(y_{0})}\Big)^{-N}\\
&&\lesssim\Big(\frac{|h|}{|x_{0}-y_{0}|}\Big)^{1-n/q}\Big(\frac{1}{t^{(n+1)/2}}
+\frac{1}{t^{n/2}|x_{0}-y_{0}|}\Big)e^{-c|x_{0}-y_{0}|^{2}/t}\Big(1+\frac{\sqrt{t}}{\rho(x_{0})}\Big)^{-N}
\Big(1+\frac{\sqrt{t}}{\rho(y_{0})}\Big)^{-N}\\
&&\lesssim\Big(\frac{|h|}{|x_{0}-y_{0}|}\Big)^{1-n/q}\frac{1}{t^{n/2}|x_{0}-y_{0}|}e^{-c|x_{0}-y_{0}|^{2}/t}\Big(1+\frac{\sqrt{t}}{\rho(x_{0})}\Big)^{-N}
\Big(1+\frac{\sqrt{t}}{\rho(y_{0})}\Big)^{-N}.
\end{eqnarray*}

{\it Case 2: $\sqrt{t}<\min\{\rho(x_{0}), |x_{0}-y_{0}|/8\}$.} Similar to Case 1, we divide the discussion into two subcases again.

{\it Case 2.1: $0<R<\sqrt{t}<\sqrt{(2q-n)t/n}<\min\{\rho(x_{0}), |x_{0}-y_{0}|/8\}$.} It follows from (\ref{eq-3.9}) that
\begin{eqnarray}\label{eq-3.3}
&&|\nabla_{x}K^{L}_{t}(x_{0}+h,y_{0})-\nabla_{x}K^{L}_{t}(x_{0}, y_{0})|\\
&&\leq C_{N}\Big(\frac{|h|}{\sqrt{t}}\Big)^{1-n/q}\frac{e^{-c|x_{0}-y_{0}|^{2}/t}}{t^{(n+1)/2}}\Big(1+\frac{\sqrt{t}}{\rho(x_{0})}\Big)^{-N}
\Big(1+\frac{\sqrt{t}}{\rho(y_{0})}\Big)^{-N}\Big(\frac{\sqrt{t}}{R}\Big)^{1-n/q}
\Big(\frac{\sqrt{t}}{R}+\frac{R}{\sqrt{t}}\Big).\nonumber
\end{eqnarray}
Taking the infimum on both sides (\ref{eq-3.3}) reaches
$$|\nabla_{x}K^{L}_{t}(x_{0}+h,y_{0})-\nabla_{x}K^{L}_{t}(x_{0}, y_{0})|\leq \frac{C_{N}}{t^{(n+1)/2}}
\Big(\frac{|h|}{\sqrt{t}}\Big)^{1-n/q}e^{-c|x_{0}-y_{0}|^{2}/t}\Big(1+\frac{\sqrt{t}}{\rho(x_{0})}\Big)^{-N}
\Big(1+\frac{\sqrt{t}}{\rho(y_{0})}\Big)^{-N}.$$

{\it Case 2.2: $0<R<\sqrt{t}<\min\{\rho(x_{0}), |x_{0}-y_{0}|/8\}<\sqrt{(2q-n)t/n}$.} Similarly, taking the infimum on both sides of (\ref{eq-3.3}), we obtain
\begin{eqnarray*}
&&|\nabla_{x}K^{L}_{t}(x_{0}+h,y_{0})-\nabla_{x}K^{L}_{t}(x_{0}, y_{0})|\\
&&\leq C_{N}\Big(\frac{|h|}{\sqrt{t}}\Big)^{1-n/q}\frac{1}{t^{(n+1)/2}}e^{-c|x_{0}-y_{0}|^{2}/t}\Big(1+\frac{\sqrt{t}}{\rho(x_{0})}\Big)^{-N}
\Big(1+\frac{\sqrt{t}}{\rho(y_{0})}\Big)^{-N}\nonumber\\
&&\times\Bigg\{\Big(\frac{\sqrt{t}}{\min\{\rho(x_{0}), |x_{0}-y_{0}|/8\}}\Big)^{1-n/q}
\Big(\frac{\sqrt{t}}{\min\{\rho(x_{0}), |x_{0}-y_{0}|/8\}}+\frac{\min\{\rho(x_{0}), |x_{0}-y_{0}|/8\}}{\sqrt{t}}\Big)\Bigg\}.
\end{eqnarray*}

If $\rho(x_{0})<|x_{0}-y_{0}|/8$, then
\begin{eqnarray*}
&&|\nabla_{x}K^{L}_{t}(x_{0}+h,y_{0})-\nabla_{x}K^{L}_{t}(x_{0}, y_{0})|\\
&&\leq C_{N}\Big(\frac{|h|}{\sqrt{t}}\Big)^{1-n/q}\frac{1}{t^{(n+1)/2}}e^{-c|x_{0}-y_{0}|^{2}/t}\Big(1+\frac{\sqrt{t}}{\rho(x_{0})}\Big)^{-N}
\Big(1+\frac{\sqrt{t}}{\rho(y_{0})}\Big)^{-N}\\
&&\quad\times\Big\{\Big(\frac{\sqrt{t}}{\rho(x_{0})}\Big)^{1-n/q}
\Big(\frac{\sqrt{t}}{\rho(x_{0})}+\frac{\rho(x_{0})}{\sqrt{t}}\Big)\Big\}\\
&&\leq C_{N} \Big(\frac{|h|}{\sqrt{t}}\Big)^{1-n/q}\frac{1}{t^{(n+1)/2}}e^{-c|x_{0}-y_{0}|^{2}/t}\Big(1+\frac{\sqrt{t}}{\rho(x_{0})}\Big)^{-N}
\Big(1+\frac{\sqrt{t}}{\rho(y_{0})}\Big)^{-N}.
\end{eqnarray*}

If $\rho(x_{0})\geq|x_{0}-y_{0}|/8$, we have
\begin{eqnarray*}
&&|\nabla_{x}K^{L}_{t}(x_{0}+h,y_{0})-\nabla_{x}K^{L}_{t}(x_{0}, y_{0})|\\
&&\leq \frac{C_{N}}{t^{(n+1)/2}}\Big(\frac{|h|}{\sqrt{t}}\Big)^{1-n/q}e^{-c|x_{0}-y_{0}|^{2}/t}\Big(1+\frac{\sqrt{t}}{\rho(x_{0})}\Big)^{-N}
\Big(1+\frac{\sqrt{t}}{\rho(y_{0})}\Big)^{-N}\\
&&\quad\times\Big\{\Big(\frac{\sqrt{t}}{|x_{0}-y_{0}|}\Big)^{1-n/q}
\Big(\frac{\sqrt{t}}{|x_{0}-y_{0}|}+\frac{|x_{0}-y_{0}|}{\sqrt{t}}\Big)\Big\}\\
&&\leq \frac{C_{N}}{t^{(n+1)/2}}\Big(\frac{|h|}{|x_{0}-y_{0}|}\Big)^{1-n/q}e^{-c|x_{0}-y_{0}|^{2}/t}\Big(1+\frac{\sqrt{t}}{\rho(x_{0})}\Big)^{-N}
\Big(1+\frac{\sqrt{t}}{\rho(y_{0})}\Big)^{-N}\Big(\frac{\sqrt{t}}{|x_{0}-y_{0}|}+\frac{|x_{0}-y_{0}|}{\sqrt{t}}\Big)\\
&&\leq\frac{C_{N}}{t^{n/2}|x_{0}-y_{0}|}\Big(\frac{|h|}{|x_{0}-y_{0}|}\Big)^{1-n/q}e^{-c|x_{0}-y_{0}|^{2}/t}\Big(1+\frac{\sqrt{t}}{\rho(x_{0})}\Big)^{-N}
\Big(1+\frac{\sqrt{t}}{\rho(y_{0})}\Big)^{-N}\\
&&\quad+\frac{C_{N}}{t^{(n+1)/2}}\Big(\frac{|h|}{|x_{0}-y_{0}|}\Big)^{1-n/q}e^{-c|x_{0}-y_{0}|^{2}/t}\Big(1+\frac{\sqrt{t}}{\rho(x_{0})}\Big)^{-N}
\Big(1+\frac{\sqrt{t}}{\rho(y_{0})}\Big)^{-N}.
\end{eqnarray*}

If $\sqrt{t}<|x_{0}-y_{0}|$, then
\begin{eqnarray*}
|\nabla_{x}K^{L}_{t}(x_{0}+h,y_{0})-\nabla_{x}K^{L}_{t}(x_{0}, y_{0})|
&\leq& C_{N}\Big(\frac{|h|}{\sqrt{t}}\Big)^{1-n/q}
\frac{e^{-c|x_{0}-y_{0}|^{2}/t}}{t^{(n+1)/2}}\Big(1+\frac{\sqrt{t}}{\rho(x_{0})}+\frac{\sqrt{t}}{\rho(y_{0})}\Big)^{-N}.
\end{eqnarray*}

If $\sqrt{t}\geq |x_{0}-y_{0}|$, then
\begin{eqnarray*}
|\nabla_{x}K^{L}_{t}(x_{0}+h,y_{0})-\nabla_{x}K^{L}_{t}(x_{0}, y_{0})|
&\leq&C_{N}\Big(\frac{|h|}{|x_{0}-y_{0}|}\Big)^{1-n/q}\frac{e^{-c|x_{0}-y_{0}|^{2}/t}}{t^{n/2}|x_{0}-y_{0}|}\Big(1+\frac{\sqrt{t}}{\rho(x_{0})}
+\frac{\sqrt{t}}{\rho(y_{0})}\Big)^{-N}.
\end{eqnarray*}
\end{proof}

\begin{lemma}\label{le-3.4}
Suppose that $V\in B_{q}$ for some $q>n$. Let $\delta'=1-n/q$. For every $N>0$, there exists  a constant $C_{N}>0$ such that for all $x,y\in\mathbb{R}^{n}$ and $t>0$, the semigroup kernels $K_{t}^{L}(\cdot,\cdot)$ satisfy the following estimate: for $|h|<|x-y|/4$,
$$|\nabla_{x}K^{L}_{t}(x+h,y)-\nabla_{x}K^{L}_{t}(x, y)|\leq\frac{C_{N}}{t^{(n+1)/2}}\Big(\frac{|h|}{\sqrt{t}}\Big)^{\delta'}\Big(1+\frac{\sqrt{t}}{\rho(x)}
+\frac{\sqrt{t}}{\rho(y)}\Big)^{-N}.$$
\end{lemma}

\begin{proof}
Similar to Lemma \ref{le-3.3}, we take $R\in (0, \min\{\rho(x_{0}), \sqrt{t}\})$ and obtain that
\begin{eqnarray}\label{eq-3.11}
&&|\nabla_{x}K^{L}_{t}(x_{0}+h,y_{0})-\nabla_{x}K^{L}_{t}(x_{0}, y_{0})|\\
&&\quad\leq \frac{C_{N}}{t^{(n+1)/2}}\Big(\frac{|h|}{\sqrt{t}}\Big)^{1-n/q}e^{-c|x_{0}-y_{0}|^{2}/t}\Big(1+\frac{\sqrt{t}}{\rho(x_{0})}\Big)^{-N}
\Big(1+\frac{\sqrt{t}}{\rho(y_{0})}\Big)^{-N}\Big\{\Big(\frac{\sqrt{t}}{R}\Big)^{1-n/q}
\Big(\frac{\sqrt{t}}{R}+\frac{R}{\sqrt{t}}\Big)\Big\}\nonumber\\
&&\quad\leq \frac{C_{N}}{t^{(n+1)/2}}\Big(\frac{|h|}{\sqrt{t}}\Big)^{1-n/q}\Big(1+\frac{\sqrt{t}}{\rho(x_{0})}\Big)^{-N}
\Big(1+\frac{\sqrt{t}}{\rho(y_{0})}\Big)^{-N}\Big\{\Big(\frac{\sqrt{t}}{R}\Big)^{1-n/q}
\Big(\frac{\sqrt{t}}{R}+\frac{R}{\sqrt{t}}\Big)\Big\}\nonumber.
\end{eqnarray}

{\it Case 1: $\rho(x_{0})\leq\sqrt{t}$}. This implies $0<R<\rho(x_{0})<\sqrt{t}<\sqrt{(2q-n)t/n}$. We can get
\begin{eqnarray*}
&&|\nabla_{x}K^{L}_{t}(x_{0}+h,y_{0})-\nabla_{x}K^{L}_{t}(x_{0}, y_{0})|\\
&&\quad\leq \frac{C_{N}}{t^{(n+1)/2}}\Big(\frac{|h|}{\sqrt{t}}\Big)^{1-n/q}\Big(1+\frac{\sqrt{t}}{\rho(x_{0})}\Big)^{-N}
\Big(1+\frac{\sqrt{t}}{\rho(y_{0})}\Big)^{-N}\Big\{\Big(\frac{\sqrt{t}}{\rho(x_{0})}\Big)^{1-n/q}
\Big(\frac{\sqrt{t}}{\rho(x_{0})}+\frac{\rho(x_{0})}{\sqrt{t}}\Big)\Big\}\\
&&\quad \leq \frac{C_{N}}{t^{(n+1)/2}}\Big(\frac{|h|}{\sqrt{t}}\Big)^{1-n/q}\Big(1+\frac{\sqrt{t}}{\rho(x_{0})}\Big)^{-N}
\Big(1+\frac{\sqrt{t}}{\rho(y_{0})}\Big)^{-N}.
\end{eqnarray*}

{\it Case 2: $\rho(x_{0})>\sqrt{t}$}. For this case, $0<R<\sqrt{t}$. Then the following two cases are considered.
$$\begin{cases}
{\it Case\ 2.1:}\quad 0<R<\rho(x_{0})<\sqrt{(2q-n)t/n};\\
{\it Case\ 2.2:}\quad 0<R<\sqrt{(2q-n)t/n}<\rho(x_{0}).
\end{cases}$$

It is obvious that Case 2.1 comes back to Case 1. For Case 2.2, letting $R\rightarrow \sqrt{(2q-n)t/n}$ on the right-hand side of (\ref{eq-3.11}), we have
$$|\nabla_{x}K^{L}_{t}(x_{0}+h,y_{0})-\nabla_{x}K^{L}_{t}(x_{0}, y_{0})|\leq \frac{C_{N}}{t^{(n+1)/2}}\Big(\frac{|h|}{\sqrt{t}}\Big)^{1-n/q}\Big(1+\frac{\sqrt{t}}{\rho(x_{0})}\Big)^{-N}
\Big(1+\frac{\sqrt{t}}{\rho(y_{0})}\Big)^{-N}.$$

\end{proof}

\begin{proposition}\label{prop-3.8}
Suppose that $\alpha>0$ and $V\in B_{q}$ for some $q>n$. Let $\delta'=1-n/q$. For every $N>0$, there exists a  constant $C_{N}>0$  such that for all $x,y\in\mathbb{R}^{n}$ and
$t>0$, the fractional heat kernels $K_{\alpha, t}^{L}(\cdot,\cdot)$ satisfy the following estimate: for $|h|<|x-y|/4$,
$$|\nabla_{x}K^{L}_{\alpha,t}(x+h,y)-\nabla_{x}K^{L}_{\alpha,t}(x,y)|
\leq C_{N}\Big(\frac{|h|}{t^{1/2\alpha}}\Big)^{\delta'}\frac{1}{t^{1/2\alpha}}\frac{t}{(t^{1/2\alpha}+|x-y|)^{n+2\alpha}}.$$
\end{proposition}

\begin{proof}
By the subordinative formula and Lemma \ref{le-3.3}, we can get
\begin{eqnarray*}
|\nabla_{x}K^{L}_{\alpha,t}(x+h,y)-\nabla_{x}K^{L}_{\alpha,t}(x,y)|
\leq C_{N}\int^{\infty}_{0}\frac{t}{s^{1+\alpha}}\Big|\nabla_{x}K^{L}_{t}(x+h,y)-\nabla_{x}K^{L}_{t}(x,y)\Big|ds\leq C_{N}(M_{6}+M_{7}),
\end{eqnarray*}
where
$$\left\{\begin{aligned}
M_{6}&:=\int^{|x-y|}_{0}\frac{t}{s^{1+\alpha}}\Big(\frac{|h|}{\sqrt{s}}\Big)^{\delta'}\frac{1}{s^{(n+1)/2}}e^{-c|x-y|^{2}/s}
\Big(1+\frac{\sqrt{s}}{\rho(x)}\Big)^{-N}\Big(1+\frac{\sqrt{s}}{\rho(y)}\Big)^{-N}ds;\\
M_{7}&:=\int_{|x-y|}^{\infty}\frac{t}{s^{1+\alpha}}\Big(\frac{|h|}{|x-y|}\Big)^{\delta'}\frac{1}{s^{n/2}|x-y|}e^{-c|x-y|^{2}/s}
\Big(1+\frac{\sqrt{s}}{\rho(x)}\Big)^{-N}\Big(1+\frac{\sqrt{s}}{\rho(y)}\Big)^{-N}ds.
\end{aligned}\right.$$

We first estimate $M_{6}$ and apply change of variables to obtain
\begin{eqnarray*}
M_{6}&\lesssim&\int^{\infty}_{0}\frac{t}{s^{1+\alpha}}\Big(\frac{|h|}{\sqrt{s}}\Big)^{\delta'}\frac{1}{s^{(n+1)/2}}e^{-c|x-y|^{2}/s}
\Big(1+\frac{\sqrt{s}}{\rho(x)}\Big)^{-N}\Big(1+\frac{\sqrt{s}}{\rho(y)}\Big)^{-N}ds\\
&\lesssim&t^{-(n+1)/2\alpha}\Big(\frac{t^{1/2\alpha}}{\rho(x)}\Big)^{-N}\Big(\frac{t^{1/2\alpha}}{\rho(y)}\Big)^{-N}\Big(\frac{|h|}{t^{1/2\alpha}}\Big)^{\delta'}
\int^{\infty}_{0}u^{-N-1-\alpha-(n+1)/2-\delta'/2}e^{-c|x-y|^{2}/t^{1/\alpha}u}du\\
&\lesssim&t^{-(n+1)/2\alpha}\Big(\frac{t^{1/2\alpha}}{\rho(x)}\Big)^{-N}\Big(\frac{t^{1/2\alpha}}{\rho(y)}\Big)^{-N}
\Big(\frac{|h|}{t^{1/2\alpha}}\Big)^{\delta'}\int^{\infty}_{0}\Big(\frac{r^{2}t^{1/\alpha}}{|x-y|^{2}}\Big)^{N+1+\alpha+(n+1)/2+\delta'/2}e^{-cr^{2}}
\frac{|x-y|^{2}}{t^{1/\alpha}r^{3}}dr\\
&\lesssim&\Big(\frac{t^{1/2\alpha}}{\rho(x)}\Big)^{-N}\Big(\frac{t^{1/2\alpha}}{\rho(y)}\Big)^{-N}\Big(\frac{|h|}{|x-y|}\Big)^{\delta'}
\frac{t^{1+N/\alpha}}{|x-y|^{2\alpha+2N+n+1}}.
\end{eqnarray*}

Similarly, for $M_{7}$, we have
\begin{eqnarray*}
M_{7}&\lesssim&\frac{1}{|x-y|}\int^{\infty}_{0}\frac{t}{s^{1+\alpha}}\Big(\frac{|h|}{\sqrt{s}}\Big)^{\delta'}\frac{1}{s^{n/2}}e^{-c|x-y|^{2}/s}
\Big(1+\frac{\sqrt{s}}{\rho(x)}\Big)^{-N}\Big(1+\frac{\sqrt{s}}{\rho(y)}\Big)^{-N}ds\\
&\lesssim&\Big(\frac{t^{1/2\alpha}}{\rho(x)}\Big)^{-N}\Big(\frac{t^{1/2\alpha}}{\rho(y)}\Big)^{-N}\Big(\frac{|h|}{|x-y|}\Big)^{\delta'}
\frac{t^{1+N/\alpha}}{|x-y|^{2\alpha+2N+n+1}},
\end{eqnarray*}
which gives
$$|\nabla_{x}K^{L}_{\alpha,t}(x+h,y)-\nabla_{x}K^{L}_{\alpha,t}(x,y)|\leq \frac{C_{N}t^{1+N/\alpha}}{|x-y|^{2\alpha+2N+n+1}}\Big(\frac{|h|}{|x-y|}\Big)^{\delta'}
 \Big(\frac{t^{1/2\alpha}}{\rho(x)}\Big)^{-N}\Big(\frac{t^{1/2\alpha}}{\rho(y)}\Big)^{-N}.$$

On the other hand, we can deduce from  Lemma \ref{le-3.4} that
\begin{eqnarray*}
&&|\nabla_{x}K^{L}_{\alpha,t}(x+h,y)-\nabla_{x}K^{L}_{\alpha,t}(x,y)|\\
&&\leq C_{N}\int^{\infty}_{0}s^{-(n+1)/2}\frac{1}{t^{1/\alpha}}\eta^{\alpha}_{1}\Big(\frac{s}{t^{1/\alpha}}\Big)\Big(\frac{|h|}{\sqrt{s}}\Big)^{\beta}
\Big(1+\frac{\sqrt{s}}{\rho(x)}\Big)^{-N}\Big(\frac{\sqrt{s}}{\rho(y)}\Big)^{-N}ds\\
&&\leq C_{N}\Big(\frac{t^{1/2\alpha}}{\rho(x)}\Big)^{-N}\Big(\frac{t^{1/2\alpha}}{\rho(y)}\Big)^{-N}
\Big(\frac{|h|}{t^{1/2\alpha}}\Big)^{\delta'}\frac{1}{t^{(n+1)/2\alpha}}.
\end{eqnarray*}

Finally,  the arbitrariness of $N$ indicates that
\begin{eqnarray*}
&&\Big(1+\frac{t^{1/2\alpha}}{\rho(x)}\Big)^{N}\Big(1+\frac{t^{1/2\alpha}}{\rho(y)}\Big)^{N}|\nabla_{x}K^{L}_{\alpha,t}(x+h,y)-\nabla_{x}K^{L}_{\alpha,t}(x,y)|\\
&&\quad\leq C_{N}\min\Bigg\{\Big(\frac{|h|}{|x-y|}\Big)^{\delta'}\frac{t^{1+N/\alpha}}{|x-y|^{2\alpha+2N+n+1}},\
 \Big(\frac{|h|}{t^{1/2\alpha}}\Big)^{\delta'}\frac{1}{t^{(n+1)/2\alpha}}\Bigg\},
\end{eqnarray*}
which proves Proposition \ref{prop-3.8}.

\end{proof}

\begin{proposition}\label{prop-3.3-2}
Assume that $V\in B_{q}$ for some $q>n$. Let $\alpha\in (0,1/2-n/2q)$. For every $N>0$,
 $$\Big|t^{1/2\alpha}\nabla_{x}e^{-tL^{\alpha}}(1)(x)\Big|\lesssim \min\Bigg\{\Big(\frac{t^{1/2\alpha}}{\rho(x)}\Big)^{1+2\alpha},\
 \Big(\frac{t^{1/2\alpha}}{\rho(x)}\Big)^{-N}\Bigg\}.$$
\end{proposition}
\begin{proof}
We divide the proof into two cases.

{\it Case 1: $t^{1/2\alpha}>\rho(x)$.} By Proposition \ref{prop-3.3-1},  we  use  a direct computation to obtain \begin{eqnarray*}
|t^{1/2\alpha}\nabla_{x}e^{-t{L}^{\alpha}}(1)(x)|
&\lesssim&\int_{\mathbb{R}^{n}}\frac{t}{(t^{1/2\alpha}+|x-y|)^{n+2\alpha}}\Big(1+\frac{t^{1/2\alpha}}{\rho(x)}\Big)^{-N}
\Big(1+\frac{t^{1/2\alpha}}{\rho(y)}\Big)^{-N}dy\\
&\lesssim&\Big(\frac{t^{1/2\alpha}}{\rho(x)}\Big)^{-N}\int_{\mathbb{R}^{n}}\frac{t}{(t^{1/2\alpha}+|x-y|)^{n+2\alpha}}dy\\
&\lesssim&\Big(\frac{t^{1/2\alpha}}{\rho(x)}\Big)^{-N}.
\end{eqnarray*}
Because $t^{1/2\alpha}>\rho(x)$, then
$$|t^{1/2\alpha}\nabla_{x}e^{-t{L}^{\alpha}}(1)(x)|\lesssim \min\Bigg\{\Big(\frac{t^{1/2\alpha}}{\rho(x)}\Big)^{\delta},\
 \Big(\frac{t^{1/2\alpha}}{\rho(x)}\Big)^{-N}\Bigg\}.$$

{\it Case 2: $t^{1/2\alpha}\leq\rho(x)$.} It follows from (\ref{eq-sub-for-1}) that
\begin{eqnarray*}
t^{1/2\alpha}\nabla_{x}e^{-tL^{\alpha}}(1)(x)=t^{1/2\alpha}\nabla_{x}\int_{\mathbb{R}^{n}}K^{L}_{\alpha,t}(x,y)dy=I_{1}+I_{2},
\end{eqnarray*}
where
$$\left\{\begin{aligned}
I_{1}&:=t^{1/2\alpha}\int^{\infty}_{\rho^{2}(x)}\eta_{t}^{\alpha}(s)\Big(\int_{\mathbb{R}^{n}}\nabla_{x}K^{L}_{s}(x,y)dy\Big)ds;\\
I_{2}&:=t^{1/2\alpha}\int_{0}^{\rho^{2}(x)}\eta_{t}^{\alpha}(s)\Big(\int_{\mathbb{R}^{n}}\nabla_{x}K^{L}_{s}(x,y)dy\Big)ds.
\end{aligned}
\right.$$
We claim that
\begin{equation}\label{eq-3.5}
L_{s}(x):=\int_{\mathbb{R}^{n}}\nabla_{x}K^{L}_{s}(x,y)dy\lesssim\frac{1}{\sqrt{s}}.
\end{equation}
In fact, by Lemma \ref{le3.1}, we have
$L_{s}(x)\lesssim L_{s,1}(x)+L_{s,2}(x)$, where
$$\left\{\begin{aligned}
&L_{s,1}(x):=\int_{\{y: \sqrt{s}\leq |x-y|\}}\frac{1}{s^{(n+1)/2}}e^{-c|x-y|^{2}/s}
\Big(1+\frac{\sqrt{s}}{\rho(x)}+\frac{\sqrt{s}}{\rho(y)}\Big)^{-N}dy;\\
&L_{s,2}(x):= \int_{\{y: |y-x|<\sqrt{s}\}}\frac{1}{s^{n/2}|x-y|}e^{-c|x-y|^{2}/s}
\Big(1+\frac{\sqrt{s}}{\rho(x)}+\frac{\sqrt{s}}{\rho(y)}\Big)^{-N}dy.\end{aligned}\right.$$

Taking $N$ large enough, it is easy to see that
\begin{eqnarray*}
L_{s,1}(x)&\lesssim&\int_{\mathbb{R}^{n}}\frac{1}{s^{(n+1)/2}}e^{-c|x-y|^{2}/s}dy\lesssim\frac{1}{\sqrt{s}}.
\end{eqnarray*}

Similarly, a direct calculus gives, together with changing variable: $|x-y|/\sqrt{s}=u$,
\begin{eqnarray*}
L_{s,2}(x)&\lesssim&\frac{1}{\sqrt{s}}\int_{\mathbb{R}^{n}}\frac{1}{s^{n/2}|x-y|/\sqrt{s}}e^{-c|x-y|^{2}/s}dy
\lesssim\frac{1}{\sqrt{s}}\int_{0}^{\infty}u^{n-2}e^{-cu^{2}}du
\lesssim\frac{1}{\sqrt{s}}.
\end{eqnarray*}

Then we can deduce from (\ref{eq-3.5}) that
\begin{eqnarray}\label{eq-3.15}
I_{1}&\lesssim&t^{1/2\alpha}\int^{\infty}_{\rho^{2}(x)}\eta_{t}^{\alpha}(s)\frac{ds}{\sqrt{s}}
\lesssim t^{1+1/2\alpha}\int^{\infty}_{\rho^{2}(x)}\frac{1}{s^{\alpha+3/2}}ds
\lesssim \Big(\frac{t^{1/2\alpha}}{\rho(x)}\Big)^{1+2\alpha}.
\end{eqnarray}

For $I_{2}$, it follows from the formula
$$h_{u}(x-y)-K^{L}_{u}(x,y)=\int^{u}_{0}\int_{\mathbb{R}^{n}}h_{s}(x-z)V(z)K_{u-s}(z,y)dzds$$
that for $\delta=2-n/q>1$,
\begin{eqnarray*}
|\sqrt{u}\nabla_{x}e^{-uL}(1)(x)|&\lesssim&\int^{u}_{0}\sqrt{\frac{u}{s}}\Big(\frac{\sqrt{s}}{\rho(x)}\Big)^{\delta}\frac{ds}{s}
\lesssim\Big(\frac{\sqrt{u}}{\rho(x)}\Big)^{\delta}.
\end{eqnarray*}
Therefore, noting that $0<2\alpha<1-n/q$, we can use the change of variables to obtain
\begin{eqnarray}\label{eq-3.14}
I_{2}&\lesssim&t^{1/2\alpha}\int_{0}^{\rho^{2}(x)}\eta_{t}^{\alpha}(s)\frac{1}{\sqrt{s}}\Big(\sqrt{s}\int_{\mathbb{R}^{n}}\nabla_{x}K^{L}_{s}(x,y)dy\Big)ds\\
&\lesssim&t^{1/2\alpha}\int_{0}^{\rho^{2}(x)}\frac{1}{t^{1/\alpha}}\eta_{1}^{\alpha}(s/t^{1/\alpha})\frac{1}{\sqrt{s}}\Big(\frac{\sqrt{s}}{\rho(x)}\Big)^{\delta}ds\nonumber\\
&\lesssim&t^{1/2\alpha-1/\alpha}\int_{0}^{\rho(x)^{2}/t^{1/\alpha}}\eta_{1}^{\alpha}(\tau)\frac{1}{\sqrt{t^{1/\alpha}\tau}}
\Big(\frac{\sqrt{t^{1/\alpha}\tau}}{\rho(x)}\Big)^{\delta}t^{1/\alpha}d\tau\nonumber\\
&\lesssim&\Big(\frac{t^{1/2\alpha}}{\rho(x)}\Big)^{1+2\alpha}.\nonumber
\end{eqnarray}

The above estimates (\ref{eq-3.14})\ and \  (\ref{eq-3.15}) imply that
\begin{eqnarray*}
|t^{1/2\alpha}\nabla_{x}e^{-tL^{\alpha}}(1)(x)|\lesssim\Big(\frac{t^{1/2\alpha}}{\rho(x)}\Big)^{1+2\alpha}.
\end{eqnarray*}
\end{proof}

\subsection{Estimation on time-fractional derivatives}\label{sec-3.3}
In this section we give some gradient estimate for the fractional heat kernel associated with the variable $t$.
Define an operator
$$D_{\alpha,t}^{L,\beta}(f)=t^{\beta}\partial^{\beta}_{t}e^{-tL^{\alpha}}f,\ \ \alpha\in(0,1)\ \&\ \beta>0.$$
Denote by $D_{\alpha,t}^{L,\beta}(\cdot,\cdot)$ the integral kernel of $D_{\alpha,t}^{L,\beta}$. Then we can get the following proposition.
\begin{proposition}\label{prop-3.6}
Let $\alpha\in(0,1)$ and $\beta>0$. For every $N>0$, there exists a constant $C_{N}>0$ such that
\begin{equation}\label{eq-3.6}
|D_{\alpha,t}^{L,\beta}(x,y)|\leq \frac{C_{N}t^{\beta}}{(t^{1/2\alpha}+|x-y|)^{n+2\alpha\beta}}\Big(1+\frac{t^{1/2\alpha}}{\rho(x)}+\frac{t^{1/2\alpha}}{\rho(y)}\Big)^{-N}.
\end{equation}
\end{proposition}

\begin{proof}
The following two cases are considered.

{\it Case 1: $\beta\in(0,1)$.}
It is easy to see that
\begin{eqnarray*}
t^{\beta}\partial_{t}^{\beta}e^{-tL^{\alpha}}&=&c_{\beta}t^{\beta}\int^{\infty}_{0}\partial_{t}e^{-(t+s)L^{\alpha}}\frac{ds}{s^{\beta}}
= c_{\beta}t^{\beta}\int^{\infty}_{0}(-L)^{\alpha}e^{-(t+s)L^{\alpha}}\frac{ds}{s^{\beta}}\\
&=&c_{\beta}t^{\beta}\int^{\infty}_{0}(t+s)(-L)^{\alpha}e^{-(t+s)L^{\alpha}}\frac{ds}{(t+s)s^{\beta}},
\end{eqnarray*}
which, together with Proposition \ref{pro2.6}, gives
\begin{eqnarray*}
\Big|D^{L,\beta}_{\alpha,t}(x,y)\Big|&\leq&C_{N}t^{\beta}\int^{\infty}_{0}\frac{1}{((t+s)^{1/2\alpha}+|x-y|)^{n+2\alpha}}
\Big(1+\frac{(t+s)^{1/2\alpha}}{\rho(x)}\Big)^{-N}\Big(1+\frac{(t+s)^{1/2\alpha}}{\rho(y)}\Big)^{-N}\frac{ds}{s^{\beta}}\\
&\leq&C_{N}t^{\beta}\int^{\infty}_{0}\frac{1}{(t+s)^{n/2\alpha+1}}
\Big(\frac{(t+s)^{1/2\alpha}}{\rho(x)}\Big)^{-N}\Big(\frac{(t+s)^{1/2\alpha}}{\rho(y)}\Big)^{-N}\frac{ds}{s^{\beta}}\\
&\leq&C_{N}t^{\beta}\rho(x)^{N}\rho(y)^{N}\int^{\infty}_{0}(t+s)^{-n/2\alpha-N/\alpha-1}s^{-\beta}ds\\
&\leq&\frac{C_{N}}{t^{n/2\alpha}}\Big(\frac{t^{1/2\alpha}}{\rho(x)}\Big)^{-N}\Big(\frac{t^{1/2\alpha}}{\rho(y)}\Big)^{-N}.
\end{eqnarray*}

One the other hand, since $e^{-tL^{\alpha}}=\int^{\infty}_{0}\eta^{\alpha}_{1}(\tau)e^{-t^{1/\alpha}\tau L}d\tau$,
\begin{eqnarray*}
 t^{\beta}\partial^{\beta}_{t}e^{-tL^{\alpha}} &\simeq&t^{\beta}\int^{\infty}_{0}\partial_{r}\Big(\int^{\infty}_{0}\eta^{\alpha}_{1}(\tau)e^{-(t+r)^{1/\alpha}\tau L}d\tau\Big)\frac{dr}{r^{\beta}}  \\
  &=&t^{\beta}\int^{\infty}_{0}\Big(\int^{\infty}_{0}\eta^{\alpha}_{1}(\tau)(t+r)^{1/\alpha-1}\tau Le^{-(t+r)^{1/\alpha}\tau L}d\tau\Big)\frac{dr}{r^{\beta}}\\
        &=&C_{\beta}t^{\beta}\int^{\infty}_{0}\Big(\int^{\infty}_{0}Q^{L}_{\sqrt{(t+r)^{1/\alpha}\tau},1}\frac{dr}{(t+r)r^{\beta}}\Big)\eta^{\alpha}_{1}(\tau)d\tau.
\end{eqnarray*}
By Proposition \ref{prop-2.1}, we can get
\begin{eqnarray*}
  \Big|D_{\alpha,t}^{L,\beta}(x,y)\Big| &\leq& C_{N}t^{\beta}\int^{\infty}_{0}\eta^{\alpha}_{1}(\tau)\Bigg\{\int^{\infty}_{0}((t+r)^{1/\alpha}\tau)^{-n/2}e^{-c|x-y|^{2}/(t+r)^{1/\alpha}\tau}\\
  &&\quad  \Big(1+\frac{(t+r)^{1/2\alpha}\sqrt{\tau}}{\rho(x)}\Big)^{-N}\Big(1+\frac{(t+r)^{1/2\alpha}\sqrt{\tau}}{\rho(y)}\Big)^{-N}
  \frac{dr}{r^{\beta}(t+r)}\Bigg\}d\tau \\
         &\leq&C_{N}\frac{t^{\beta}\rho(x)^{N}\rho(y)^{N}}{|x-y|^{n+2\alpha\beta}}\int^{\infty}_{0}\eta^{\alpha}_{1}(\tau)\tau^{\alpha\beta-N}
   \Big(\int^{\infty}_{0}(t+r)^{\beta-N/\alpha-1}r^{-\beta}dr\Big)d\tau  \\
   &\leq& \frac{C_{N}t^{\beta}}{|x-y|^{n+2\alpha\beta}}\Big(\frac{t^{1/2\alpha}}{\rho(x)}\Big)^{-N}\Big(\frac{t^{1/2\alpha}}{\rho(y)}\Big)^{-N}.
\end{eqnarray*}
By the arbitrariness of $N$, we obtain
\begin{eqnarray*}
|D_{\alpha,t}^{L,\beta}(x,y)|&\leq&C_{N}\min\Big\{\frac{1}{t^{n/2\alpha}},\ \frac{t^{\beta}}{|x-y|^{n+2\alpha\beta}}\Big\}\Big(1+\frac{t^{1/2\alpha}}{\rho(x)}\Big)^{-N}\Big(1+\frac{t^{1/2\alpha}}{\rho(y)}\Big)^{-N}\\
&\leq& \frac{C_{N}t^{\beta}}{(t^{1/2\alpha}+|x-y|)^{n+2\beta \alpha}}\Big(1+\frac{t^{1/2\alpha}}{\rho(x)}+\frac{t^{1/2\alpha}}{\rho(y)}\Big)^{-N}.\nonumber
\end{eqnarray*}
{\it Case 2: $\beta\geq 1$.} Let $m=[\beta]+1$. We can get
\begin{eqnarray*}
 t^{\beta}\partial_{t}^{\beta}e^{-tL^{\alpha}}&=&c_{\beta}t^{\beta}\int^{\infty}_{0}\partial^{m}_{t}e^{-(t+s)L^{\alpha}}\frac{ds}{s^{\beta+1-m}}\\
&=&c_{\beta}t^{\beta}\int^{\infty}_{0}(-L)^{m\alpha}e^{-(t+s)L^{\alpha}}\frac{ds}{s^{1+\beta-m}}\\
&=&c_{\beta}t^{\beta}\int^{\infty}_{0}(t+s)^{m}(-L)^{m\alpha}e^{-(t+s)L^{\alpha}}\frac{ds}{(t+s)^{m}s^{1+\beta-m}}.
\end{eqnarray*}
It follows from  Proposition \ref{pro2.6} that
\begin{eqnarray*}
\Big|D_{\alpha,t}^{L,\beta}(x,y)\Big| &\leq&C_{N}t^{\beta}\int^{\infty}_{0}\frac{(t+s)^{m}}{(t+s)^{n/2\alpha+m}}
\Big(\frac{(t+s)^{1/2\alpha}}{\rho(x)}\Big)^{-N}\Big(\frac{(t+s)^{1/2\alpha}}{\rho(y)}\Big)^{-N}\frac{ds}{(t+s)^{m}s^{1+\beta-m}}\\
&\leq&C_{N}t^{\beta}\rho(x)^{N}\rho(y)^{N}\int^{\infty}_{0}(t+s)^{-n/2\alpha-N/2\alpha-m}\frac{ds}{s^{1+\beta-m}}\\
&\leq& \frac{C_{N}}{t^{n/2\alpha}}\Big(1+\frac{t^{1/2\alpha}}{\rho(x)}\Big)^{-N}\Big(1+\frac{t^{1/2\alpha}}{\rho(y)}\Big)^{-N}.
\end{eqnarray*}
On the other hand, we obtain
\begin{eqnarray*}
  \Big|D_{\alpha,t}^{L,\beta}(x,y)\Big| &\leq& t^{\beta}\int^{\infty}_{0}\Bigg\{\int^{\infty}_{0}\Big|\partial^{m}_{r}K^{ L}_{(t+r)^{1/\alpha}\tau}(x,y)\Big|\frac{dr}{r^{\beta+1-m}}\Bigg\}\eta^{\alpha}_{1}(\tau)d\tau \\
   &\leq&C_{N}t^{\beta}\int^{\infty}_{0}\Bigg\{\int^{\infty}_{0}(t+r)^{-m}((t+r)^{1/\alpha}\tau)^{-n/2}e^{-c|x-y|^{2}/(t+r)^{1/\alpha}\tau}\\
   &&\quad
   \Big(\frac{\sqrt{\tau}(t+r)^{1/2\alpha}}{\rho(x)}\Big)^{-N}\Big(\frac{\sqrt{\tau}(t+r)^{1/2\alpha}}{\rho(y)}\Big)^{-N}\frac{dr}{r^{\beta+1-m}}
   \Bigg\}\eta^{\alpha}_{1}(\tau)d\tau  \\
     &\leq&C_{N}\frac{t^{\beta}\rho(x)^{N}\rho(y)^{N}}{|x-y|^{n+2\alpha\beta}}t^{-N/\alpha}\int^{\infty}_{0}\Big(\int^{\infty}_{0}
   (1+u)^{-m+\beta-N/\beta}u^{m-\beta-1}du\Big)\eta^{\alpha}_{1}(\tau)\tau^{\alpha\beta-N}d\tau  \\
   &\leq& \frac{C_{N}t^{\beta}}{|x-y|^{n+2\alpha\beta}}\Big(1+\frac{t^{1/2\alpha}}{\rho(x)}\Big)^{-N}\Big(1+\frac{t^{1/2\alpha}}{\rho(y)}\Big)^{-N},
     \end{eqnarray*}
which indicates (\ref{eq-3.6}) holds.
\end{proof}

In the next proposition, we give the Lipschitz continuity of $D^{L}_{\alpha,t}(\cdot,\cdot)$.
\begin{proposition}\label{prop-3.4}
Let $\alpha\in(0,1)$ and $\beta>0$. Let $0<\delta'\leq\delta=\min\{2\alpha,\delta_{0}\}$. For every $N>0$, there exists a constant $C_{N}>0$ such that for all $|h|\leq t^{1/2\alpha}$,
$$|D_{\alpha,t}^{L,\beta}(x+h,y)-D_{\alpha,t}^{L,\beta}(x,y)|\leq C_{N}\Big(\frac{|h|}{t^{1/2\alpha}}\Big)^{\delta'}\frac{t^{\beta}}{(t^{1/2\alpha}+|x-y|)^{n+2 \alpha\beta }}\Big(1+\frac{t^{1/2\alpha}}{\rho(x)}+\frac{t^{1/2\alpha}}{\rho(y)}\Big)^{-N}.$$
\end{proposition}
\begin{proof}
It is equivalent to verify
\begin{eqnarray}\label{eq-3.7}
&&|D_{\alpha,t}^{L,\beta}(x+h,y)-D_{\alpha,t}^{L,\beta}(x,y)|\leq C_{N}\Big(\frac{|h|}{t^{1/2\alpha}}\Big)^{\delta'}\min\Big\{\frac{1}{t^{n/2\alpha}},\ \frac{t^{\beta}}{|x-y|^{n+2\alpha\beta}}\Big\}\Big(1+\frac{t^{1/2\alpha}}{\rho(x)}+\frac{t^{1/2\alpha}}{\rho(y)}\Big)^{-N}.
\end{eqnarray}
Without loss of generality, for $m=[\beta]+1$,  it holds
\begin{eqnarray*}
  t^{\beta}\partial^{\beta}_{t}e^{-tL^{\alpha}}(x,y) = c_{\beta}t^{\beta}\int^{\infty}_{0}(t+s)^{m}(-L)^{m\alpha}e^{-(t+s)L^{\alpha}}\frac{ds}{(t+s)^{m}s^{1+\beta-m}}.
\end{eqnarray*}
By Proposition \ref{pro2.6}, we can get
\begin{eqnarray*}
  &&\Big|D_{\alpha,t}^{L,\beta}(x+h,y)-D_{\alpha,t}^{L,\beta}(x,y)\Big| \\
&&\leq C_{N} t^{\beta}\int^{\infty}_{0}\frac{(t+s)^{m}({|h|}/{(t+s)^{1/2\alpha}})^{\delta'}}{((t+s)^{1/2\alpha}+|x-y|)^{n+2\alpha m}}\Big(1+\frac{(t+s)^{1/2\alpha}}{\rho(x)}\Big)^{-N}\Big(1+\frac{(t+s)^{1/2\alpha}}{\rho(y)}\Big)^{-N}\frac{ds}{(t+s)^{m}s^{1+\beta-m}}\\
&&\leq C_{N}t^{\beta}|h|^{\delta'}\rho(x)^{N}\rho(y)^{N}\int^{\infty}_{0}(t+s)^{-(n+\delta')/2\alpha-N/2\alpha-m}\frac{ds}{s^{1+\beta-m}}\\
&&\leq C_{N}t^{-n/2\alpha}\Big(\frac{t^{1/2\alpha}}{\rho(x)}\Big)^{-N}\Big(\frac{t^{1/2\alpha}}{\rho(y)}\Big)^{-N}\Big(\frac{|h|}{t^{1/2\alpha}}\Big)^{\delta'}.
\end{eqnarray*}
On the other hand, we obtain
\begin{eqnarray*}
  &&\Big|D_{\alpha,t}^{L,\beta}(x+h,y)-D_{\alpha,t}^{L,\beta}(x,y)\Big|\\
     &&\leq C_{N}t^{\beta}\int^{\infty}_{0}\Bigg\{\int^{\infty}_{0}(t+r)^{-m-n/2\alpha}\tau^{-n/2}\Big(\frac{|h|}{\sqrt{(t+r)^{1/\alpha}\tau}}\Big)^{\delta'}\\
   &&\quad \times e^{-c|x-y|^{2}/(t+r)^{1/\alpha}\tau}
   \Big(\frac{\sqrt{\tau}(t+r)^{1/2\alpha}}{\rho(x)}\Big)^{-N}\Big(\frac{\sqrt{\tau}(t+r)^{1/2\alpha}}{\rho(y)}\Big)^{-N}\frac{dr}{r^{\beta+1-m}}
   \Bigg\}\eta^{\alpha}_{1}(\tau)d\tau  \\
      &&\leq C_{N}\frac{t^{\beta}|h|^{\delta'}\rho(x)^{N}\rho(y)^{N}}{|x-y|^{n+2\alpha\beta}}\int^{\infty}_{0}\Bigg\{\int^{\infty}_{0}
   (t+r)^{-m+\beta-N/\alpha-\delta'/2\alpha}r^{m-\beta-1}
   dr\Bigg\}\eta^{\alpha}_{1}(\tau)\tau^{\alpha\beta-N-\delta'/2}d\tau \\
    &&\leq \frac{C_{N}t^{\beta}}{|x-y|^{n+2\alpha\beta}}\Big(\frac{|h|}{t^{1/2\alpha}}\Big)^{\delta'}
   \Big(1+\frac{t^{1/2\alpha}}{\rho(x)}\Big)^{-N}\Big(1+\frac{t^{1/2\alpha}}{\rho(y)}\Big)^{-N},
   \end{eqnarray*}
which implies (\ref{eq-3.7}).

\end{proof}

\begin{proposition}\label{prop-3.5}
Let $\alpha\in(0,1)$,  $\beta>0$ and   $0<\delta'\leq\min\{2\alpha,\delta_{0}\}$. For every $N>0$, there exist constants $C_{N}>0$ such that
$$\Big|\int_{\mathbb{R}^{n}}D_{\alpha,t}^{L,\beta}(x,y)dy\Big|\leq C_{N}\frac{(t^{1/2\alpha}/\rho(x))^{\delta'}}{(1+t^{1/2\alpha}/\rho(x))^{N}}.$$
\end{proposition}
\begin{proof}
 Let $m=[\beta]+1$. By (iii) of Proposition \ref{pro2.6}, we change the order of integrations to obtain
\begin{eqnarray*}
\Big|\int_{\mathbb R^{n}}D^{L,\beta}_{\alpha, t}(x,y)dy\Big|&=&\Big|\int_{\mathbb R^{n}}\Big\{t^{\beta}\int^{\infty}_{0}\partial^{m}_{t}D^{L,\beta}_{\alpha, t+s}(x,y)\frac{ds}{s^{1+\beta-m}}\Big\}dy\Big|\\
&\leq&t^{\beta}\int^{\infty}_{0}\int_{\mathbb R^{n}}\Big|\partial^{m}_{t}D^{L,\beta}_{\alpha, t+s}(x,y)\Big|\frac{dyds}{s^{1+\beta-m}}\\
&\leq& C_{N}t^{\beta}\int^{\infty}_{0}\frac{((t+s)^{1/2\alpha}/\rho(x))^{\delta'}}{(1+(t+s)^{1/2\alpha}/\rho(x))^{N}}\frac{ds}{s^{1+\beta-m}(t+s)^{m}}.
\end{eqnarray*}
If $t^{1/2\alpha}>\rho(x)$, then
\begin{eqnarray*}
\Big|\int_{\mathbb R^{n}}D^{L,\beta}_{\alpha, t}(x,y)dy\Big|&\leq&C_{N}t^{\beta}\rho(x)^{N-\delta'}\int^{\infty}_{0}(t+s)^{\delta'/2\alpha-N/2\alpha-m}s^{m-\beta-1}ds\\
&\leq&C_{N}\frac{(t^{1/2\alpha}/\rho(x))^{\delta'}}{(1+t^{1/2\alpha}/\rho(x))^{N}}.
\end{eqnarray*}
If $t^{1/2\alpha}\leq\rho(x)$, then
\begin{eqnarray*}
\Big|\int_{\mathbb R^{n}}D^{L,\beta}_{\alpha, t}(x,y)dy\Big|&\leq&C_{N}t^{\beta}\int^{\infty}_{0}\Big((t+s)^{1/2\alpha}/\rho(x)\Big)^{\delta'}\frac{ds}{s^{1+\beta-m}(t+s)^{m}}\\
&\leq&C_{N}t^{\beta}\rho(x)^{\delta'}\int^{\infty}_{0}(t+s)^{\delta'/2\alpha-m}s^{m-1-\beta}ds\\
&\leq&C_{N}\Big(t^{1/2\alpha}/\rho(x)\Big)^{\delta'}\leq C_{N}\frac{(t^{1/2\alpha}/\rho(x))^{\delta'}}{(1+t^{1/2\alpha}/\rho(x))^{N}},
\end{eqnarray*}
which completes the proof of Proposition \ref{prop-3.5}.
\end{proof}

\section{Characterization of Campanato-Morrey spaces associated with $L$}\label{sec-4}

Firstly, we deduce a reproducing formula.

\begin{lemma}\label{le-4.7-add}
 Let $\alpha\in(0,1)$ and $\beta>0$.
The operator $t^{\beta}\partial_{t}^{\beta}e^{-tL^{\alpha}}$ defines an isometry from $L^{2}(\mathbb R^{n})$ into $L^{2}(\mathbb R^{n+1}_{+}, dxdt/t)$. Moreover, in the sense of $L^{2}(\mathbb R^{n})$, it holds
$$f(x)=c_{\alpha, \beta}\lim_{N\rightarrow\infty}\lim_{\epsilon\rightarrow 0}\int^{N}_{\epsilon}(t^{\beta}\partial_{t}^{\beta}e^{-tL^{\alpha}})^{2}(f)(x)\frac{dt}{t}.$$
\end{lemma}

\begin{proof}
Note that for $dE(\lambda)$ the spectral resolution of the operator $L$, it follows from
$$e^{-tL^{\alpha}}=\int^{\infty}_{0}e^{-t\lambda^{\alpha}}dE(\lambda)$$
that
\begin{eqnarray*}
t^{\beta}\partial^{\beta}_{t}e^{-tL^{\alpha}}&=&t^{\beta}\int^{\infty}_{0}\partial^{m}_{t}\Big(\int^{\infty}_{0}e^{-(t+s)\lambda^{\alpha}}dE(\lambda)\Big)
\frac{ds}{s^{1+\beta-m}}\\
&=&t^{\beta}\int^{\infty}_{0}\Big(\int^{\infty}_{0}(-1)^{m}\lambda^{\alpha m}e^{-(t+s)\lambda^{\alpha}}dE(\lambda)\Big)\frac{ds}{s^{1+\beta-m}}\\
&\approx&t^{\beta}\int^{\infty}_{0}(-1)^{m}\lambda^{m\alpha}e^{-t\lambda^{\alpha}}\lambda^{\alpha(\beta-m)}dE(\lambda)\\
&=&\int^{\infty}_{0}(t\lambda^{\alpha})^{\beta}e^{-t\lambda^{\alpha}}dE(\lambda).
\end{eqnarray*}

Then for $f\in L^{2}(\mathbb R^{n})$, we have
\begin{eqnarray*}
\|t^{\beta}\partial_{t}^{\beta}e^{-tL^{\alpha}}(f)(x)\|^{2}_{L^{2}(\mathbb R^{n+1}_{+}, dxdt/t)}&=&\int^{\infty}_{0}\Big(\int_{\mathbb R^{n}}|t^{\beta}\partial_{t}^{\beta}e^{-tL^{\alpha}}(f)(x)|^{2}dx\Big)\frac{dt}{t}\\
&=&\int^{\infty}_{0}\Big\langle t^{\beta}\partial_{t}^{\beta}e^{-tL^{\alpha}}(f),\ t^{\beta}\partial_{t}^{\beta}e^{-tL^{\alpha}}(f)\Big\rangle\frac{dt}{t}\\
&=&\int^{\infty}_{0}\Big\langle (t^{\beta}\partial_{t}^{\beta}e^{-tL^{\alpha}})^{2}(f),\ f\Big\rangle\frac{dt}{t}\\
&=&\int^{\infty}_{0}\int_{0}^{\infty}t^{2\beta}\lambda^{2\alpha\beta}e^{-2t\lambda^{\alpha}}\frac{dt}{t}dE_{f,f}(\lambda)\\
&=&C_{\alpha,\beta}\|f\|^{2}_{2}.
\end{eqnarray*}

Below we only prove that for every pair of sequences $n_{k}\uparrow\infty$ and $\epsilon_{k}\downarrow0$ as $k\rightarrow \infty$,
\begin{equation}\label{eq-4.1}
\lim_{k\rightarrow\infty}\int^{n_{k+m}}_{n_{k}}(t^{\beta}\partial^{\beta}_{t}e^{-tL^{\alpha}})^{2}f(x)\frac{dt}{t}
=\lim_{k\rightarrow\infty}\int^{\epsilon_{k+m}}_{\epsilon_{k}}(t^{\beta}\partial^{\beta}_{t}e^{-tL^{\alpha}})^{2}f(x)\frac{dt}{t}=0.
\end{equation}
If (\ref{eq-4.1}) holds, there exists a function $h\in L^{2}(\mathbb R^{n})$ such that
$$\lim_{k\rightarrow\infty}\int^{n_{k}}_{\epsilon_{k}}(t^{\beta}\partial^{\beta}_{t}e^{-tL^{\alpha}})^{2}f(x)\frac{dt}{t}=h(x),$$
which implies that for all $g\in L^{2}(\mathbb R^{n})$,
\begin{eqnarray*}
\langle h,\ g\rangle&=&\Big\langle\lim_{k\rightarrow\infty}\int^{n_{k}}_{\epsilon_{k}}(t^{\beta}\partial^{\beta}_{t}e^{-tL^{\alpha}})^{2}f\frac{dt}{t},\ g\Big\rangle\\
&=&\lim_{k\rightarrow\infty}\int^{n_{k}}_{\epsilon_{k}}\Big\langle(t^{\beta}\partial^{\beta}_{t}e^{-tL^{\alpha}})^{2}f,\ g\Big\rangle\frac{dt}{t}\\
&=&\lim_{k\rightarrow\infty}\int^{n_{k}}_{\epsilon_{k}}\Big\langle t^{\beta}\partial^{\beta}_{t}e^{-tL^{\alpha}}f,\ t^{\beta}\partial^{\beta}_{t}e^{-tL^{\alpha}}g\Big\rangle\frac{dt}{t}\\
&=&C_{\alpha,\beta}\langle f,\ g\rangle.
\end{eqnarray*}
This means $h=C_{\alpha,\beta}f$. Now we verify (\ref{eq-4.1}). As $k\rightarrow\infty$,
\begin{eqnarray*}
\Big\|\int^{n_{k+m}}_{n_{k}}(t^{\beta}\partial^{\beta}_{t}e^{-tL^{\alpha}})^{2}f(x)\frac{dt}{t}\Big\|^{2}_{L^{2}}&\lesssim&
\int^{n_{k+m}}_{n_{k}}\Big\|(t^{\beta}\partial^{\beta}_{t}e^{-tL^{\alpha}})^{2}f(x)\Big\|^{2}_{L^{2}}\frac{dt}{t}\\
&=&\int^{\infty}_{0}\int^{n_{k+m}}_{n_{k}}t^{2\beta}\lambda^{2\alpha\beta}e^{-2t\lambda^{\alpha}}\frac{dt}{t}dE_{f,f}(\lambda)\rightarrow0
\end{eqnarray*}
since
$$\lim_{k\rightarrow\infty}\Big|\int^{n_{k+m}}_{n_{k}}t^{2\beta}\lambda^{2\alpha\beta}e^{-2t\lambda^{\alpha}}\frac{dt}{t}\Big|=0.$$
The integral
$$\lim_{k\rightarrow\infty}\int^{\epsilon_{k+m}}_{\epsilon_{k}}(t^{\beta}\partial^{\beta}_{t}e^{-tL^{\alpha}})^{2}f(x)\frac{dt}{t}$$
can be dealt with similarly.

\end{proof}

The following inequality was established by   Harboure-Salinas-Viviani \cite{HSV}.
\begin{lemma} \label{le-4.3-1} {\rm (\cite[(5.3)]{HSV})}
Let $0<\gamma\leq 1$.
For any pair of measurable functions  $F$ and $G$ on $\mathbb R^{n+1}_{+}$, we have
\begin{eqnarray*}
&&\iint_{\mathbb R^{n+1}_{+}}|F(x,t)|\cdot|G(x,t)|\frac{dxdt}{t}\\
&&\leq C\sup_{B}\Bigg\{\frac{1}{|B|^{1+2\gamma/n}}\iint_{\widehat{B}}|F(x,t)|^{2}\frac{dxdt}{t}\Bigg\}^{1/2}\Bigg\{\int_{\mathbb R^{n}}\Big(\int^{\infty}_{0}\int_{|x-y|<t}|G(y,t)|^{2}\frac{dydt}{t^{n+1}}\Big)^{n/2(n+\gamma)}dx\Bigg\}^{1+\gamma/n}.
\end{eqnarray*}
\end{lemma}
In Lemma \ref{le-4.3-1}, letting
$$\left\{\begin{aligned}
F(x,t):=t^{2\alpha\beta}\partial^{\beta}_{s}e^{-sL^{\alpha}}\mid_{s=t^{2\alpha}}(f)(x)=Q^{L,\beta}_{\alpha,t}(f)(x);\\
G(x,t):=t^{2\alpha\beta}\partial^{\beta}_{s}e^{-sL^{\alpha}}\mid_{s=t^{2\alpha}}(g)(x)=Q^{L,\beta}_{\alpha,t}(g)(x),
\end{aligned}\right.$$
we have
\begin{eqnarray}\label{eq-4.5}
&&\iint_{\mathbb R^{n+1}_{+}}|Q^{L,\beta}_{\alpha,t}(f)(x)|\cdot|Q^{L,\beta}_{\alpha,t}(g)(x)|\frac{dxdt}{t}\\
&&\quad\leq C\sup_{B}\Big(\frac{1}{|B|^{1+2\gamma/n}}\iint_{\widehat{B}}|Q^{L,\beta}_{\alpha,t}(f)(x)|^{2}\frac{dxdt}{t}\Big)^{1/2}\nonumber\\
&&\ \quad\times\Bigg\{\int_{\mathbb R^{n}}\Big(\int^{\infty}_{0}\int_{|x-y|<t}|Q^{L,\beta}_{\alpha,t}(g)(x)|^{2}\frac{dydt}{t^{n+1}}\Big)^{n/2(n+\gamma)}dx\Bigg\}^{1+\gamma/n}.\nonumber
\end{eqnarray}
On the left-hand side of (\ref{eq-4.5}), since
$$\left\{\begin{aligned}
Q^{L,\beta}_{\alpha,t}(f)(x)=t^{2\alpha\beta}L^{\alpha\beta}e^{-t^{2\alpha}L^{\alpha}}(f);\\
Q^{L,\beta}_{\alpha,t}(g)(x)=t^{2\alpha\beta}L^{\alpha\beta}e^{-t^{2\alpha}L^{\alpha}}(g),
\end{aligned}\right.$$
we can get, via the change of variables,
\begin{eqnarray*}
&&\iint_{\mathbb R^{n+1}_{+}}|Q^{L,\beta}_{\alpha,t}(f)(x)|\cdot|Q^{L,\beta}_{\alpha,t}(g)(x)|\frac{dxdt}{t}\\
&&\quad=\iint_{\mathbb R^{n+1}_{+}}|t^{2\alpha\beta}L^{\alpha\beta}e^{-t^{2\alpha}L^{\alpha}}(f)(x)|\cdot|t^{2\alpha\beta}L^{\alpha\beta}e^{-t^{2\alpha}L^{\alpha}}(g)(x)|\frac{dxdt}{t}\\
&&\quad=\iint_{\mathbb R^{n+1}_{+}}|s^{\beta}L^{\alpha\beta}e^{-sL^{\alpha}}(f)(x)|\cdot|s^{\beta}L^{\alpha\beta}e^{-sL^{\alpha}}(g)(x)|\frac{dxds}{s}\\
&&\quad=\iint_{\mathbb R^{n+1}_{+}}|s^{\beta}\partial^{\beta}_{s}e^{-sL^{\alpha}}(f)(x)|\cdot|s^{\beta}\partial^{\beta}_{s}e^{-sL^{\alpha}}(g)(x)|\frac{dxds}{s}.
\end{eqnarray*}

On the right-hand side of (\ref{eq-4.5}), using change of variables again, we obtain
\begin{eqnarray*}
&&\sup_{B}\Big(\frac{1}{|B|^{1+2\gamma/n}}\int^{r_{B}}_{0}\int_{B}|t^{2\alpha\beta}L^{\alpha\beta}e^{-t^{2\alpha}L^{\alpha}}(f)(x)|^{2}\frac{dxdt}{t}\Big)^{1/2}\\
&&\quad\lesssim\sup_{B}\Big(\frac{1}{|B|^{1+2\gamma/n}}\int^{r_{B}^{2\alpha}}_{0}\int_{B}|s^{\beta}L^{\alpha\beta}e^{-sL^{\alpha}}(f)(x)|^{2}\frac{dxds}{s}\Big)^{1/2}\\
&&\quad=\sup_{B}\Big(\frac{1}{|B|^{1+2\gamma/n}}\int^{r_{B}^{2\alpha}}_{0}\int_{B}|s^{\beta}\partial_{s}^{\beta}e^{-sL^{\alpha}}(f)(x)|^{2}\frac{dxds}{s}\Big)^{1/2},
\end{eqnarray*}
meanwhile,
\begin{eqnarray*}
&&\Bigg\{\int_{\mathbb R^{n}}\Big(\int^{\infty}_{0}\int_{|x-y|<t}|t^{2\alpha\beta}L^{\alpha\beta}e^{-t^{2\alpha}L^{\alpha}}(g)(y)|^{2}\frac{dydt}{t^{n+1}}\Big)^{n/2(n+\gamma)}dx\Bigg\}^{1+\gamma/n}\\
&&\quad\lesssim \Bigg\{\int_{\mathbb R^{n}}\Big(\int^{\infty}_{0}\int_{|x-y|<s^{1/2\alpha}}|s^{\beta}L^{\alpha\beta}
e^{-sL^{\alpha}}(f)(y)|^{2}\frac{s^{1/2\alpha-1}dyds}{s^{(n+1)/2\alpha}}\Big)^{n/2(n+\gamma)}dx\Bigg\}^{1+\gamma/n}\\
&&\quad\lesssim \Bigg\{\int_{\mathbb R^{n}}\Big(\int^{\infty}_{0}\int_{|x-y|<s^{1/2\alpha}}|s^{\beta}L^{\alpha\beta}e^{-sL^{\alpha}}(g)(y)|^{2}\frac{s^{1/2\alpha-1}dyds}
{s^{n/2\alpha+1}}\Big)^{n/2(n+\gamma)}dx\Bigg\}^{1+\gamma/n}\\
&&\quad\lesssim \Bigg\{\int_{\mathbb R^{n}}\Big(\int^{\infty}_{0}\int_{|x-y|<s^{1/2\alpha}}|s^{\beta}\partial^{\beta}_{s}e^{-sL^{\alpha}}(g)(y)|^{2}\frac{s^{1/2\alpha-1}dyds}
{s^{n/2\alpha+1}}\Big)^{n/2(n+\gamma)}dx\Bigg\}^{1+\gamma/n}.
\end{eqnarray*}

Finally, we have
\begin{eqnarray}\label{eq-4.8}
&&\iint_{\mathbb R^{n+1}_{+}}|s^{\beta}\partial^{\beta}e^{-sL^{\alpha}}(f)(x)|\cdot|s^{\beta}\partial^{\beta}_{s}e^{-sL^{\alpha}}(g)(x)|\frac{dxds}{s}\\
&&\quad \lesssim \sup_{B}\Big(\frac{1}{|B|^{1+2\gamma/n}}\int^{r_{B}^{2\alpha}}_{0}\int_{B}|s^{\beta}\partial_{s}^{\beta}e^{-sL^{\alpha}}(f)(x)|^{2}\frac{dxds}{s}\Big)^{1/2}\nonumber\\
&&\ \quad\times \Bigg\{\int_{\mathbb R^{n}}\Big(\iint_{|x-y|<s^{1/2\alpha}}|s^{\beta}\partial^{\beta}_{s}e^{-sL^{\alpha}}(g)(y)|^{2}\frac{s^{1/2\alpha-1}dyds}
{s^{n/2\alpha+1}}\Big)^{n/2(n+\gamma)}dx\Bigg\}^{1+\gamma/n}.\nonumber
\end{eqnarray}

For $\alpha\in (0,1)$ and $\beta>0$, define  an area function $S^{L}_{\alpha,\beta}$ as follows:
$$S^{L}_{\alpha,\beta}(h)(x):=\Big(\iint_{\Gamma_{\alpha}(x)}|t^{\beta}\partial_{t}^{\beta}e^{-tL^{\alpha}}(h)(y)|^{2}\frac{dydt}{t^{n/2\alpha+1}}\Big)^{1/2},$$
where $\Gamma_{\alpha}(x)$ denotes the cone $\{(y,t):\ |x-y|<t^{1/2\alpha}\}$.

\begin{lemma}\label{le-4.4}
Let $\alpha\in (0,1)$ and $\beta>0$.
The area function $S^{L}_{\alpha,\beta}$ is bounded on $L^{2}(\mathbb R^{n})$.
\end{lemma}
\begin{proof}
Let
$$g^{L}_{\alpha,\beta}(h)(x):=\Big(\int^{\infty}_{0}|t^{\beta}\partial^{\beta}_{t}e^{-tL^{\alpha}}h(x)|^{2}\frac{dt}{t}\Big)^{1/2}.$$
We can get
\begin{eqnarray}\label{eq-4.6}
\|g^{L}_{\alpha,\beta}(h)\|^{2}_{2}&=&\int_{\mathbb R^{n}}\Big(\int^{\infty}_{0}|t^{\beta}\partial_{t}^{\beta}e^{-tL^{\alpha}}h(x)|^{2}\frac{dt}{t}\Big)dx\\
&=&\int^{\infty}_{0}\Big(\int_{\mathbb R^{n}}|t^{\beta}\partial_{t}^{\beta}e^{-tL^{\alpha}}h(x)|^{2}dx\Big)\frac{dt}{t}\nonumber\\
&=&\int^{\infty}_{0}\Big\langle t^{\beta}\partial_{t}^{\beta}e^{-tL^{\alpha}}h,\ t^{\beta}\partial_{t}^{\beta}e^{-tL^{\alpha}}h\Big\rangle\frac{dt}{t}\nonumber\\
&=&\int^{\infty}_{0}\Big\langle (t^{\beta}\partial_{t}^{\beta}e^{-tL^{\alpha}})^{2}h,\ h\Big\rangle\frac{dt}{t}\nonumber\\
&=&\int^{\infty}_{0}\int^{\infty}_{0}t^{2\beta}\lambda^{2\alpha\beta}e^{-t\lambda^{\alpha}}dE_{h,h}(\lambda)\frac{dt}{t}\nonumber\\
&\lesssim&\|h\|_{2}^{2}.\nonumber
\end{eqnarray}
Hence, it follows from (\ref{eq-4.6}) that
\begin{eqnarray*}
\|S^L_{\alpha,\beta}h\|^{2}_{2}&=&\int_{\mathbb R^{n}}\Big(\iint_{\Gamma_{\alpha}(x)}|t^{\beta}\partial_{t}^{\beta}e^{-tL^{\alpha}}(h)(y)|^{2}\frac{dydt}{t^{\frac{n}{2\alpha}+1}}\Big)dx\\
&=&\int_{\mathbb R^{n}}\Big(\int^{\infty}_{0}\int_{\mathbb R^{n}}|t^{\beta}\partial_{t}^{\beta}e^{-tL^{\alpha}}(h)(y)|^{2}\chi_{\Gamma_{\alpha}(x)}(y)\frac{dydt}{t^{\frac{n}{2\alpha}+1}}\Big)dx\\
&\lesssim&\int^{\infty}_{0}\int_{\mathbb R^{n}}|t^{\beta}\partial_{t}^{\beta}e^{-tL^{\alpha}}(h)(y)|^{2}\Big(\int_{\mathbb R^{n}}\chi_{\Gamma_{\alpha}(y)}(x)dx\Big)\frac{dydt}{t^{\frac{n}{2\alpha}+1}}\\
&\lesssim&\int^{\infty}_{0}\int_{\mathbb R^{n}}|t^{\beta}\partial_{t}^{\beta}e^{-tL^{\alpha}}(h)(y)|^{2}\frac{dydt}{t}\lesssim \|h\|_{2}^{2}.
\end{eqnarray*}

\end{proof}

\begin{theorem}\label{th-4.1}
Assume that $\alpha\in (0,1)$, $\beta>0$ and $0<\gamma\leq \min\{2\alpha,2\alpha\beta\}$. Let $f$ be a linear combination of $H^{n/(n+\gamma)}_{L}$-atoms. There exists a constant $C$ such that
$$\|S^{L}_{\alpha,\beta}(f)\|_{L^{n/(n+\gamma)}}\leq C\|f\|_{H^{n/(n+\gamma)}_{L}}.$$
\end{theorem}
\begin{proof}
Let $a$ be an $H^{n/(n+\gamma)}_{L}$-atom associated with a ball $B=B(x_{0}, r)$. Then
we write
$$\|S^{L}_{\alpha,\beta}(a)\|^{n/(n+\gamma)}_{L^{n/(n+\gamma)}}\leq I+II,$$
where
$$\left\{\begin{aligned}
&I:=\int_{8B}|S^{L}_{\alpha,\beta}(a)(x)|^{n/(n+\gamma)}dx;\\
&II:=\int_{(8B)^{c}}|S^{L}_{\alpha,\beta}(a)(x)|^{n/(n+\gamma)}dx.
\end{aligned}\right.$$

We use Lemma \ref{le-4.4} and H\"older's inequality to obtain
\begin{eqnarray*}
I\lesssim\Big(\int_{8B}|S^{L}_{\alpha,\beta}a(x)|^{2}dx\Big)^{n/2(n+\gamma)}|B|^{(n+2\gamma)/2(n+\gamma)}
\lesssim\|a\|^{n/(n+\gamma)}_{2}|B|^{(n+2\gamma)/2(n+\gamma)}\lesssim1.
\end{eqnarray*}
Now we deal with $II$ in the following two cases.

{\it Case 1: $r<\rho(x_{0})/4$.} For this case, $\int_{\mathbb R^{n}}a(x)dx=0$.
We write $(S^{L}_{\alpha,\beta}(a)(x))^{2}\lesssim I_{1}(x)+I_{2}(x),$
where
$$\left\{\begin{aligned}
I_{1}(x)&:=\int^{(|x-x_{0}|/2)^{2\alpha}}_{0}\int_{|x-y|<t^{1/2\alpha}}\Big\{\int_{\mathbb R^{n}}(t^{\beta}\partial_{t}^{\beta}e^{-tL^{\alpha}}(y,x')-t^{\beta}\partial_{t}^{\beta}e^{-tL^{\alpha}}(y,x_{0}))a(x')dx'\Big\}^{2}\frac{dydt}{t^{n/2\alpha+1}};\\
I_{2}(x)&:=\int_{(|x-x_{0}|/2)^{2\alpha}}^{\infty}\int_{|x-y|<t^{1/2\alpha}}\Big\{\int_{\mathbb R^{n}}(t^{\beta}\partial_{t}^{\beta}e^{-tL^{\alpha}}(y,x')-t^{\beta}\partial_{t}^{\beta}e^{-tL^{\alpha}}(y,x_{0}))a(x')dx'\Big\}^{2}\frac{dydt}{t^{n/2\alpha+1}}.
\end{aligned}\right.$$

We first estimate $I_{1}$. Since $|x-y|\leq |x-x_{0}|/2$, then $|y-x'|\sim |y-x_{0}|$ for $x'\in B(x_{0}, r)$ and $x\notin B(x_{0}, 8r)$. We can use Propositions \ref{prop-3.6}\ and \ \ref{prop-3.4} to deduce there exists $\delta'>\gamma$ such that
\begin{eqnarray*}
|I_{1}(x)|&\lesssim&\int^{(|x-x_{0}|/2)^{2\alpha}}_{0}\int_{|x-y|<t^{1/2\alpha}}\Big\{\int_{B}\Big(\frac{|x'-x_{0}|}{t^{1/2\alpha}}\Big)^{\delta'}
\frac{t^{\beta}}{(t^{1/2\alpha}+|y-x_{0}|)^{n+2\alpha\beta}}\frac{dx'}{|B|^{1+\gamma/n}}\Big\}^{2}\frac{dydt}{t^{n/2\alpha+1}}\\
&\lesssim&\int^{(|x-x_{0}|/2)^{2\alpha}}_{0}\int_{|x-y|<t^{1/2\alpha}}\Big\{\int_{B}\Big(\frac{r}{t^{1/2\alpha}}\Big)^{\delta'}
\frac{t^{\beta}}{t^{\beta+n/2\alpha}(1+|y-x_{0}|/t^{1/2\alpha})^{n+2\alpha\beta}}\frac{dx'}{|B|^{1+\gamma/n}}\Big\}^{2}\frac{dydt}{t^{n/2\alpha+1}}\\
&\lesssim&\int^{(|x-x_{0}|/2)^{2\alpha}}_{0}\int_{|x-y|<t^{1/2\alpha}}\Big(\frac{r}{t^{1/2\alpha}}\Big)^{2\delta'}
\frac{1}{t^{n/\alpha}(1+|y-x_{0}|/t^{1/2\alpha})^{2n+4\alpha\beta}}\frac{1}{|B|^{2\gamma/n}}\frac{dydt}{t^{n/2\alpha+1}}.
\end{eqnarray*}
Because $0<t<|x-x_{0}|^{2\alpha}/2^{2\alpha}$ and $|x-y|<t^{1/2\alpha}$, then $|x-y|<|x-x_{0}|/2$. This implies $|y-x_{0}|\gtrsim |x-x_{0}|/2$. We have
\begin{eqnarray*}
|I_{1}(x)|&\lesssim&\int^{(|x-x_{0}|/2)^{2\alpha}}_{0}\int_{|x-y|<t^{1/2\alpha}}\Big(\frac{r}{t^{1/2\alpha}}\Big)^{2\delta'}
\frac{1}{t^{n/\alpha}(1+|x-x_{0}|/t^{1/2\alpha})^{2n+4\alpha\beta}}\frac{1}{|B|^{2\gamma/n}}\frac{dydt}{t^{n/2\alpha+1}}\\
&\lesssim&\int^{(|x-x_{0}|/2)^{2\alpha}}_{0}\Big(\frac{r}{t^{1/2\alpha}}\Big)^{2\delta'}
\frac{1}{t^{n/\alpha}(1+|x-x_{0}|/t^{1/2\alpha})^{2n+4\alpha\beta}}\frac{1}{|B|^{2\gamma/n}}\frac{dt}{t}\\
&\lesssim&\frac{r^{2(\delta'-\gamma)}}{|x-x_{0}|^{2(n+\delta')}},
\end{eqnarray*}
which, via a direct computation gives
$$\int_{(8B)^{c}}|I_{1}(x)|^{n/2(n+\gamma)}dx\lesssim \int_{(8B)^{c}}\Big(\frac{r^{\delta'-\gamma}}{|x-x_{0}|^{n+\delta'}}\Big)^{n/(n+\gamma)}dx\lesssim C.$$

Let us continue with $I_{2}$. Similarly, it follows from Proposition \ref{prop-3.4} that
\begin{eqnarray*}
|I_{2}(x)|&\lesssim&
\int^{\infty}_{|x-x_{0}|^{2\alpha}/2^{2\alpha}}\int_{|x-y|<t^{1/2\alpha}}\Big\{\int_{B}\Big(\frac{|x'-x_{0}|}{t^{1/2\alpha}}\Big)^{\delta'}\frac{1}{t^{n/2\alpha}}
\frac{dx'}{|B|^{1+\gamma/n}}\Big\}^{2}\frac{dydt}{t^{n/2\alpha+1}}\\
&\lesssim&\int^{\infty}_{|x-x_{0}|^{2\alpha}/2^{2\alpha}}\int_{|x-y|<t^{1/2\alpha}}\Big(\frac{r}{t^{1/2\alpha}}\Big)^{2\delta'}\frac{1}{t^{n/\alpha}}
\frac{1}{|B|^{2\gamma/n}}\frac{dydt}{t^{n/2\alpha+1}}\\
&\lesssim&\frac{r^{2(\delta'-\gamma)}}{|x-x_{0}|^{2n+2\delta'}}.
\end{eqnarray*}
Hence we still have
$\int_{(8B)^{c}}|I_{2}(x)|^{n/2(n+\gamma)}dx\lesssim 1.$

{\it Case 2: $\rho(x_{0})/4<r<\rho(x_{0})$.} For this case, the atom $a$ has no canceling condition. We have
$(S_{\alpha,\beta}^{L}(a)(x))^{2}\lesssim I_{3}(x)+I_{4}(x)+I_{5}(x)$, where
$$\left\{
\begin{aligned}
I_{3}(x)&:=\int^{r^{2\alpha}/4^{\alpha}}_{0}\int_{|x-y|<t^{1/2\alpha}}\Big|\int_{\mathbb R^{n}}t^{\beta}\partial_{t}^{\beta}e^{-tL^{\alpha}}(y,x')a(x')dx'\Big|^{2}\frac{dydt}{t^{n/2\alpha+1}};\\
I_{4}(x)&:=\int_{r^{2\alpha}/4^{\alpha}}^{|x-x_{0}|^{2\alpha}/4^{2\alpha}}\int_{|x-y|<t^{1/2\alpha}}\Big|\int_{\mathbb R^{n}}t^{\beta}\partial_{t}^{\beta}e^{-tL^{\alpha}}(y,x')a(x')dx'\Big|^{2}\frac{dydt}{t^{n/2\alpha+1}};\\
I_{5}(x)&:=\int^{\infty}_{|x-x_{0}|^{2\alpha}/4^{2\alpha}}\int_{|x-y|<t^{1/2\alpha}}\Big|\int_{\mathbb R^{n}}t^{\beta}\partial_{t}^{\beta}e^{-tL^{\alpha}}(y,x')a(x')dx'\Big|^{2}\frac{dydt}{t^{n/2\alpha+1}}.
\end{aligned}
\right.$$

Because $x\in (8B)^{c}$ and $x'\in B$, then $|x'-x_{0}|<|x-x_{0}|/8$. On the other hand, for $t\in (0, r^{2\alpha}/4^{\alpha})$, $|y-x|<t^{1/2\alpha}\leq r/2<|x-x_{0}|/8$. This means that $|y-x'|\geq c|x-x_{0}|$. We can get
\begin{eqnarray*}
|I_{3}(x)|&\lesssim&\int^{r^{2\alpha}/4^{\alpha}}_{0}\int_{|x-y|<t^{1/2\alpha}}\Big(\int_{B}\frac{t^{\beta}}{(t^{1/2\alpha}+|x-x_{0}|)^{n+2\alpha\beta}}\frac{dx'}{|B|^{1+\gamma/n}}\Big)^{2}\frac{dydt}{t^{n/2\alpha+1}}\\
&\lesssim&\frac{1}{|B|^{2\gamma/n}}\int^{r^{2\alpha}/4^{\alpha}}_{0}\int_{|x-y|<t^{1/2\alpha}}\frac{t^{2\beta}}{(t^{1/2\alpha}+|x-x_{0}|)^{2n+4\alpha\beta}}
\frac{dydt}{t^{n/2\alpha+1}}\\
&\lesssim&\frac{1}{|B|^{2\gamma/n}}\int^{r^{2\alpha}/4^{\alpha}}_{0}\frac{t^{2\beta}}{(t^{1/2\alpha}+|x-x_{0}|)^{2n+4\alpha\beta}}
\frac{dt}{t}\\
&\lesssim&\frac{r^{4\alpha\beta-2\gamma}}{|x-x_{0}|^{2n+4\alpha\beta}},
\end{eqnarray*}
which indicates that
\begin{eqnarray*}
\int_{(8B)^{c}}|I_{3}(x)|^{n/2(n+\gamma)}dx&\lesssim&\int_{(8B)^{c}}\Big(\frac{r^{2\alpha\beta-\gamma}}{|x-x_{0}|^{n+2\alpha\beta}}\Big)^{n/(n+\gamma)}dx\leq C.
\end{eqnarray*}

Similarly,
\begin{eqnarray*}
|I_{4}(x)|&\lesssim&\int_{r^{2\alpha}/4^{\alpha}}^{|x-x_{0}|^{2\alpha}/4^{2\alpha}}\int_{|x-y|<t^{1/2\alpha}}\Big\{\int_{B}\frac{t^{\beta}}{(t^{1/2\alpha}+|y-x'|)^{n+2\alpha\beta}}
\frac{1}{|B|^{1+\gamma/n}}\Big(\frac{\rho(x')}{t^{1/2\alpha}}\Big)^{N}dx'
\Big\}^{2}\frac{dydt}{t^{n/2\alpha+1}}\\
&\lesssim&\int_{r^{2\alpha}/4^{\alpha}}^{|x-x_{0}|^{2\alpha}/4^{2\alpha}}\int_{|x-y|<t^{1/2\alpha}}\Big\{\int_{B}\frac{t^{\beta}}{(t^{1/2\alpha}+|y-x'|)^{n+2\alpha\beta}}
\frac{1}{|B|^{1+\gamma/n}}\Big(\frac{\rho(x_0)}{t^{1/2\alpha}}\Big)^{N}dx'
\Big\}^{2}\frac{dydt}{t^{n/2\alpha+1}}.
\end{eqnarray*}
Notice that $r/2\leq t^{1/2\alpha}\lesssim|x-x_{0}|/8$ for $t\in (r^{2\alpha}/4^{\alpha}, |x-x_{0}|^{2\alpha}/4^{2\alpha})$. It can be deduced from the triangle inequality that
$|y-x'|\sim |x-x_{0}|$. Then
\begin{eqnarray*}
|I_{4}(x)|&\lesssim&\int_{r^{2\alpha}/4^{\alpha}}^{|x-x_{0}|^{2\alpha}/4^{2\alpha}}\int_{|x-y|<t^{1/2\alpha}}\frac{1}{t^{n/\alpha}(1+|x-x_{0}|/t^{1/2\alpha})^{2n+4\alpha\beta}}
\frac{1}{|B|^{2\gamma/n}}\Big(\frac{\rho(x_0)}{t^{1/2\alpha}}\Big)^{2N}\frac{dydt}{t^{n/2\alpha+1}}\\
&\lesssim&\int_{r^{2\alpha}/4^{\alpha}}^{|x-x_{0}|^{2\alpha}/4^{2\alpha}}\Big(\frac{t^{n/2\alpha+\beta}}{t^{n/2\alpha}|x-x_{0}|^{n+2\alpha\beta}}
\frac{r^{N}}{t^{N/2\alpha}r^{\gamma}}\Big)^{2}\frac{dt}{t}\\
&\lesssim&\frac{r^{4\alpha\beta-2\gamma}}{|x-x_{0}|^{2n+4\alpha\beta}}.
\end{eqnarray*}

The estimate for $I_{5}$ is similar to that of $I_{4}$. In fact, due to $r\sim \rho(x_{0})$,
\begin{eqnarray*}
|I_{5}(x)|&\lesssim&\int^{\infty}_{|x-x_{0}|^{2\alpha}/4^{2\alpha}}\int_{|x-y|<t^{1/2\alpha}}\Bigg\{\int_{B}\frac{t^{\beta}}{(t^{1/2\alpha}+|y-x'|)^{n+2\alpha\beta}}
\frac{dx'}{r^{\gamma+n}}\Big(\frac{\rho(x_0)}{t^{1/2\alpha}}\Big)^{N}
\Bigg\}^{2}\frac{dydt}{t^{n/2\alpha+1}}\\
&\lesssim&\int^{\infty}_{|x-x_{0}|^{2\alpha}/4^{2\alpha}}\int_{|x-y|<t^{1/2\alpha}}\Big(\frac{r}{t^{1/2\alpha}}\Big)^{2N}\frac{1}{r^{2\gamma}}\frac{dydt}{t^{3n/2\alpha+1}}\\
&\lesssim&\frac{r^{2N-2\gamma}}{|x-x_{0}|^{2\alpha(n/\alpha+N)}}.
\end{eqnarray*}
The estimates for $I_{4}$ and $ I_{5}$ indicate that
$$\int_{(8B)^{c}}|I_{4}(x)+I_{5}(x)|^{n/2(n+\gamma)}dx\lesssim 1.$$

\end{proof}

\begin{lemma}\label{le-4.6-add}
Let $\alpha\in(0,1)$, $q_{t}(\cdot,\cdot)$ be a function of $x, y\in\mathbb R^{n}$ and $t>0$. Assume that for each $N>0$, there exists a constant $C_{N}$
such that for $\theta>\gamma$,
\begin{equation*}\label{eq-4.10}
|q_{t}(x,y)|\leq C_{N}\Big(1+\frac{t^{1/2\alpha}}{\rho(x)}+\frac{t^{1/2\alpha}}{\rho(x)}\Big)^{-N}t^{-n/2\alpha}\Big(1+
\frac{|x-y|}{t^{1/2\alpha}}\Big)^{-(n+\theta)}.
\end{equation*}
Then for any $H_{L}^{n/(n+\gamma)}$-atom $a$ supported on $B(x_{0}, r)$, there exists a constant $C_{x_{0}, r}$
such that
$$\sup_{t>0}\Big|\int_{\mathbb R^{n}}q_{t}(x,y)a(y)dy\Big|\leq C_{N,x_{0},r}(1+|x|)^{-n-\theta}, x\in\mathbb R^{n}.$$
\end{lemma}
\begin{proof}
If $x\in B(x_{0}, 2r)$, then  $1+|x|\leq1+|x-x_{0}|+|x_{0}|\leq 1+2r+|x_{0}|$. It follows from the condition $\|a\|_{\infty}\leq |B(x_{0}, r)|^{-1-\gamma/n}$ that
\begin{eqnarray*}
\Big|\int_{\mathbb R^{n}}q_{t}(x,y)a(y)dy\Big|&\lesssim&\int_{B(x_{0}, r)}|q_{t}(x,y)||a(y)|dy\\
&\lesssim&\int_{B(x_{0}, r)}t^{-n/2\alpha}\Big(1+\frac{|x-y|}{t^{1/2\alpha}}\Big)^{-n-\theta}r^{-n-\gamma}dy\\
&\lesssim&r^{-n-\gamma}\frac{(1+2r+|x_{0}|)^{n+\theta}}{(1+2r+|x_{0}|)^{n+\theta}}\\
&\lesssim&C_{N,x_{0},r}(1+|x|)^{-n-\theta}.
\end{eqnarray*}

If $x\notin B(x_{0}, 2r)$, then for any $y\in B(x_{0}, r)$, $|x-y|\sim |x-x_{0}|$. On the other hand,  $\rho(y)\sim\rho(x_{0})$ since $r<\rho(x_{0})$ and $|y-x_{0}|<r$. By Proposition \ref{prop-3.6}, we have
\begin{eqnarray*}
|q_{t}(x,y)|&\lesssim&t^{-n/2\alpha}\Big(1+\frac{|x-y|}{t^{1/2\alpha}}\Big)^{-(n+\theta)}\Big(1+\frac{t^{1/2\alpha}}{\rho(x_{0})}\Big)^{-N}\\
&\lesssim&t^{-n/2\alpha}\Big(1+\frac{|x-y|}{t^{1/2\alpha}}\Big)^{-(n+\theta)}\Big(\frac{t^{1/2\alpha}}{\rho(x_{0})}\Big)^{-\theta},
\end{eqnarray*}
which implies that
\begin{eqnarray*}
\Big|\int_{\mathbb R^{n}}q_{t}(x,y)a(y)dy\Big|&\lesssim&\Big(\frac{t^{1/2\alpha}}{\rho(x_{0})}\Big)^{-\theta}t^{-n/2\alpha}
\Big(\frac{|x-x_{0}|}{t^{1/2\alpha}}\Big)^{-(n+\theta)}\int_{\mathbb R^{n}}|a(y)|dy\\
&\lesssim&(\rho(x_{0}))^{\theta}r^{-\gamma}|x-x_{0}|^{-(n+\theta)}:=C_{\gamma, x_{0}, r}|x-x_{0}|^{-(n+2\alpha\beta)}.
\end{eqnarray*}
Because $x\notin B(x_{0}, 2r)$, set $x=x_{0}+2rz$, where $|z|\geq 1$. Then $1+|x|\leq 1+|x_{0}|+2r|z|$ and
$$\frac{1+|x_{0}|+2r}{2r}|x-x_{0}|=(1+|x_{0}|+2r)|z|\geq 1+|x_{0}|+2r|z|,$$
which implies that $|x_{0}-x|\geq(1+|x|)/{C_{x_{0}, r}}$. This completes the proof of Lemma \ref{le-4.6-add}.
\end{proof}

\begin{lemma}\label{le-4.3}
Given $\alpha\in (0,1)$, $\beta>0$ and $0<\gamma\leq \min\{2\alpha,2\alpha\beta\}$. Let $f\in L^{1}(\mathbb R^{n},\ (1+|x|)^{-(n+\gamma+\epsilon)}dx)$ for any $\epsilon>0$ and let $a$ be an $H^{n/(n+\gamma)}_{L}$-atom. Then
for
$$\begin{cases}
F(x,t):=t^{\beta}\partial^{\beta}_{t}e^{-tL^{\alpha}}(f)(x);\\
G(x,t):=t^{\beta}\partial^{\beta}_{t}e^{-tL^{\alpha}}(a)(x),
\end{cases}$$
there exists a constant $C_{\alpha,\beta}$ such that
$$C_{\alpha,\beta}\int_{\mathbb R^{n}}f(x)\overline{a(x)}dx=\iint_{\mathbb R^{n+1}_{+}}F(x,t)\overline{G(x,t)}\frac{dt}{t}.$$
\end{lemma}

\begin{proof}
Assume that $a$ is an $H^{n/(n+\gamma)}_{L}$-atom associated  to a ball $B(x_{0}, r)$. By Lemma \ref{le-4.3-1} and Theorem \ref{th-4.1}, we get
\begin{eqnarray*}
I&=&\iint_{\mathbb R^{n+1}_{+}}F(x,t)\overline{G(x,t)}\frac{dt}{t}\\
&=&\lim_{\epsilon\rightarrow0}\int^{1/\epsilon}_{\epsilon}\int_{\mathbb R^{n}}
t^{\beta}\partial_{t}^{\beta}e^{-tL^{\alpha}}(f)(x)\overline{t^{\beta}\partial_{t}^{\beta}e^{-tL^{\alpha}}(a)(x)}\frac{dxdt}{t}\\
&=&\lim_{\epsilon\rightarrow0}\int^{1/\epsilon}_{\epsilon}\int_{\mathbb R^{n}}D^{L,\beta}_{\alpha,t}(f)(x)\overline{D^{L,\beta}_{\alpha,t}(a)(x)}\frac{dxdt}{t}.
\end{eqnarray*}
The inner integration satisfies
\begin{eqnarray*}
\Big|\int_{\mathbb R^{n}}D^{L,\beta}_{\alpha,t}(f)(x)\overline{D^{L,\beta}_{\alpha,t}(a)(x)}dx\Big|&\leq&\int_{\mathbb R^{n}}\Big|D^{L,\beta}_{\alpha,t}(f)(x)\Big|\cdot\Big|\overline{D^{L,\beta}_{\alpha,t}(a)(x)}\Big|dx\\
&\leq&\Bigg\{\sup_{t>0}\Big|\overline{D^{L,\beta}_{\alpha,t}(a)(x)}\Big|\Bigg\}\Bigg\{\int_{\mathbb R^{n}}\Big|D^{L,\beta}_{\alpha,t}(f)(x)\Big|dx\Bigg\}.
\end{eqnarray*}

By Proposition \ref{prop-3.6}, we can see that
\begin{eqnarray*}
|D^{L,\beta}_{\alpha,t}(x,y)|&\lesssim&\frac{t^{\beta}}{(t^{1/2\alpha}+|x-y|)^{n+2\alpha\beta}}
\Big(1+\frac{t^{1/2\alpha}}{\rho(x)}+\frac{t^{1/2\alpha}}{\rho(y)}\Big)^{-N}\\
&\lesssim&t^{-n/2\alpha}\frac{1}{(1+|x-y|/t^{1/2\alpha})^{n+2\alpha\beta}}
\Big(1+\frac{t^{1/2\alpha}}{\rho(x)}+\frac{t^{1/2\alpha}}{\rho(y)}\Big)^{-N}.
\end{eqnarray*}
If $x\in B(x_{0}, 2r)$, then  $1+|x|\leq1+|x-x_{0}|+|x_{0}|\leq 1+2r+|x_{0}|$. It follows from the condition $\|a\|_{\infty}\leq |B(x_{0}, r)|^{-1-\gamma/n}$ that
\begin{eqnarray*}
\Big|\int_{\mathbb R^{n}}D^{L,\beta}_{\alpha,t}(x,y)a(y)dy\Big|&\lesssim&\int_{B(x_{0}, r)}|D^{L,\beta}_{\alpha,t}(x,y)||a(y)|dy\\
&\lesssim&\int_{B(x_{0}, r)}t^{-n/2\alpha}\Big(1+\frac{|x-y|}{t^{1/2\alpha}}\Big)^{-n-2\alpha\beta}r^{-n-\gamma}dy\\
&\lesssim&r^{-n-\gamma}\frac{(1+2r+|x_{0}|)^{n+2\alpha\beta}}{(1+2r+|x_{0}|)^{n+2\alpha\beta}}\\
&\lesssim&C_{N,x_{0},\gamma}(1+|x|)^{-n-2\alpha\beta}.
\end{eqnarray*}

If $x\notin B(x_{0}, 2r)$, then for any $y\in B(x_{0}, r)$, $|x-y|\sim |x-x_{0}|$. On the other hand,  $\rho(y)\sim\rho(x_{0})$ since $r<\rho(x_{0})$ and $|y-x_{0}|<r$. By Proposition \ref{prop-3.6}, we have
\begin{eqnarray*}
|D^{L,\beta}_{\alpha,t}(x,y)|&\lesssim&\frac{t^{\beta}}{(t^{1/2\alpha}+|x-y|)^{n+2\alpha\beta}}
\Big(1+\frac{t^{1/2\alpha}}{\rho(x)}+\frac{t^{1/2\alpha}}{\rho(y)}\Big)^{-N}\\
&\lesssim&t^{-n/2\alpha}\Big(1+\frac{|x-y|}{t^{1/2\alpha}}\Big)^{-(n+2\alpha\beta)}\Big(1+\frac{t^{1/2\alpha}}{\rho(x_{0})}\Big)^{-N}\\
&\lesssim&t^{-n/2\alpha}\Big(1+\frac{|x-y|}{t^{1/2\alpha}}\Big)^{-(n+2\alpha\beta)}\Big(\frac{t^{1/2\alpha}}{\rho(x_{0})}\Big)^{-2\alpha\beta},
\end{eqnarray*}
which implies that
\begin{eqnarray*}
\Big|\int_{\mathbb R^{n}}D^{L,\beta}_{\alpha,t}(x,y)a(y)dy\Big|&\lesssim&\Big(\frac{t^{1/2\alpha}}{\rho(x_{0})}\Big)^{-2\alpha\beta}t^{-n/2\alpha}
\Big(\frac{|x-x_{0}|}{t^{1/2\alpha}}\Big)^{-(n+2\alpha\beta)}\int_{\mathbb R^{n}}|a(y)|dy\\
&\lesssim&(\rho(x_{0}))^{2\alpha\beta}r^{-\gamma}|x-x_{0}|^{-(n+2\alpha\beta)}:=C_{\gamma, x_{0}, r}|x-x_{0}|^{-(n+2\alpha\beta)}.
\end{eqnarray*}
Because $x\notin B(x_{0}, 2r)$, set $x=x_{0}+2rz$, where $|z|\geq 1$. Then $1+|x|\leq 1+|x_{0}|+2r|z|$ and
$$\frac{1+|x_{0}|+2r}{2r}|x-x_{0}|=(1+|x_{0}|+2r)|z|\geq 1+|x_{0}|+2r|z|,$$
which implies that $|x_{0}-x|\geq(1+|x|)/{C_{x_{0}, r}}$. The above estimate indicate that $D^{L.\beta}_{\alpha,t}(\cdot,\cdot)$ satisfies (\ref{eq-4.10}) with $\theta=2\alpha\beta$. On the other hand,
\begin{eqnarray*}
\int_{\mathbb R^{n}}|D^{L,\beta}_{\alpha,t}(f)(x)||D^{L,\beta}_{\alpha,t}(a)(x)|dx
&\lesssim&\int_{\mathbb R^{n}}\Big(\int_{\mathbb R^{n}}\frac{|f(y)|t^{\beta}}{(|x-y|+t^{1/2\alpha})^{n+2\alpha\beta}}dy\Big)\frac{dx}{(1+|x|)^{n+2\alpha\beta}}
\lesssim I_{1}+I_{2},
\end{eqnarray*}
where
$$\left\{\begin{aligned}
I_{1}:=\iint_{|x-y|>|y|/2}\frac{|f(y)|t^{\beta}}{(|x-y|+t^{1/2\alpha})^{n+2\alpha\beta}}\frac{dxdy}{(1+|x|)^{n+2\alpha\beta}};\\
I_{2}:=\iint_{|x-y|\leq|y|/2}\frac{|f(y)|t^{\beta}}{(|x-y|+t^{1/2\alpha})^{n+2\alpha\beta}}\frac{dxdy}{(1+|x|)^{n+2\alpha\beta}}.
\end{aligned}\right.$$

If $|x-y|<|y|/2$, then $|y|\leq |x-y|+|x|\leq |y|/2+|x|$, i.e., $|y|\leq 2|x|$.
\begin{eqnarray*}
I_{2}&\lesssim&\int_{\mathbb R^{n}}\Big(\int_{\mathbb R^{n}}\frac{t^{\beta}}{(|x-y|+t^{1/2\alpha})^{n+2\alpha\beta}}dx\Big)\frac{|f(y)|}{(1+|y|)^{n+2\alpha\beta}}dy<\infty.
\end{eqnarray*}
For $I_{1}$, we have $|x|\geq |y|/2$ since $|x-y|\leq |y|/2$.
\begin{eqnarray*}
I_{1}&\lesssim& \iint_{\mathbb R^{n}\times\mathbb R^{n}}\frac{|f(y)|t^{\beta}}{(|x-y|+t^{1/2\alpha})^{n+2\alpha\beta}}\frac{dxdy}{(1+|y|)^{n+2\alpha\beta}}\\
&\lesssim&C_{t}\int_{\mathbb R^{n}}\Big(\int_{\mathbb R^{n}}\frac{|f(y)|}{(|y|+1)^{n+2\alpha\beta}}dy\Big)\frac{dx}{(1+|x|)^{n+2\alpha\beta}}<\infty.
\end{eqnarray*}
Notice that
\begin{eqnarray*}
\int_{\mathbb R^{n}}D^{L,\beta}_{\alpha,t}(f)(x)\overline{D^{L,\beta}_{\alpha,t}(a)(x)}dx&=&\int_{\mathbb R^{n}}f(x)\overline{(D^{L,\beta}_{\alpha,t})^{2}(a)(x)}dx\\
&=&\int_{\mathbb R^{n}}f(x)\overline{t^{2\beta}\partial_{t}^{2\beta}e^{-2tL^{\alpha}}(a)(x)}dx,
\end{eqnarray*}
which, together with the Fubini theorem, indicates that
\begin{eqnarray}\label{eq-4.10}
I&=&\lim_{\epsilon\rightarrow 0}\int^{1/\epsilon}_{\epsilon}\Big\{\int_{\mathbb R^{n}}f(y)\overline{t^{2\beta}\partial_{t}^{2\beta}e^{-2tL^{\alpha}}(a)(y)}dy\Big\}\frac{dt}{t}\\
&=&\lim_{\epsilon\rightarrow 0}\int_{\mathbb R^{n}}f(y)\Big\{\int^{1/\epsilon}_{\epsilon}\overline{t^{2\beta}\partial_{t}^{2\beta}e^{-2tL^{\alpha}}(a)(y)}
\frac{dt}{t}\Big\}dy.\nonumber
\end{eqnarray}
For the term
$$\int^{1/\epsilon}_{\epsilon}\overline{t^{2\beta}\partial_{t}^{2\beta}e^{-2tL^{\alpha}}(a)(y)}
\frac{dt}{t},$$
we can see that
\begin{eqnarray*}
\Big|\int^{1/\epsilon}_{\epsilon}{t^{2\beta}\partial_{t}^{2\beta}e^{-2tL^{\alpha}}(a)(y)}
\frac{dt}{t}\Big|&\leq&\Big|\int^{\infty}_{\epsilon}{t^{2\beta}\partial_{t}^{2\beta}e^{-2tL^{\alpha}}(a)(y)}
\frac{dt}{t}\Big|+\Big|\int_{1/\epsilon}^{\infty}{t^{2\beta}\partial_{t}^{2\beta}e^{-2tL^{\alpha}}(a)(y)}
\frac{dt}{t}\Big|\\
&=&\Big|\int_{\mathbb R^{n}}\int^{\infty}_{\epsilon}t^{2\beta}\partial^{2\beta}_{t}e^{-2tL^{\alpha}}(x,y)\frac{dt}{t}a(x)dx\Big|\\
&&+\Big|\int_{\mathbb R^{n}}\int^{\infty}_{1/\epsilon}t^{2\beta}\partial^{2\beta}_{t}e^{-2tL^{\alpha}}(x,y)\frac{dt}{t}a(x)dx\Big|.
\end{eqnarray*}
By the change of variables, we obtain
\begin{eqnarray*}
\Big|\int^{\infty}_{\epsilon}t^{2\beta}\partial^{2\beta}_{t}e^{-2tL^{\alpha}}(x,y)\frac{dt}{t}\Big|
&=&\Big|C_{\beta}\int^{\infty}_{\epsilon}t^{2\beta}\int^{\infty}_{0}\partial_{t}^{m}e^{-(2t+s)L^{\alpha}}(x,y)\frac{ds}{s^{1+2\beta-m}}\frac{dt}{t}\Big|\\
&=&\Big|C_{\beta}\int^{\infty}_{\epsilon}t^{2\beta}\int^{\infty}_{0}L^{\alpha m}e^{-(2t+s)L^{\alpha}}(x,y)\frac{ds}{s^{1+2\beta-m}}\frac{dt}{t}\Big|\\
&\simeq&\Big|C_{\beta}\int^{\infty}_{2\epsilon}t^{2\beta}\int^{\infty}_{0}\partial_{t}^{m}e^{-(t+s)L^{\alpha}}(x,y)\frac{ds}{s^{1+2\beta-m}}\frac{dt}{t}\Big|\\
&\simeq&\Big|\int^{\infty}_{2\epsilon}t^{2\beta}\partial^{2\beta}_{t}e^{-tL^{\alpha}}(x,y)\frac{dt}{t}\Big|.
\end{eqnarray*}
The rest of the proof is divided into three cases.

{\it Case 1: $2\beta<1$.} For this case, $[2\beta]+1=1$. Then a change of variable reaches
\begin{eqnarray*}
\int^{\infty}_{2\epsilon}t^{2\beta}\partial^{2\beta}_{t}e^{-tL^{\alpha}}(x,y)\frac{dt}{t}
&=&\int^{\infty}_{2\epsilon}t^{2\beta}\int^{\infty}_{0}\partial_{t}e^{-(t+s)L^{\alpha}}(x,y)\frac{ds}{s^{2\beta}}\frac{dt}{t}\\
&=&\int^{\infty}_{2\epsilon}t^{2\beta}\int^{\infty}_{t}\partial_{u}e^{-uL^{\alpha}}(x,y)\frac{du}{(u-t)^{2\beta}}\frac{dt}{t}\\
&=&\int^{\infty}_{2\epsilon}\partial_{u}e^{-uL^{\alpha}}(x,y)\Big(\int^{u}_{2\epsilon}\Big(\frac{t}{u-t}\Big)^{2\beta}\frac{dt}{t}\Big)du\\
&=&\int^{\infty}_{2\epsilon}\partial_{u}e^{-uL^{\alpha}}(x,y)\Big(\int^{1}_{2\epsilon/u}\Big(\frac{w}{1-w}\Big)^{2\beta}\frac{dw}{w}\Big)du.
\end{eqnarray*}
Notice that
$$e^{-L^{\alpha}}(x,y)\int^{1}_{2\epsilon/2\epsilon}\Big(\frac{w}{1-w}\Big)^{2\beta}\frac{dw}{w}=0,$$
and as $u\rightarrow\infty$,
$$|e^{-uL^{\alpha}}(x,y)|\leq\frac{u}{(u^{1/2\alpha}+|x-y|)^{n+2\alpha}}\leq \frac{1}{u^{n/2\alpha}}\rightarrow 0.$$
An application of integration by parts gives
\begin{eqnarray*}
\int^{\infty}_{2\epsilon}t^{2\beta}\partial^{2\beta}_{t}e^{-tL^{\alpha}}(x,y)\frac{dt}{t}&=&-\int^{\infty}_{2\epsilon}e^{-uL^{\alpha}}(x,y)\frac{\partial}{\partial u}\Big(\int^{1}_{2\epsilon/u}\Big(\frac{w}{1-w}\Big)^{2\beta}\frac{dw}{w}\Big)du\\
&=&\int^{\infty}_{2\epsilon}e^{-uL^{\alpha}}(x,y)\Big(\frac{2\epsilon}{u-2\epsilon}\Big)^{2\beta}\frac{du}{u}\\
&=&I+II,
\end{eqnarray*}
where
$$\left\{\begin{aligned}
I&:=\int^{\infty}_{2\epsilon}e^{-uL^{\alpha}}(x,y)\Big(\frac{2\epsilon}{u-2\epsilon}\Big)^{2\beta}\chi_{A}(u)\frac{du}{u};\\
II&:=\int^{\infty}_{2\epsilon}e^{-uL^{\alpha}}(x,y)\Big(\frac{2\epsilon}{u-2\epsilon}\Big)^{2\beta}\chi_{A^{c}}(u)\frac{du}{u},
\end{aligned}\right.$$
where $A:=\{u:\ u-2\epsilon\leq \epsilon+|x-y|^{2\alpha}\}$. By Proposition \ref{prop-3.1},
\begin{eqnarray*}
|I|&\lesssim&\int^{\infty}_{2\epsilon}\frac{u}{(u^{1/2\alpha}+|x-y|)^{n+2\alpha}}\Big(1+\frac{u^{1/2\alpha}}{\rho(x)}+\frac{u^{1/2\alpha}}{\rho(y)}\Big)^{-N}
\Big(\frac{2\epsilon}{u-2\epsilon}\Big)^{2\beta}\chi_{A}(u)\frac{du}{u}\\
&\lesssim&\frac{\epsilon^{2\beta}}{(\epsilon^{1/2\alpha}+|x-y|)^{n+2\alpha}}
\Big(1+\frac{\epsilon^{1/2\alpha}}{\rho(x)}+\frac{\epsilon^{1/2\alpha}}{\rho(y)}\Big)^{-N}\int^{3\epsilon+|x-y|^{2\alpha}}_{2\epsilon}(u-2\epsilon)^{-2\beta}du\\
&\lesssim&\frac{1}{\epsilon^{n/2\alpha}}\frac{1}{(1+|x-y|/\epsilon^{1/2\alpha})^{n+4\alpha\beta}}
\Big(1+\frac{\epsilon^{1/2\alpha}}{\rho(x)}+\frac{\epsilon^{1/2\alpha}}{\rho(y)}\Big)^{-N}.
\end{eqnarray*}

For $II$,
\begin{eqnarray*}
|II|&\lesssim&\int^{\infty}_{2\epsilon}\frac{u}{(u^{1/2\alpha}+|x-y|)^{n+2\alpha}}
\Big(1+\frac{\epsilon^{1/2\alpha}}{\rho(x)}+\frac{\epsilon^{1/2\alpha}}{\rho(y)}\Big)^{-N}\Big(\frac{2\epsilon}{u-2\epsilon}\Big)^{2\beta}
\chi_{A^{c}}(u)\frac{du}{u}\\
&\lesssim&\int^{\infty}_{3\epsilon+|x-y|^{2\alpha}}\frac{u}{(u^{1/2\alpha}+|x-y|)^{n+2\alpha}}
\Big(1+\frac{\epsilon^{1/2\alpha}}{\rho(x)}+\frac{\epsilon^{1/2\alpha}}{\rho(y)}\Big)^{-N}\Big(\frac{2\epsilon}{\epsilon+|x-y|^{2\alpha}}\Big)^{2\beta}
\frac{du}{u}\\
&\lesssim&\Big(\frac{2\epsilon}{\epsilon+|x-y|^{2\alpha}}\Big)^{2\beta}
\Big(1+\frac{\epsilon^{1/2\alpha}}{\rho(x)}+\frac{\epsilon^{1/2\alpha}}{\rho(y)}\Big)^{-N}\int^{\infty}_{3\epsilon+|x-y|^{2\alpha}}(u^{1/2\alpha}+|x-y|)^{-n-2\alpha}
du\\
&\lesssim&\frac{1}{\epsilon^{n/2\alpha}}\frac{1}{(1+|x-y|/\epsilon^{1/2\alpha})^{n+4\alpha\beta}}
\Big(1+\frac{\epsilon^{1/2\alpha}}{\rho(x)}+\frac{\epsilon^{1/2\alpha}}{\rho(y)}\Big)^{-N}.
\end{eqnarray*}

{\it Case 2: $2\beta=1$.} A direct computation gives
\begin{eqnarray*}
\Big|\int^{\infty}_{2\epsilon}t\partial_{t}e^{-tL^{\alpha}}(x,y)\frac{dt}{t}\Big|
&\lesssim&|e^{-\epsilon L^{\alpha}}(x,y)|\\
&\lesssim&\frac{1}{\epsilon^{n/2\alpha}}\frac{1}{(1+ {|x-y|}/{\epsilon^{1/2\alpha}})^{n+2\alpha}}\Big(1+\frac{\epsilon^{1/2\alpha}}{\rho(x)}+\frac{\epsilon^{1/2\alpha}}{\rho(y)}\Big)^{-N}.
\end{eqnarray*}

{\it Case 3: $2\beta>1$.} Let $k\in\mathbb Z_{+}$ such that $k-1<2\beta\leq k$, $k\ge 2$.
We obtain
\begin{eqnarray*}
M&=&\int^{\infty}_{2\epsilon}t^{2\beta}\partial^{2\beta}_{t}e^{-tL^{\alpha}}(x,y)\frac{dt}{t}\\
&=&\int^{\infty}_{2\epsilon}t^{2\beta}\Big(\int^{\infty}_{0}\partial^{k}_{t}e^{-(t+s)L^{\alpha}}(x,y)\frac{ds}{s^{1+2\beta-k}}\Big)\frac{dt}{t}\\
&=&\int^{\infty}_{2\epsilon}t^{2\beta}L^{\alpha k}\Big(\int^{\infty}_{0}e^{-(t+s)L^{\alpha}}(x,y)\frac{ds}{s^{1+2\beta-k}}\Big)\frac{dt}{t}\\
&=&\int^{\infty}_{2\epsilon}t^{2\beta}\Big(\int^{\infty}_{t}L^{\alpha k}e^{-uL^{\alpha}}(x,y)\frac{du}{(u-t)^{1+2\beta-k}}\Big)\frac{dt}{t}\\
&=&\int^{\infty}_{2\epsilon}t^{2\beta}\Big(\int^{\infty}_{t}\partial_{u}^{k}e^{-uL^{\alpha}}(x,y)\frac{du}{(u-t)^{1+2\beta-k}}\Big)\frac{dt}{t}\\
&=&\int^{\infty}_{2\epsilon}\partial_{u}^{k}e^{-uL^{\alpha}}(x,y)\Big(\int^{u}_{2\epsilon}t^{2\beta}(u-t)^{k-2\beta-1}\frac{dt}{t}\Big)du\\
&=&\int^{\infty}_{2\epsilon}u^{k-1}\partial_{u}^{k}e^{-uL^{\alpha}}(x,y)\Big(\int^{1}_{2\epsilon/u}w^{2\beta}(1-w)^{k-2\beta-1}\frac{dw}{w}\Big)du,
\end{eqnarray*}
where in the last step we have used the change of variables: $w=t/u$. Notice that
\begin{eqnarray*}
u^{k-1}\partial_{u}^{k}e^{-uL^{\alpha}}(x,y)
&=&\partial_{u}(u^{k-1}\partial_{u}^{k-1}e^{-uL^{\alpha}}(x,y))-(k-1)\partial_{u}(u^{k-2}\partial_{u}^{k-2}e^{-uL^{\alpha}}(x,y))\\
&&+\cdots+(-1)^{k-1}(k-1)!e^{-uL^{\alpha}}(x,y).
\end{eqnarray*}
Then the integration by parts yields $M= \sum^{k-1}_{m=1}C_{m}I_{m}$,  where $C_{m}=(-1)^{m-1}(k-1)!/(k-m+1)!$ and
$$I_{m}:=\int^{\infty}_{2\epsilon}u^{m}\partial_{u}^{m}e^{-uL^{\alpha}}(x,y)\frac{(2\epsilon)^{2\beta}u^{1-k}}{(u-2\epsilon)^{1+2\beta-k}}\frac{du}{u}.$$
We obtain
\begin{eqnarray*}
|I_{m}|\lesssim\int^{\infty}_{2\epsilon}\frac{u^{m}}{(u^{1/2\alpha}+|x-y|)^{n+2\alpha m}}\Big(1+\frac{u^{1/2\alpha}}{\rho(x)}+\frac{u^{1/2\alpha}}{\rho(y)}\Big)^{-N}\frac{(2\epsilon)^{2\beta}u^{1-k}}{(u-2\epsilon)^{1+2\beta-k}}\frac{du}{u}
\lesssim I_{m}^{(1)}+I^{(2)}_{m},
\end{eqnarray*}
where
$$\left\{\begin{aligned}
I_{m}^{(1)}&:=&\frac{\epsilon^{2\beta}}{(\epsilon^{1/2\alpha}+|x-y|)^{n+2\alpha m}}\Big(1+\frac{\epsilon^{1/2\alpha}}{\rho(x)}+\frac{\epsilon^{1/2\alpha}}{\rho(y)}\Big)^{-N}
\int^{3\epsilon}_{2\epsilon}\frac{1}{(u-2\epsilon)^{1+2\beta-k}}\frac{du}{u^{k-m}};\\
I_{m}^{(2)}&:=&\frac{\epsilon^{2\beta}}{(\epsilon^{1/2\alpha}+|x-y|)^{n+2\alpha m}}\Big(1+\frac{\epsilon^{1/2\alpha}}{\rho(x)}+\frac{\epsilon^{1/2\alpha}}{\rho(y)}\Big)^{-N}
\int^{\infty}_{3\epsilon}\frac{1}{(u-2\epsilon)^{1+2\beta-k}}\frac{du}{u^{k-m}}.
\end{aligned}\right.$$

For $I^{(1)}_{m}$, since $2\epsilon<u<3\epsilon$, we get
\begin{eqnarray*}
I_{m}^{(1)}&\lesssim&\frac{\epsilon^{2\beta}}{(\epsilon^{1/2\alpha}+|x-y|)^{n+2\alpha m}}\Big(1+\frac{\epsilon^{1/2\alpha}}{\rho(x)}+\frac{\epsilon^{1/2\alpha}}{\rho(y)}\Big)^{-N}\frac{1}{\epsilon^{k-m}}
\int^{3\epsilon}_{2\epsilon}\frac{du}{(u-2\epsilon)^{1+2\beta-k}}\\
&\lesssim&\frac{\epsilon^{2\beta}}{(\epsilon^{1/2\alpha}+|x-y|)^{n+2\alpha m}}\Big(1+\frac{\epsilon^{1/2\alpha}}{\rho(x)}+\frac{\epsilon^{1/2\alpha}}{\rho(y)}\Big)^{-N}\frac{1}{\epsilon^{k-m}}
\epsilon^{k-2\beta}\\
&\lesssim&\frac{1}{\epsilon^{n/2\alpha}}\frac{1}{(1+|x-y|/\epsilon^{1/2\alpha})^{n+2\alpha m}}\Big(1+\frac{\epsilon^{1/2\alpha}}{\rho(x)}+\frac{\epsilon^{1/2\alpha}}{\rho(y)}\Big)^{-N}.
\end{eqnarray*}
Similarly, for $I^{(2)}_{m}$, because $u\in (3\epsilon, \infty)$, then $1/u\lesssim 1/(u-2\epsilon)$. Noticing that $m<2\beta$, we obtain
\begin{eqnarray*}
I^{(2)}_{m}&\lesssim&\frac{\epsilon^{2\beta}}{(\epsilon^{1/2\alpha}+|x-y|)^{n+2\alpha m}}\Big(1+\frac{\epsilon^{1/2\alpha}}{\rho(x)}+\frac{\epsilon^{1/2\alpha}}{\rho(y)}\Big)^{-N}
\int_{3\epsilon}^{\infty}\frac{du}{(u-2\epsilon)^{1+2\beta-m}}\\
&\lesssim&\frac{\epsilon^{2\beta}}{(\epsilon^{1/2\alpha}+|x-y|)^{n+2\alpha m}}\Big(1+\frac{\epsilon^{1/2\alpha}}{\rho(x)}+\frac{\epsilon^{1/2\alpha}}{\rho(y)}\Big)^{-N}\epsilon^{m-2\beta}\\
&\lesssim&\frac{1}{\epsilon^{n/2\alpha}}\frac{1}{(1+|x-y|/\epsilon^{1/2\alpha})^{n+2\alpha m}}\Big(1+\frac{\epsilon^{1/2\alpha}}{\rho(x)}+\frac{\epsilon^{1/2\alpha}}{\rho(y)}\Big)^{-N}.
\end{eqnarray*}
By Lemma \ref{le-4.6-add}, the above estimates in Cases 1-3 indciate that
$$\sup_{\epsilon>0}\Big|\int^{1/\epsilon}_{\epsilon}{t^{2\beta}\partial_{t}^{2\beta}e^{-2tL^{\alpha}}(a)(y)}
\frac{dt}{t}\Big|\lesssim (1+|y|)^{-(n+\gamma+\epsilon)},$$
Therefore we can use Lemma \ref{le-4.7-add} completes the proof.
\end{proof}

Finally, we can obtain the following characterization of $BMO^{\gamma}_{L}(\mathbb R^{n})$ corresponding to the time-fractional derivative.

\begin{theorem}\label{th-4.2}
Let $V\in B_{q}, q>n$. Assume that $\alpha\in(0,1)$, $\beta>0$, $0<\gamma\leq 1$ with
$$0<\gamma<\min\{2\alpha, 2\beta\alpha\}.$$
 Let $f$ be a function such that
\begin{equation}\label{eq-4.5-1}
\int_{\mathbb R^{n}}\frac{|f(x)|}{(1+|x|)^{n+\gamma+\epsilon}}dx<\infty
\end{equation}
for some $\epsilon>0$. The following statements are equivalent:
\item{\rm (i)} $f\in BMO^{\gamma}_{L}(\mathbb R^{n})$;
\item{\rm (ii)} There exists $C_{\alpha,\beta}$ such that $\|D^{L,\beta}_{\alpha,t}(f)\|_{\infty}\leq C_{\alpha,\beta} t^{\gamma/2\alpha}$;
\item{\rm (iii)} For all $B=B(x_{B}, r_B)\subset \mathbb R^{n}$,
\begin{equation}\label{eq-4.7}
\Big(\frac{1}{|B|}\int^{r_B^{2\alpha}}_{0}\int_{B}|D^{L,\beta}_{\alpha,t}(f)(x)|^{2}\frac{dxdt}{t}\Big)^{1/2}\lesssim |B|^{\gamma/n}.
\end{equation}
\end{theorem}

\begin{proof}
(i)$\Longrightarrow$(ii). If $f\in BMO^{\gamma}_{L}(\mathbb R^{n})$, then $|t^{\beta}\partial^{\beta}_{t}e^{-tL^{\alpha}}f(x)|\lesssim I+II$, where
$$\left\{\begin{aligned}
&I:=\Big|\int_{\mathbb R^{n}}D^{L,\beta}_{\alpha,t}(x,y)(f(y)-f(x))dy\Big|;\\
&II:=\Big|f(x)\int_{\mathbb R^{n}}D^{L,\beta}_{\alpha,t}(x,y)dy\Big|.
\end{aligned}
\right.$$

For $I$, we have
\begin{eqnarray*}
I&\lesssim&\|f\|_{BMO^{\gamma}_{L}}\int_{\mathbb R^{n}}\frac{t^{\beta}|x-y|^{\gamma}}{(t^{1/2\alpha}+|x-y|)^{n+2\alpha\beta}}dy\lesssim t^{\gamma/2\alpha}\|f\|_{BMO^{\gamma}_{L}}.
\end{eqnarray*}

We further divide the estimation of $II$ into the following two cases.

{\it Case 1: $\rho(x)\leq t^{1/2\alpha}$.} By Proposition \ref{prop-3.6},
\begin{eqnarray*}
II&\lesssim&\|f\|_{BMO^{\gamma}_{L}}\rho(x)^{\gamma}\Big|\int_{\mathbb R^{n}}D^{L,\beta}_{\alpha,t}(x,y)dy\Big|\\
&\lesssim&\|f\|_{BMO^{\gamma}_{L}}t^{\gamma/2\alpha}\int_{\mathbb R^{n}}\frac{t^{\beta}}{(t^{1/2\alpha}+|x-y|)^{n+2\alpha\beta}}dy\\
&\lesssim&\|f\|_{BMO^{\gamma}_{L}}t^{\gamma/2\alpha}.
\end{eqnarray*}

{\it Case 2: $\rho(x)> t^{1/2\alpha}$.} We use Proposition \ref{prop-3.5} to obtain that, there exists $\delta'>\gamma$ such that
\begin{eqnarray*}
II&\lesssim&\|f\|_{BMO^{\gamma}_{L}}\rho(x)^{\gamma}\Big|\int_{\mathbb R^{n}}D^{L,\beta}_{\alpha,t}(x,y)dy\Big|\\
&\lesssim&\|f\|_{BMO^{\gamma}_{L}}\rho(x)^{\gamma}\frac{(t^{1/2\alpha}/\rho(x))^{\delta'}}
{(1+t^{1/2\alpha}/\rho(x))^{N}}\\
&\lesssim&\|f\|_{BMO^{\gamma}_{L}}t^{\gamma/2\alpha}\frac{(t^{1/2\alpha}/\rho(x))^{\delta'-\gamma}}
{(1+t^{1/2\alpha}/\rho(x))^{N}}\\
&\lesssim&\|f\|_{BMO^{\gamma}_{L}}t^{\gamma/2\alpha}.
\end{eqnarray*}

(ii)$\Longrightarrow$(iii). Assume that (ii) holds.
Then
\begin{eqnarray*}
\Big(\frac{1}{|B|}\int^{r_B^{2\alpha}}_{0}\int_{B}|D^{L,\beta}_{\alpha,t}(f)(x)|^{2}\frac{dxdt}{t}\Big)^{1/2}
\lesssim \|f\|_{BMO^{\gamma}_{L}}\Big(\frac{1}{|B|}\int^{r_B^{2\alpha}}_{0}\int_{B}t^{\gamma/\alpha}\frac{dxdt}{t}\Big)^{1/2}
\lesssim \|f\|_{BMO^{\gamma}_{L}}|B|^{\gamma/n}.
\end{eqnarray*}

(iii)$\Longrightarrow$(i). Assume that (\ref{eq-4.7}) holds. Let $a$ be an $H^{n/(n+\gamma)}_{L}$-atom associated with $B=B(x_{B}, r_B)$. Then by Lemma \ref{le-4.3},
$$\int_{\mathbb R^{n}}f(x)\overline{a(x)}dx\simeq\iint_{\mathbb R^{n+1}_{+}}t^{\beta}\partial^{\beta}_{t}e^{-tL^{\alpha}}(f)(x)\overline{t^{\beta}\partial^{\beta}_{t}e^{-tL^{\alpha}}(a)(x)}\frac{dt}{t},$$
which, together with (\ref{eq-4.8}) and Theorem \ref{th-4.1}, gives
\begin{eqnarray*}
\Big|\int_{\mathbb R^{n}}f(x)\overline{a(x)}dx\Big|&\lesssim&\sup_{B}\Big(\frac{1}{|B|^{1+2\gamma/n}}\int^{r_{B}^{2\alpha}}_{0}\int_{B}
|t^{\beta}\partial_{t}^{\beta}e^{-tL^{\alpha}}(f)(x)|^{2}\frac{dxdt}{t}\Big)^{1/2}\\
&&\quad\times \Bigg\{\int_{\mathbb R^{n}}\Big(\iint_{|x-y|<s^{1/2\alpha}}|t^{\beta}\partial^{\beta}e^{-tL^{\alpha}}(a)(y)|^{2}\frac{t^{1/2\alpha-1}dyds}
{t^{n/2\alpha+1}}\Big)^{n/2(n+\gamma)}dx\Bigg\}^{1+\gamma/n}\\
&\lesssim&\|S^{L}_{\alpha,\beta}(a)\|_{L^{n/(n+\gamma)}}\sup_{B}\Big(\frac{1}{|B|^{1+2\gamma/n}}\int^{r_{B}^{2\alpha}}_{0}\int_{B}
|t^{\beta}\partial_{t}^{\beta}e^{-tL^{\alpha}}(f)(x)|^{2}\frac{dxdt}{t}\Big)^{1/2}\\
&\lesssim&\|a\|_{H^{n/(n+\gamma)}_{L}}.
\end{eqnarray*}
Hence
$$T(g):=\int_{\mathbb R^{n}}f(x)\overline{g(x)}dx, \ g\in H^{n/(n+\gamma)}_{L}(\mathbb R^{n}),$$
is a bounded  linear functional on $H^{n/(n+\gamma)}_{L}(\mathbb R^{n})$, equivalently, $f\in \big(H_L^{n/( n+\gamma)}(\mathbb R^n)\big)^*= BMO^{\gamma}_{L}(\mathbb R^{n})$.

\end{proof}



Below we consider the characterization of $BMO^{\gamma}_{L}(\mathbb R^{n})$ via the the spatial gradient. Define a general gradient as $\nabla_{\alpha}:=(\nabla_{x}, \partial_{t}^{1/2\alpha})$.



\begin{theorem}\label{th-4.3}
Let $V\in B_{q}, q>n$. Assume that $\alpha\in(0,1/2-n/2q)$, $\beta>0$ and $0<\gamma\leq 1$ with
$$0<\gamma<\min\{2\alpha, 2\alpha\beta\}.$$ Let $f$ be a function satisfying (\ref{eq-4.5-1}). The following statements are equivalent:
\item{{\rm (i)}} $f\in BMO^{\gamma}_{L}(\mathbb R^{n})$;
\item{{\rm (ii)}} There exists a constant $C>0$ such that
$$\|t^{1/2\alpha}\nabla_{\alpha}e^{-tL^{\alpha}}f\|_{\infty}\leq Ct^{\gamma/2\alpha};$$
\item{{\rm (iii)}} $u(x,t)=e^{-tL^{\alpha}}f(x)$ satisfies that, for any balls $B=B(x_B, r_
B),$
\begin{equation}\label{eq-4.9}
 {1\over {|B|}^{1+2\gamma/n}}\int^{r^{2\alpha}_{B}}_{0}\int_{B}|t^{1/2\alpha}\nabla_{\alpha}e^{-tL^\alpha}(f)(x)|^{2}
 \frac{dxdt}{t}\le  C.
\end{equation}
\end{theorem}

\begin{proof}
(i) $\Longrightarrow$ (ii). Let $f\in BMO^{\gamma}_{L}(\mathbb{R}^{n})$. By Theorem \ref{th-4.2},
$\|t^{1/2\alpha}\partial^{1/2\alpha}_{t}e^{-tL^{\alpha}}(f)\|_{\infty}\leq C_{\alpha,\beta} t^{\gamma/2\alpha}.$
One writes
\begin{eqnarray*}
t^{1/2\alpha}\nabla_{x}e^{-tL^{\alpha}}f(x)&=&\int_{\mathbb R^{n}}t^{1/2\alpha}\nabla_{x}K^{L}_{\alpha,t}\Big(f(z)-f(x)\Big)dz
+f(x)t^{1/2\alpha}\nabla_{x}e^{-tL^{\alpha}}(1)(x):=I(x)+II(x).
\end{eqnarray*}

We first estimate the term $I$. Because $f\in BMO^{\gamma}_{L}(\mathbb R^{n})$, then
$|f(x)-f(z)|\leq \|f\|_{BMO^{\gamma}_{L}}|x-z|^{\gamma}.$ Since
$$|t^{1/2\alpha}\nabla_{x}K^{L}_{\alpha,t}(x,z)|\lesssim\frac{t}{(t^{1/2\alpha}+|x-z|)^{n+2\alpha}},$$
a direct computation gives
\begin{eqnarray*}
|I(x)|&\lesssim&\|f\|_{BMO^{\gamma}_{L}}\int_{\mathbb R^{n}}|t^{1/2\alpha}\nabla_{x}K^{L}_{\alpha,t}(x,z)|\cdot|x-z|^{\gamma}dz\\
&\lesssim&\|f\|_{BMO^{\gamma}_{L}}\int_{\mathbb R^{n}}\frac{t|x-z|^{\gamma}}{(t^{1/2\alpha}+|x-z|)^{n+2\alpha}}dz\\
&\lesssim&t^{\gamma/2\alpha}\|f\|_{BMO^{\gamma}_{L}}.
\end{eqnarray*}
By Proposition \ref{prop-3.3-2}, we have  $$|t^{1/2\alpha}\nabla_{x}e^{-tL^{\alpha}}(1)|\lesssim \min\Big\{(t^{1/2\alpha}/\rho(x))^{1+2\alpha},\ (t^{1/2\alpha}/\rho(x))^{-N}\Big\}.$$
The estimate of $II$ is divided into two cases.

{\it Case 1: $\rho(x)\leq t^{1/2\alpha}$}.
$f\in BMO^{\gamma}_{L}(\mathbb R^{n})$ implies that
$|f(x)|\lesssim \rho^{\gamma}(x)\|f\|_{BMO^{\gamma}_{L}}$. Then
\begin{eqnarray*}
II(x)&\lesssim&\|f\|_{BMO^{\gamma}_{L}}\rho^{\gamma}(x)|t^{1/2\alpha}\nabla_{x}e^{-tL^{\alpha}}(1)(x)|\\
&\lesssim&\|f\|_{BMO^{\gamma}_{L}}\rho^{\gamma}(x)\Big(\frac{t^{1/2\alpha}}{\rho(x)}\Big)^{-N}\\
&\lesssim&\|f\|_{BMO^{\gamma}_{L}}t^{\gamma/2\alpha}\Big(\frac{t^{1/2\alpha}}{\rho(x)}\Big)^{-N-\gamma}\\
&\lesssim&\|f\|_{BMO^{\gamma}_{L}}t^{\gamma/2\alpha}\|f\|_{BMO^{\gamma}_{L}}.
\end{eqnarray*}

{\it Case 2: $\rho(x)>t$.} We can get
\begin{eqnarray*}
II(x)&\lesssim&\rho^{\gamma}(x)\|f\|_{BMO^{\gamma}_{L}}|t^{1/2\alpha}\nabla_{x}e^{-tL^{\alpha}}(1)(x)|\\
&\lesssim&\rho^{\gamma}(x)\|f\|_{BMO^{\gamma}_{L}}\Big(\frac{t^{1/2\alpha}}{\rho(x)}\Big)^{1+2\alpha}\\
&\lesssim&t^{\gamma/2\alpha}\|f\|_{BMO^{\gamma}_{L}}\Big(\frac{t^{1/2\alpha}}{\rho(x)}\Big)^{1+2\alpha-\gamma}\\
&\lesssim&t^{\gamma/2\alpha}\|f\|_{BMO^{\gamma}_{L}}.
\end{eqnarray*}

(ii)$\Longrightarrow$(iii). For every ball $B=B(x_{B}, r_{B})$,
\begin{eqnarray*}
\int^{r^{2\alpha}_{B}}_{0}\int_{B}\Big|t^{1/2\alpha}\nabla_{\alpha}e^{-tL^{\alpha}}f(x)\Big|^{2}\frac{dxdt}{t}
&\lesssim& \int^{r^{2\alpha}_{B}}_{0}\int_{B}t^{\gamma/\alpha}\frac{dxdt}{t}\lesssim r_B^{n+2\gamma},
\end{eqnarray*}
which implies that (\ref{eq-4.9}) holds.

(iii)$\Longrightarrow$(i). Assume that (\ref{eq-4.9}) holds. For any ball $B=B(x_{B}, r_{B})$, it holds
$$\sup_{B}r_{B}^{-(n+2\gamma)}\int^{r^{2\alpha}_{B}}_{0}\int_{B(x_{B}, r_{B})}|t^{1/2\alpha}\partial_{t}^{1/2\alpha}e^{-tL^{\alpha}}(f)(x)|^{2}
 \frac{dxdt}{t}<\infty.$$
 It is a corollary of Theorem \ref{th-4.2} that $f\in BMO^{\gamma}_{L}(\mathbb R^{n})$ with
 $$\|f\|_{BMO^{\gamma}_{L}}\lesssim \sup_{B}r_{B}^{-(n+2\gamma)}\int^{r^{2\alpha}_{B}}_{0}\int_{B(x_{B}, r_{B})}|t^{1/2\alpha}\partial_{t}^{1/2\alpha}e^{-tL^{\alpha}}(f)(x)|^{2}
 \frac{dxdt}{t}<\infty.$$

\end{proof}

A positive measure $\nu$ on $\mathbb R^{n+1}_{+}$ is called a $\kappa$-Carleson measure if
$$\|\nu\|_\mathcal{C}:=\sup_{x\in\mathbb R^{n}, r>0}\frac{\nu(B(x,r)\times (0,r))}{|B(x,r)|^{\kappa}}<\infty.$$

The following result can be deduced from Theorem \ref{th-4.3} immediately.

\begin{theorem}\label{th-4.4}
Let $V\in B_{q}, q>n$. Assume that $\alpha\in(0,1/2-n/2q)$, $\beta>0$ and $0<\gamma\leq 1$ with
$$0<\gamma<\min\{2\alpha, 2\alpha\beta\}.$$ Let $d\nu_{\alpha}$ be a measure defined by
$$d\nu_{\alpha}(x,t):=|t\nabla e^{-t^{2\alpha}L^{\alpha}}(f)(x)|^{2}dxdt/t,\ (x,t)\in\mathbb R^{n+1}_{+}.$$

\item{\rm (i)}\ If $f\in BMO^{\gamma}_{L}(\mathbb R^{n})$, then $d\nu_{\alpha}$ is a $(1+2\gamma/n)$-Carleson measure;
\item{\rm (ii)}\ Conversely, if $f\in L^{1}((1+|x|)^{-n-1}dx)$ and $d\nu_{\alpha}$ is a $(1+2\gamma/n)$-Carleson measure, then $f\in BMO^{\gamma}_{L}(\mathbb R^{n})$.

Moreover, in any case, there exists  a constant $C>0$ such that
$$C^{-1}\|f\|^{2}_{BMO^{\gamma}_{L}}\leq\|d\nu_{\alpha}\|_{\mathcal C}\leq C\|f\|^{2}_{BMO^{\gamma}_{L}}.$$
\end{theorem}

\begin{proof}
(i). In Theorem \ref{th-4.2}, letting $\beta=1$, we obtain for $f\in BMO^{\gamma}_{L}(\mathbb R^{n})$,
$$\frac{1}{|B|^{1+2\gamma/n}}\int^{r^{2\alpha}}_{0}\int_{B}|t\partial_{t}e^{-tL^{\alpha}}(f)(x)|^{2}\frac{dxdt}{t}\lesssim \|f\|^{2}_{BMO^{\gamma}_{L}},$$
which, together with a change of variable, gives
\begin{multline*}
\frac{1}{|B|^{1+2\gamma/n}}\int^{r}_{0}\int_{B}|t\partial_{t}e^{-t^{2\alpha}L^{\alpha}}(f)(x)|^{2}
\frac{dxdt}{t}=
\frac{1}{|B|^{1+2\gamma/n}}\int^{r}_{0}\int_{B}|t^{2\alpha}L^{\alpha}e^{-t^{2\alpha}L^{\alpha}}(f)(x)|^{2}
\frac{dxdt}{t}\\
=\frac{1}{|B|^{1+2\gamma/n}}\int^{r^{2\alpha}}_{0}\int_{B}|t\partial_t e^{-tL^{\alpha}}(f)(x)|^{2}
\frac{dxdt}{t}
\lesssim \|f\|^{2}_{BMO^{\gamma}_{L}}.
\end{multline*}
The estimation
$$\frac{1}{|B|^{1+2\gamma/n}}\int^{r_{B}}_{0}\int_{B}|t\nabla_{x}e^{-t^{2\alpha}L^{\alpha}}f(x)|^{2}\frac{dxdt}{t}\lesssim \|f\|^{2}_{BMO^{\gamma}_{L}}$$
can be obtained in the manner of Theorem \ref{th-4.2}.

(ii). Assume that $d\nu_{\alpha}$ is a $(1+2\gamma/n)$-Carleson measure, i.e., for any ball $B(x,r)$,
$$\sup_{B}\frac{1}{|B|^{1+2\gamma/n}}\int^{r_{B}}_{0}\int_{B}|t\nabla e^{-t^{2\alpha}L^{\alpha}}(f)(x)|^{2}\frac{dxdt}{t}<\infty.$$
Subsequently,
$$\sup_{B}\frac{1}{|B|^{1+2\gamma/n}}\int^{r_{B}}_{0}\int_{B}|t\partial_{t} e^{-t^{2\alpha}L^{\alpha}}(f)(x)|^{2}\frac{dxdt}{t}<\infty.$$
It can be deduced from Theorem \ref{th-4.3} that $f\in BMO^{\gamma}_{L}(\mathbb R^{n})$.
\end{proof}

\end{document}